\documentclass{amsart}
\usepackage{amsmath,amssymb, amsfonts,mathrsfs,enumerate,graphicx,subfig}
\usepackage{color}

\newtheorem{theorem}{Theorem}

\newtheorem{proposition}{Proposition}[section]
\newtheorem{corollary}[proposition]{Corollary}
\newtheorem{lemma}[proposition]{Lemma}

\newtheorem{remark}[proposition]{Remark}

\usepackage{soul}

\DeclareMathOperator{\dimaff}{dim_{\mathsf{aff}}}
\DeclareMathOperator{\dimlyap}{dim_{\mathsf{Lyap}}}
\DeclareMathOperator{\dimh}{dim_{\mathsf{H}}}

\DeclareMathOperator{\GL}{GL}
\DeclareMathOperator{\SO}{SO}
\DeclareMathOperator{\SL}{SL}
\DeclareMathOperator{\Ad}{Ad}

\DeclareMathOperator{\End}{End}
\DeclareMathOperator{\Mat}{Mat}
\DeclareMathOperator{\ess}{ess}

\newcommand{\Z}{\mathbb{Z}}

\newcommand{\C}{\mathbb{C}}
\newcommand{\R}{\mathbb{R}}
\newcommand{\N}{\mathbb{N}}

\usepackage{marginnote}

\DeclareMathOperator{\rank}{rank}
\DeclareMathOperator{\sspan}{span}
\newcommand{\trip}[1]{{\left\vert\kern-0.25ex\left\vert\kern-0.25ex\left\vert #1 
    \right\vert\kern-0.25ex\right\vert\kern-0.25ex\right\vert}}
\newcommand{\iii}{\mathtt{i}}
\newcommand{\jjj}{\mathtt{j}}
\newcommand{\kkk}{\mathtt{k}}
\usepackage{xcolor}

\allowdisplaybreaks

\title[Dimension gap for self-affine measures]{A converse statement to Hutchinson's theorem and a dimension gap for self-affine measures}
\author{Ian D. Morris and Cagri Sert} 
\address{I. D. Morris: School of Mathematical Sciences, Queen Mary, University of London, Mile End Road, London E1 4NS, UK}
\email{i.morris@qmul.ac.uk }
\address{C.
 Sert: Institut f\"{u}r Mathematik,
Universit\"{a}t Z\"{u}rich,
Winterthurerstrasse 190, 8057 Z\"{u}rich, Switzerland}
\email{cagri.sert@math.uzh.ch}
\begin{document}

\begin{abstract}A well-known theorem of J.E. Hutchinson states that if an iterated function system consists of similarity transformations and satisfies the open set condition then its attractor supports a self-similar measure with Hausdorff dimension equal to the similarity dimension. In this article we prove the following result which may be regarded as a form of partial converse: if an iterated function system consists of invertible affine transformations whose linear parts do not preserve a common invariant subspace, and its attractor supports a self-affine measure with Hausdorff dimension equal to the affinity dimension, then the system necessarily consists of similarity transformations. We obtain this result by showing that the equilibrium measures of an affine iterated function system are never Bernoulli measures unless the system either is reducible or consists of similarity transformations. The proof builds on earlier results in the thermodynamic formalism of affine iterated function systems due to Bochi, Feng, K\"aenm\"aki, Shmerkin and the first named author and also relies on the work of Benoist on the spectral properties of Zariski-dense subsemigroups of reductive linear groups.
\end{abstract}
\maketitle

\section{Introduction}

An \emph{iterated function system} is by definition a tuple $(T_1,\ldots,T_N)$ of contracting transformations of some metric space $X$, which in this article will be taken to be $\mathbb{R}^d$. To avoid trivialities it will be assumed throughout this article that $N \geq 2$. If $(T_1,\ldots,T_N)$ is an iterated function system acting on $\mathbb{R}^d$ then it is well-known that there exists a unique nonempty compact subset $Z\subset \mathbb{R}^d$ with the property $Z =\bigcup_{i=1}^N T_iZ$, called the \emph{attractor} or \emph{limit set} of the iterated function system. If additionally any probability vector $(p_1,\ldots,p_N)$ is specified then there exists a unique Borel probability measure $m$ on $\mathbb{R}^d$ such that $m=\sum_{i=1}^Np_i (T_i)_*m$. In the case where the transformations $T_i$ are contracting similitudes of $\mathbb{R}^d$ we call the limit set $Z$ a \emph{self-similar set} and the measure $m$ a \emph{self-similar measure}.

For each $x \in \mathbb{R}^d$ and $r>0$ let $B_r(x)$ denote the open Euclidean ball with radius $r$ and centre $x$. If $m$ is a Borel probability measure $m$ on $\mathbb{R}^d$ such that the limit
\[\lim_{r\to 0} \frac{\log m(B_r(x))}{\log r}\]
exists for $m$-a.e.~$x$ and is constant $m$-a.e, we say that $m$ is \emph{exact-dimensional} and define the dimension of $m$ to be the value of this almost-everywhere limit. We denote the dimension of such a measure by $\dim m$. It was shown in 2009 by D.-J. Feng and H. Hu that every self-similar measure on $\mathbb{R}^d$ is exact-dimensional \cite{FeHu09}. We denote the Hausdorff dimension of any subset $Z$ of $\mathbb{R}^d$ by $\dimh Z$.

An iterated function system is said to satisfy the \emph{open set condition} if there exists a nonempty open set $U$ such that $T_iU \subseteq U$ for all $i=1,\ldots,N$ and such that $T_iU \cap T_jU =\emptyset$ whenever $i \neq j$, and is said to satisfy the \emph{strong open set condition} if additionally $U \cap Z \neq \emptyset$. The starting point of the motivation for this article is the following landmark theorem of J.E. Hutchinson \cite{Hu81}:
\begin{theorem}[Hutchinson]\label{th:hutch}
Let $T_1,\ldots,T_N \colon \mathbb{R}^d \to \mathbb{R}^d$ be contracting similitudes of the form $T_ix:=r_i O_ix+v_i$ for some $r_i \in (0,1)$, $O_i \in O(d)$ and $v_i \in \mathbb{R}^d$ and suppose that $(T_1,\ldots,T_N)$ satisfies the open set condition. Then the Hausdorff dimension of the attractor $Z$ of the iterated function system $(T_1,\ldots,T_N)$ is equal to the unique real number $s \in (0,d]$ such that $\sum_{i=1}^N r_i^s=1$. Moreover there exists a unique self-similar measure $m$ supported on $Z$ with dimension $s$.
\end{theorem}
The extension of Theorem \ref{th:hutch} in various directions has been an active topic of research since its original publication. One major area of research has been the problem of understanding systematically what happens when the open set condition is removed (such as in \cite{BaFrMa19,BrMoSi04,Ho15,Ho14,LiVa16,Ra17,SaShSo18,So95}) and this line of research has focused especially on the dimensions of the resulting measures as opposed to the resulting sets. A second major direction of extension of Theorem \ref{th:hutch} is that in which the transformations $T_i$ are allowed to be arbitrary affine contractions instead of similitudes: this line of research dates back to the work of Bedford, McMullen and Falconer in the 1980s \cite{Be84,Fa88,Mc84} and has been particularly active within the last few years (see for example \cite{BaHoRa19,BaKa17,BaFe13,BoMo18,DaSi17,Fe19,FeSh14,KaMo18,Ra18}). It is with this second direction of extension that this article is concerned.

When $(T_1,\ldots,T_N)$ is an iterated function system consisting of affine contractions of $\mathbb{R}^d$ the attractor of $(T_1,\ldots,T_N)$ is referred to as a \emph{self-affine set} and Borel probability measures satisfying $m=\sum_{i=1}^Np_i (T_i)_*m$  are referred to as \emph{self-affine measures}. It was shown recently by D.-J. Feng in \cite{Fe19} that every self-affine measure is exact-dimensional; previous partial results in this direction include \cite{Ba15,BaKa17,FrJoJu18}. Let us now describe the most natural generalisation of Hutchinson's dimension formula $\sum_{i=1}^N r_i^s=1$ to the affine context. We recall that the \emph{singular values} of a $d\times d$ real matrix $A$ are defined to be the square roots of the (necessarily non-negative) eigenvalues of the positive semidefinite matrix $A^\top A$. We denote the singular values of $A$ by $\sigma_1(A),\ldots,\sigma_d(A)$ where it is always understood that $\sigma_1(A) \geq \sigma_2(A) \geq \cdots \geq \sigma_d(A)$. Following the notation of \cite{Fa88},  given a $d\times d$ real matrix $A$, for each $s \geq 0$ we define the \emph{singular value function} $\varphi^s(A)$ applied to $A$ by
\[\varphi^s(A):=\left\{\begin{array}{cl}\sigma_1(A)\cdots \sigma_{\lfloor s\rfloor}(A)\sigma_{\lceil s\rceil}(A)^{s-\lfloor s\rfloor}&\text{if }0 \leq s \leq d,\\
|\det A|^{\frac{s}{d}}&\text{if }s \geq d.
\end{array}\right.\]
The singular value function satisfies the useful inequality $\varphi^s(AB) \leq \varphi^s(A)\varphi^s(B)$ for all $A,B \in \GL_d(\mathbb{R})$, as is noted in \cite{Fa88}. 
Given $(A_1,\ldots,A_N) \in \GL_d(\mathbb{R})^N$ we define the \emph{singular value pressure} of $(A_1,\ldots,A_N)$ at $s$ to be the real number
\[P(A_1,\ldots,A_N;s):=\lim_{n \to \infty} \frac{1}{n}\log \sum_{i_1,\ldots,i_n=1}^N \varphi^s\left(A_{i_n}\cdots A_{i_1}\right),\]
the existence of the limit being guaranteed by subadditivity. When $A_1,\ldots,A_N \in \GL_d(\mathbb{R})$ are contracting in the Euclidean norm (or indeed with respect to an arbitrary norm on $\mathbb{R}^d$) it is not difficult to show that the function $s \mapsto P(A_1,\ldots,A_N;s)$ is strictly decreasing and locally Lipschitz continuous and has a unique zero in $(0,+\infty)$ which we denote by $\dimaff(A_1,\ldots,A_N)$. We observe that when every $A_i$ has the form $A_i=r_iO_i$ for some $r_i \in (0,1)$ and $O_i \in O(d)$ as in Theorem \ref{th:hutch}, the pressure simplifies to $P(A_1,\ldots,A_N;s)=\log \sum_{i=1}^N r_i^s$ and thus in this case $\dimaff(A_1,\ldots,A_N)$ is simply the unique solution $s$ to Hutchinson's equation $\sum_{i=1}^N r_i^s=1$. If $(T_1,\ldots,T_N)$ is an affine iterated function system of the form $T_ix=A_ix+v_i$ then we will also find it useful to write $\dimaff(T_1,\ldots,T_N):=\dimaff(A_1,\ldots,A_N)$.

We note that the singular value potential and affinity dimension have a number of antecedents in the literature in the context of the dimension theory of attractors of dynamical systems: a version of the singular value potential was considered by Douady--Oesterl\'{e} \cite{douady-oesterle} in the study of the Hausdorff dimensions of attractors, and in the same context the relevance of the asymptotics of singular values (in the form of Lyapunov exponents) was foreseen by Kaplan--Yorke \cite{kaplan-yorke} who conjectured that in generic situations the Hausdorff dimension should be related to the growth asymptotics of singular values.

An active area of research in the theory of self-affine sets is the problem of obtaining analogues of Theorem \ref{th:hutch} for affine iterated function systems. The first general result in this direction was obtained by K. Falconer in the 1988 article \cite{Fa88}:
\begin{theorem}[Falconer]\label{th:fa}
Let $A_1,\ldots,A_N \in \GL_d(\mathbb{R})$. If $\max_{1 \leq i \leq N}\|A_i\|<\frac{1}{2}$ then for Lebesgue a.e.~$(v_1,\ldots,v_N) \in (\mathbb{R}^d)^N$ the attractor $Z$ of the iterated function system $(T_1,\ldots,T_N)$ defined by $T_ix:=A_ix+v_i$ satisfies
\[\dimh Z = \min\{d,\dimaff(A_1,\ldots,A_N)\}.\]
If $\max_{1 \leq i \leq N}\|A_i\|<1$, then for \emph{every}  $(v_1,\ldots,v_N) \in (\mathbb{R}^d)^N$ the attractor satisfies
\[\dimh Z \leq \min\{d,\dimaff(A_1,\ldots,A_N)\}.\]
\end{theorem}
Here $\|\cdot\|$ denotes the operator norm induced by the Euclidean norm. 
Falconer's original argument assumed $\max_{1 \leq i \leq N}\|A_i\|<\frac{1}{3}$, the improvement to $\frac{1}{2}$ being due to Solomyak \cite{So98}, who also noted that the value of $\frac{1}{2}$ cannot be further improved to any $\frac{1}{2}+\varepsilon$. We remark that the hypothesis $\max_{1 \leq i \leq N}\|A_i\|<\frac{1}{2}$ and the conclusion $\dimh Z = \min\{d,\dimaff(A_1,\ldots,A_N)\}$ contain a minor asymmetry: it is clear that if each $A_i$ is replaced with $X^{-1}A_iX$ for some fixed $X \in \GL_d(\mathbb{R})$ then the almost sure Hausdorff dimension $\dimh Z$ of the attractor does not change, but the condition $\max_{1 \leq i \leq N}\|A_i\|<\frac{1}{2}$ will in general be invalidated for certain choices of $X$. This asymmetry can be remedied by weakening the hypothesis to the condition $\max_{1 \leq i \leq N}\|A_i\|<\frac{1}{2}$ for the operator norm induced by \emph{some} norm $\|\cdot\|$ on $\mathbb{R}^d$, and similarly with the condition $\max_{1 \leq i \leq N}\|A_i\|<1$, and under this hypothesis Falconer's proof goes through with minimal changes. Some similar remarks relating to sufficient conditions for the existence of the attractor of $(T_1,\ldots,T_N)$ were presented in \cite[\S6]{AtBaViWi10}. To avoid similar asymmetries in our results we will assume in this article that our affine iterated function systems are contracting with respect to some unspecified norm on $\mathbb{R}^d$.

Theorem \ref{th:fa} demonstrates that the affinity dimension correctly describes the Hausdorff dimension of the attractor in a large range of cases, but this result inherently does not apply to explicit, specific examples of affine iterated function systems. Since the publication of \cite{Fa88} an active line of research, especially in recent years, has therefore been that of extending Theorem \ref{th:fa} to explicit affine iterated function systems for which the vectors $v_i$ are fixed and some version of the open set condition is satisfied (see for example \cite{FaKe18,HuLa95,MoSh19}). In this direction the following powerful result was obtained recently by B. B\'ar\'any, M. Hochman and A. Rapaport \cite{BaHoRa19}:
\begin{theorem}[B\'ar\'any-Hochman-Rapaport]\label{th:bahora}
Let $(T_1,\ldots,T_N)$ be an affine iterated function system acting on $\mathbb{R}^2$ and satisfying the strong open set condition, where each $T_i$ is contracting with respect to the Euclidean norm. Let us write $T_ix:=A_ix+v_i$ for every $i=1,\ldots,N$ and suppose that each $A_i$ is invertible. Suppose that the linear maps $|\det A_i|^{-1/2}A_i$ are not contained in a compact subgroup of $\GL_2(\mathbb{R})$ and do not preserve a finite union of one-dimensional subspaces of $\mathbb{R}^2$.  Then the Hausdorff dimension of the attractor of $(T_1,\ldots,T_N)$ is equal to $\dimaff(A_1,\ldots,A_N)$. 
\end{theorem}
In dimension $d>2$ the problem of obtaining an analogue of Theorem \ref{th:bahora} is substantially more challenging. At the time of writing, no explicit examples of affine iterated function systems in dimension higher than two are yet known where the Hausdorff and affinity dimensions coincide, other than those which fall within the scope of Theorem \ref{th:hutch}. On the other hand, in the broader setting of limit sets of actions of non-conformal transformations, Dufloux \cite{dufloux} has successfully computed the Hausdorff dimension of limit sets on the boundary $\partial H^n_\C$ of the $n$-dimensional complex hyperbolic space associated to well-positioned Schottky subgroups. We also note the work of Pozzetti--Sambarino--Wienhard \cite{PSW} who, under an asymptotic conformality assumption, have successfully calculated the Hausdorff dimensions of limit sets in projective spaces.

Returning to our setting of affine iterated function systems, while Theorems \ref{th:fa} and \ref{th:bahora} extend the part of Theorem \ref{th:hutch} which describes the dimension of the attractor, a feature which has no direct parallel in Theorem \ref{th:bahora} in particular is the question of whether or not there exists a measure supported on the attractor of the affine iterated function system $(T_1,\ldots,T_N)$ having dimension equal to the affinity dimension. While we conjecture that this should indeed be the case in the context of Theorem \ref{th:bahora} and its presumed higher-dimensional analogues (and indeed it is known that such measures exist generically in the sense of Theorem \ref{th:fa} -- see \cite{JoPoSi07}) in this article we will focus on a narrower question: under what circumstances does an affine iterated function system $(T_1,\ldots,T_N)$ acting on $\mathbb{R}^d$ admit a \emph{self-affine} measure with dimension equal to the affinity dimension? 

Theorem \ref{th:hutch} indicates that this phenomenon occurs when the affine transformations are all similitudes, or more generally when they are simultaneously conjugated to similitudes by some linear transformation of $\mathbb{R}^d$. In this situation it was observed by P. Mattila that while the open set condition is sufficient for the existence of a self-similar measure with dimension equal to the affinity dimension, it is not necessary for it (see the introduction to \cite{Sc94}). One may also show that self-affine measures with dimension equal to the affinity dimension can arise in certain circumstances when the linear parts of the affinities admit a common invariant subspace, or when the affinity dimension is precisely equal to $d$. The objective of this article is to demonstrate that these are the \emph{only} situations in which this phenomenon occurs.

Henceforth we shall say that a subset $\mathsf{A}$ of $\GL_d(\mathbb{R})$ is \emph{irreducible} if there does not exist any proper nonzero subspace of $\mathbb{R}^d$ preserved by every $A \in \mathsf{A}$, and \emph{strongly irreducible} if a finite union of such subspaces is not preserved by every element of $\mathsf{A}$. When $\mathsf{A}$ is not irreducible it will be called \emph{reducible}. Clearly $\mathsf{A}$ is (strongly) irreducible if and only if the semigroup generated by $\mathsf{A}$ is. We will at times abuse notation by saying that a tuple $(A_1,\ldots,A_N)$ is (strongly) irreducible if and only if the corresponding set is. Our main result is as follows:
\begin{theorem}\label{th:main}
Let $T_1,\ldots,T_N$ be invertible affine transformations of $\mathbb{R}^d$ having the form $T_ix:=A_ix+v_i$ for some $v_1,\ldots, v_N \in \mathbb{R}^d$, where $(A_1,\ldots,A_N) \in \GL_d(\mathbb{R})^N$ has the following four properties:
\begin{enumerate}[(i)]
\item
There exists a norm $\trip{\cdot}$ on $\mathbb{R}^d$ such that $\trip{A_i}<1$ for every $i=1,\ldots,N$;
\item
The affinity dimension $\dimaff(A_1,\ldots,A_N)$ is strictly between $0$ and $d$;
\item
The tuple $(A_1,\ldots,A_N)$ is irreducible;
\item
There does not exist an inner product on $\mathbb{R}^d$ with respect to which the linear maps $A_1,\ldots,A_N$ are similitudes.
\end{enumerate}
Then every self-affine measure $m=\sum_{i=1}^Np_i(T_i)_*m$ satisfies $\dim m < \dimaff(A_1,\ldots,A_N)$.

Furthermore this property is locally uniform in the following sense. Suppose that $\mathsf{K} \subset \GL_d(\mathbb{R})^N$ is a compact set such that every $(A_1,\ldots,A_N) \in \mathsf{K}$ satisfies hypotheses (i)--(iv) above. This applies in particular if $(B_1,\ldots,B_N) \in \GL_d(\mathbb{R})^N$ satisfies (i)--(iv) above and $\mathsf{K}$ is a sufficiently small compact neighbourhood of $(B_1,\ldots,B_N)$. Then there exists $\kappa>0$ depending on $\mathsf{K}$ with the following property: if $(A_1,\ldots,A_N) \in \mathsf{K}$, and $T_1,\ldots,T_N \colon \mathbb{R}^d \to \mathbb{R}^d$ are affine transformations of the form $T_ix=A_ix+v_i$ for some vectors $v_1,\ldots,v_N$, and $m=\sum_{i=1}^N p_i(T_i)_*m$ is a self-affine measure on $\mathbb{R}^d$ for some probability vector $(p_1,\ldots,p_N)$, then $\dim m \leq \dimaff(A_1,\ldots,A_N)-\kappa$.  
\end{theorem}
In stating this result we have taken advantage of the fact that every self-affine measure on $\mathbb{R}^d$ is exact-dimensional, but this result is not required in our proof. The proof of Theorem \ref{th:main} in fact shows that the upper packing dimension of the measure $m$,
\[{\ess \sup}_m \limsup_{r \to \infty} \frac{\log m(B_r(x))}{\log r},\]
is bounded by $\dimaff(A_1,\ldots,A_N)-\kappa$. This in turn is achieved by showing that the \emph{Lyapunov dimension} of an appropriate measure on the coding space $\Sigma_N:=\{1,\ldots,N\}^{\mathbb{N}}$ is bounded by $\dimaff(A_1,\ldots,A_N)-\kappa$. The Lyapunov dimension is relatively technical to describe and would be digressive to define in this introduction, so we defer further discussion of this point to \S\ref{se:proof-of-main} below.

The condition that the linear maps $A_i$ are not all similitudes with respect to some inner product on $\mathbb{R}^d$ is equivalent to the statement that the linear maps $|\det A_i|^{-1/d}A_i$ are not all contained in some compact subgroup of $\GL_d(\mathbb{R})$, and we will at times prefer the latter formulation in the proofs. To see that these statements are equivalent we observe that if $G \leq \GL_d(\mathbb{R})$ is a compact group containing the linear maps $|\det A_i|^{-1/d}A_i$, $\langle \cdot,\cdot\rangle$ denotes the standard inner product on $\mathbb{R}^d$, and $H$ is the normalised Haar measure on $G$, the formula
\[\langle u,v\rangle_G:= \int_G \langle Bu,Bv\rangle dH(B)\]
may easily be verified to define an inner product on $\mathbb{R}^d$ which is invariant under the action of elements of $G$. In particular the transformations $A_i$ are similitudes with respect to this inner product structure. The converse direction of implication is obvious. Theorem \ref{th:main} therefore admits the following corollary which motivates the title of this work:
\begin{corollary}\label{co:converse-hutchinson}
Let $T_1,\ldots,T_N \colon \mathbb{R}^d \to \mathbb{R}^d$ be invertible affine transformations which are contracting with respect to some norm on $\mathbb{R}^d$. Let us write $T_ix=A_ix+v_i$ for all $x \in \mathbb{R}^d$ and $i=1,\ldots,N$, and suppose that $\{A_1,\ldots,A_N\}$ is irreducible. If there exists a self-affine measure $m=\sum_{i=1}^N p_i(T_i)_*m$ such that $\dim m=\dimaff(A_1,\ldots,A_N) \in (0,d)$, then there exists an inner product on $\mathbb{R}^d$ with respect to which the transformations $T_i$ are all similitudes.
\end{corollary}

\begin{figure}%
    \centering
    \subfloat[The classical self-similar Sierpi\'nski gasket $X_1$.]{{\includegraphics[width=5.8cm]{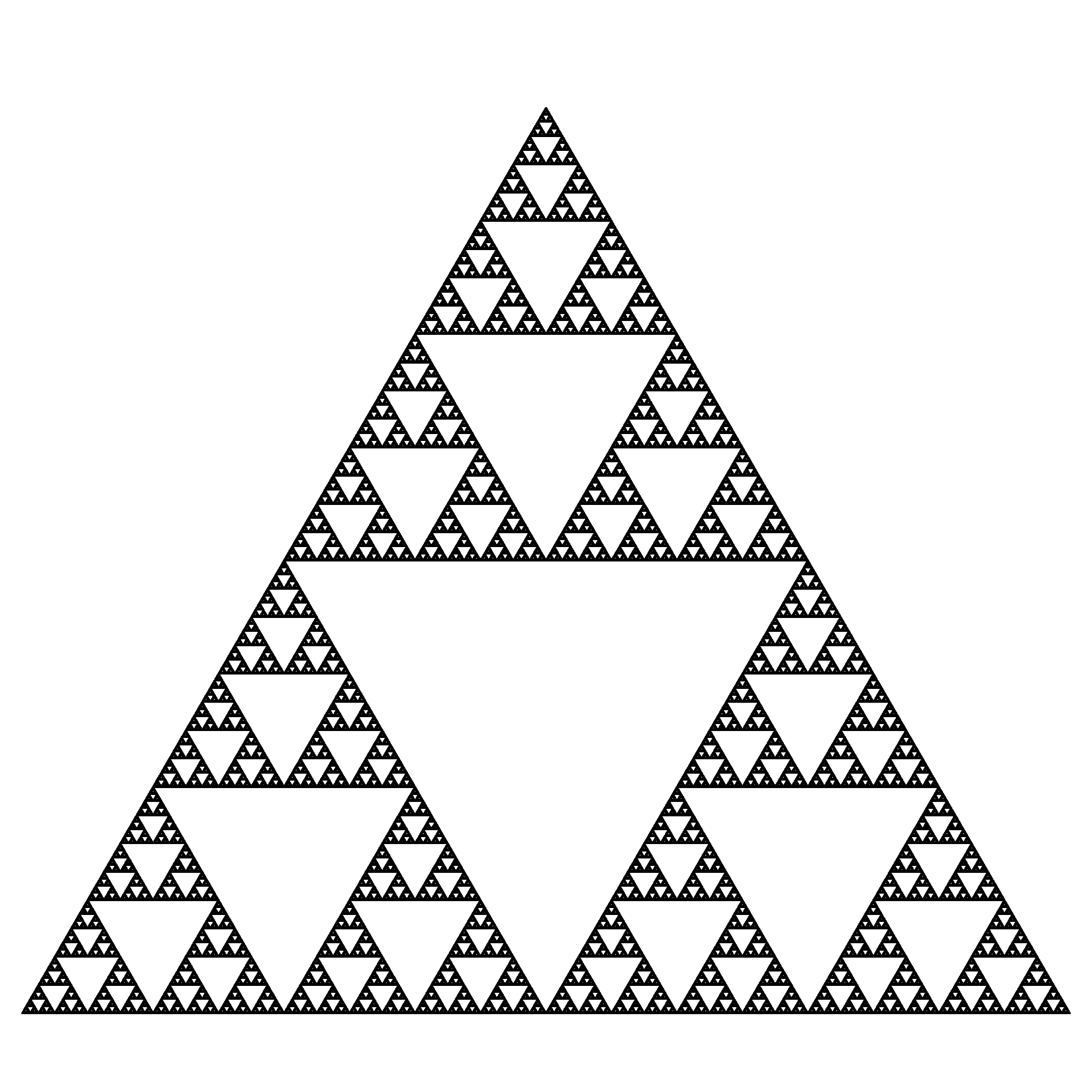}}}
    \qquad
    \subfloat[A self-affine gasket $X_2$ which is not self-similar.]{{\includegraphics[width=5.8cm]{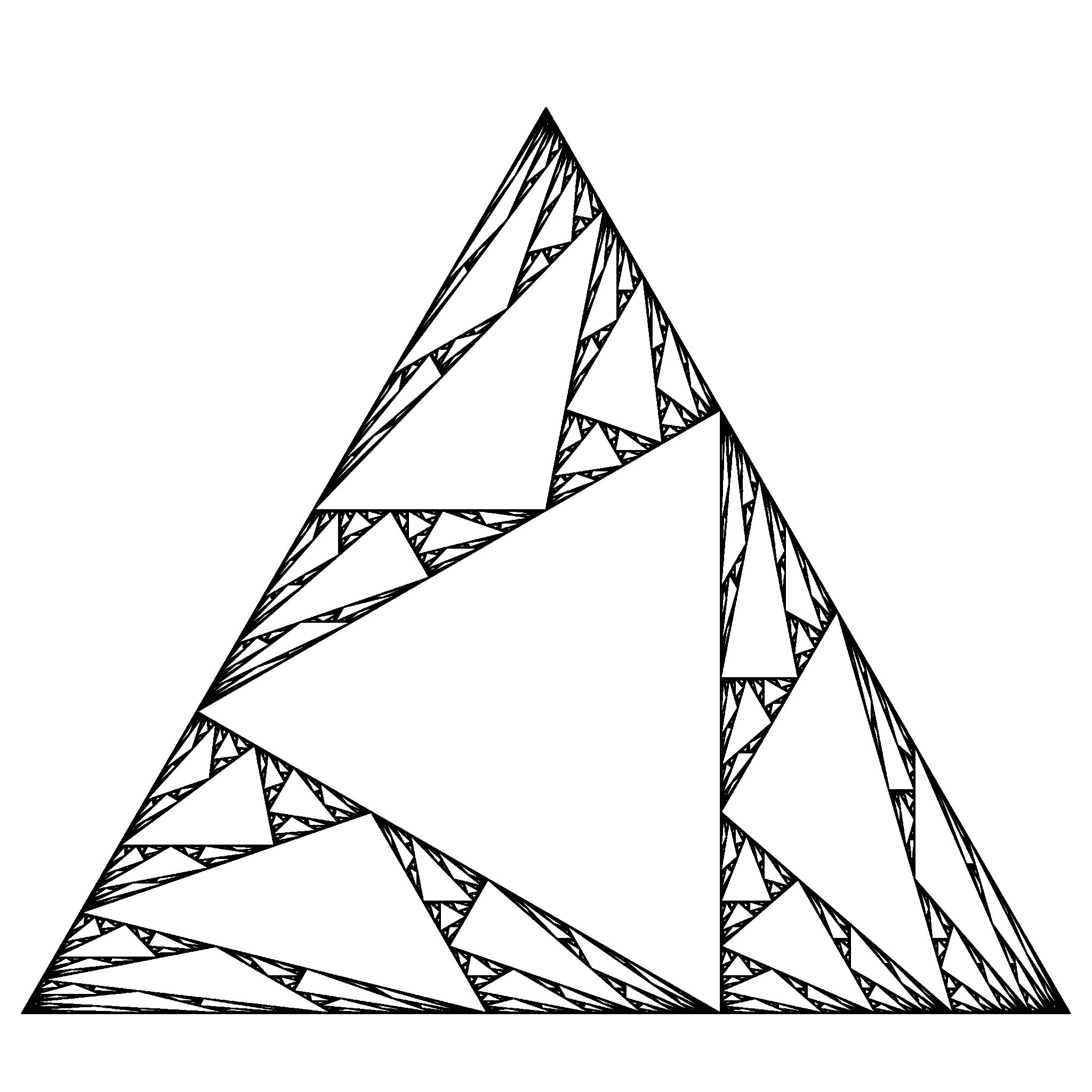}}}
    \caption{By Theorem \ref{th:hutch} there exists a self-similar measure supported on the classical Sierpi\'nski gasket $X_1$ with dimension equal to the Hausdorff dimension of the set itself,  $\log 3/\log 2$. This measure corresponds to that defined simply by giving measure $\frac{1}{3}$ to each of the three copies of $X_1$ with diameter half that of the original, measure $\frac{1}{9}$ to each of the nine sub-copies with diameter $\frac{1}{4}$ that of the original, and so forth. By the combination of Theorems \ref{th:bahora} and \ref{th:main}, for the self-affine gasket $X_2$ there is a gap between the maximum possible dimension of a self-affine measure supported on $X_2$ and the Hausdorff dimension of $X_2$ itself. }
    \label{fi:onlyfigure}%
\end{figure}

We note that the affinity dimension of an invertible affine iterated function system is never zero and therefore the endpoint case $\dimaff(A_1,\ldots,A_N)=0$ of Theorem \ref{th:main} cannot occur. In the other endpoint case $\dimaff(A_1,\ldots,A_N)=d$ it is easy to construct examples in which the normalised restriction of Lebesgue measure to a convex polyhedral body in $\mathbb{R}^d$ may be represented as a self-affine measure with respect to affine transformations which are not simultaneously conjugate to similitudes and whose linear parts do not admit an invariant proper subspace. For example, if $U \subset \mathbb{R}^2$ is an open triangular region then up to Lebesgue measure zero it may be bisected along a line passing through one vertex and its opposite edge into the union of two smaller triangular regions $U_1$ and $U_2$, each having two side lengths smaller than those of the original triangle and one side length in common with it. Taking further bisections if necessary $U$ may be written up to measure zero as a finite union of strictly smaller triangular regions $V_1,\ldots,V_N$ each of which is the image of $U$ under some contracting affine transformation $T_i$. It is clear that if $m$ denotes the normalised Lebesgue measure on $U$ then it satisfies the relation $m=\sum_{i=1}^N m(V_i)(T_i)_*m$ and hence is a self-affine measure with respect to $(T_1,\ldots,T_N)$ which has dimension $2$. In general this construction may be performed in such a way as to ensure that hypotheses (i),(iii) and (iv) of Theorem \ref{th:main} are satisfied; moreover the linear parts of the affinities may be taken to be \emph{strongly} irreducible. The details of this aspect of the construction and of its generalisation to higher dimensions are left to the reader.

We remark that if in Theorem \ref{th:main} instead of measures of the form $m=\sum_{i=1}^N p_i(T_i)_*m$ we were to consider the larger category of 
Borel probability measures $m$ which satisfy an equation of the form
\begin{equation}\label{eq:bigsa}m=\sum_{i_1,\ldots,i_n=1}^N q_{(i_1,\ldots,i_n)} (T_{i_1}\cdots T_{i_n})_*m\end{equation}
for some $n\geq 1$ and some probability vector $(q_{(1,\ldots,1)}, \ldots,q_{(N,\ldots,N)}) \in \mathbb{R}^{N^n}$, then no dimension gap would occur. In two dimensions it is known that the supremum of the Hausdorff dimensions of measures which are self-affine in the broader sense of \eqref{eq:bigsa} can be equal to the affinity dimension $\dimaff(A_1,\ldots,A_N)$ when the conditions of Theorem \ref{th:main} are satisfied. Indeed this fact played a significant role in the proof of Theorem \ref{th:bahora} by extending the results of \cite{MoSh19} which pertain to self-affine measures into a result concerning self-affine sets. Theorem \ref{th:main} demonstrates that outside the context of similarity transformations this supremum is attained only in degenerate cases in which a common invariant subspace exists.

To conclude this introduction let us briefly outline how Theorem \ref{th:main} will be proved. If $T_1,\ldots,T_N$ are contractions of $\mathbb{R}^d$ with respect to some fixed norm then there exists a well-defined coding map $\Pi \colon \{1,\ldots,N\}^{\mathbb{N}} \to \mathbb{R}^d$ with the property
\[\Pi\left[(x_k)_{k=1}^\infty\right] = \lim_{n \to \infty} T_{x_1}\cdots T_{x_n}v\]
for all $v \in \mathbb{R}^d$, and whose image is precisely the attractor of $(T_1,\ldots,T_N)$. It is a well-known result due to Hutchinson \cite[\S4]{Hu81} that a Borel probability measure $m$ on $\mathbb{R}^d$ satisfies $m=\sum_{i=1}^N p_i(T_i)_*m$ if and only if it satisfies $m=\Pi_* \mu$ where $\mu$ is the Bernoulli measure $(\sum_{i=1}^N p_i\delta_i)^{\mathbb{N}}$ on $\{1,\ldots,N\}^{\mathbb{N}}$. This measure $\mu$ is an ergodic invariant measure with respect to the shift transformation $\sigma \colon \{1,\ldots,N\}^{\mathbb{N}} \to \{1,\ldots,N\}^{\mathbb{N}}$ defined by $\sigma[(x_k)_{k=1}^\infty]:=(x_{k+1})_{k=1}^\infty$. 

Now, using a combination of results of A. K\"aenm\"aki \cite{Ka04} and T. Jordan, M. Pollicott and K. Simon \cite{JoPoSi07}, one may show that if an ergodic shift-invariant measure $\mu$ on $\{1,\ldots,N\}^n$ has the property $\dim \Pi_*\mu = \dimaff(T_1,\ldots,T_N)$ then it necessarily maximises the quantity
\[h(\mu)+ \lim_{n \to \infty} \frac{1}{n}\int \log \varphi^s(A_{x_1}\cdots A_{x_n})d\mu\left[(x_k)_{k=1}^\infty\right]\]
over all shift-invariant Borel probability measures on $\{1,\ldots,N\}^{\mathbb{N}}$, where $s:=\dimaff(T_1,\ldots,T_N)$, $A_i$ denotes the linear part of the affine transformation $T_i$ and $h(\mu)$ denotes the entropy of the measure $\mu$ with respect to the transformation $\sigma$. Measures which maximise this quantity have been named \emph{K\"aenm\"aki measures}. The critical step in proving Theorem \ref{th:main} is to show that under the hypotheses of that theorem there cannot exist a K\"aenm\"aki measure which is also a Bernoulli measure. The dimension gap result then follows by relatively straightforward compactness considerations.

The proof of this statement relies on a general theorem on the structure of K\"aenm\"aki measures which was established by J. Bochi and the first named author in \cite{BoMo18}, building on the earlier works \cite{FeKa11} and \cite{KaMo18}. Let us illustrate how this argument functions in a simple special case. Suppose that the semigroup generated by $A_1,\ldots,A_N$ is \emph{Zariski dense} as a subgroup of $\GL_d(\mathbb{R})$: that is, suppose that every function $\phi \colon \GL_d(\mathbb{R}) \to \mathbb{R}$ which corresponds to a polynomial function of the matrix entries and vanishes on the semigroup generated by $A_1,\ldots,A_N$ also vanishes identically on $\GL_d(\mathbb{R})$. 
(Equivalently, $A_1,\ldots,A_N$ is not contained in any proper algebraic subgroup of $\GL_d(\mathbb{R})$.) Then it follows by a result of A. K\"aenm\"aki and the first named author in \cite{KaMo18} that if $\mu$ is a K\"aenm\"aki measure for $(T_1,\ldots,T_N)$ then it satisfies
\begin{equation}\label{eq:bg2}C^{-1}\leq \frac{\mu(\{(x_k)\colon x_j=i_j\text{ for all }j=1,\ldots,n\})}{\varphi^s(A_{i_1}\cdots A_{i_n})} \leq C\end{equation}
for some constant $C>1$, for all $i_1,\ldots,i_n \in \{1,\ldots,N\}$ and $n \geq 1$. But if $\mu$ is also a Bernoulli measure, the value of the numerator depends only on which symbols appear in the sequence $i_1,\ldots,i_n$ and not on the order in which those symbols appear. This implies that the same property must hold for $\varphi^s(A_{i_1}\cdots A_{i_n})$ up to the introduction of a scalar multiplicative factor $C^2$. Using this principle one may deduce that if $B_1$ and $B_2$ belong to the semigroup generated by $A_1,\ldots,A_N$ then necessarily
\begin{equation}\label{eq:basicgibbs}C^{-3} \leq \frac{\varphi^s((B_1B_2)^n)}{\varphi^s(B_1^n)\varphi^s(B_2^n)} \leq C^3\end{equation}
for every $n \geq 1$. Now if $\lambda_i(B)$ denotes the $i^{\mathrm{th}}$ largest of the absolute values of the $d$ eigenvalues of $B \in \GL_d(\mathbb{R})$, and $0<s<d$, one may show that
\[\lim_{n \to \infty} \varphi^s(B^n)^{\frac{1}{n}}=\lambda_1(B)\cdots \lambda_{\lfloor s\rfloor}(B) \lambda_{\lceil s\rceil}(B)^{s-\lfloor s\rfloor}=:\xi^s(B).\]
Taking the power $\frac{1}{n}$ and letting $n \to \infty$ in \eqref{eq:basicgibbs} it follows that the function $\xi^s$ just defined satisfies $\xi^s(B_1B_2)=\xi^s(B_1)\xi^s(B_2)$ for all $B_1,B_2$ in the semigroup generated by the linear maps $A_1,\ldots,A_N$. But this turns out to be impossible for a semigroup which is Zariski dense in $\GL_d(\mathbb{R})$, essentially by a theorem of Y.~Benoist (later reproven by J.-F.~Quint using a different method, see Theorem 7.4 and Proposition 9.8 of \cite{bq.book} and additionally \cite{benoist.linear2,quint.schottky}). 

The extension of this argument to the more general circumstances of Theorem \ref{th:main} requires us to engage with a number of complications. Similarly to the special case described above, the core of the proof operates by assuming that hypotheses (i)--(iii) of Theorem \ref{th:main} hold and that a K\"aenm\"aki measure exists which is a Bernoulli measure, and proceeds to show that the linear maps $|\det A_i|^{-1/d}A_i$ necessarily belong to a compact group, contradicting (iv). In general under the hypotheses of Theorem \ref{th:main} there may be multiple inequivalent K\"aenm\"aki measures. (This remains true even under slightly stronger hypotheses: see \cite{MoSe19}.) The hypotheses imply that at least one of these measures is Bernoulli, but \emph{a priori} other K\"aenm\"aki measures may not be. In this case the denominator of \eqref{eq:bg2} will not correspond to the function $\varphi^s(A_{i_1}\cdots A_{i_n})$ but to some more complicated function derived from the action of $A_{i_1}\cdots A_{i_n}$ on finite unions of proper subspaces of exterior powers of $\mathbb{R}^d$ (see \cite[\S2]{BoMo18}). The more complicated structure of this function necessitates further steps in order to deduce the multiplicativity of some analogue of the function $\xi^s$ defined above, which in general will correspond to some spectral data relating to the action of a finite-index subsemigroup of the semigroup generated by $A_1,\ldots,A_N$ on certain pairs of subspaces of exterior powers of $\mathbb{R}^d$. This multiplicativity will allow us to show that certain homomorphic images of a finite-index subsemigroup of the semigroup generated by $|\det A_1|^{-1/d} A_1,\ldots,|\det A_N|^{-1/d}A_N$ are contained in compact groups, and this can be applied to deduce that the elements of that finite-index subsemigroup act as ``simultaneously normal'' linear maps on certain subspaces of particular exterior powers of $\mathbb{R}^d$: that is, on those spaces there exists an inner product structure with respect to which the linear maps act as orthogonal direct sums of linear similitudes. An extensive additional argument is then required to show that these normal linear maps actually \emph{are} similitudes. This additional argument makes use of the variational characterisation of K\"aenm\"aki measures to bound a weighted sum of the Lyapunov exponents of the other K\"aenm\"aki measures and so force the remaining K\"aenm\"aki measures to also be Bernoulli measures. It is then straightforward to deduce that the entire semigroup generated by $|\det A_1|^{-1/d} A_1,\ldots,|\det A_N|^{-1/d}A_N$ acts on these subspaces of exterior powers by similitudes. Still further arguments are required to deal with the possibility that these subspaces of the exterior powers may be proper.  The first two parts of the argument, in which the linear maps are first shown to act normally and then shown to act by similitudes on certain subspaces of exterior powers, are dealt with in section \ref{se:irr-case}. The final part, in which the action on proper subspaces of exterior powers is related to the action on $\mathbb{R}^d$, forms a separate argument which is presented in section \ref{se:gen-case}. 

The remainder of the article is therefore structured as follows. In the following section we review such background on the thermodynamic formalism of affine iterated function systems as is necessary to state our main technical theorem, Theorem \ref{th:main-tech}, which asserts that under the hypotheses of Theorem \ref{th:main} a K\"aenm\"aki measure cannot be a Bernoulli measure. In section \ref{se:proof-of-main} we derive Theorem \ref{th:main} from Theorem \ref{th:main-tech}; this is the most technically straightforward part of the proof of Theorem \ref{th:main}. Section \ref{se:rev} then reviews key concepts from the theory of linear algebraic groups which will be used in the proof of Theorem \ref{th:main-tech}. Section \ref{se:irr-case} proves a key special case of Theorem \ref{th:main-tech} in which the irreducibility of certain representations is assumed, and section \ref{se:gen-case} applies this result to deduce the general case.

During peer review it was brought to our attention that some of the technical arguments underlying Theorem \ref{th:main} may be expressed in intrinsic terms as a statement concerning potentials defined in terms of reductive linear algebraic groups. This is discussed in more detail in the appendix.

%
%

\section{Subadditive thermodynamic formalism and the main technical theorem}\label{se:tech-thm}

Let $\Sigma_N$ denote the set $\{1,\ldots,N\}^{\mathbb{N}}$ equipped with the infinite product topology (with respect to which it is compact and metrisable) and let $\sigma \colon \Sigma_N \to \Sigma_N$ denote the shift transformation $(x_k)_{k=1}^\infty \mapsto (x_{k+1})_{k=1}^\infty$ which is a continuous surjection. When $N$ is understood let $\mathcal{M}_\sigma$ denote the set of all $\sigma$-invariant Borel probability measures on $\Sigma_N$. Via the Riesz representation theorem we identify $\mathcal{M}_\sigma$ with a subset of $C(\Sigma_N)^*$ equipped with the corresponding weak-* topology, and in this topology it is compact and metrisable; a sequence of measures $(\mu_n)_{n=1}^\infty$ in $\mathcal{M}_\sigma$ converges to a measure $\mu \in \mathcal{M}_\sigma$ if and only if $\lim_{n \to \infty} \int f\,d\mu_n=\int f\,d\mu$ for every $f \in C(\Sigma_N)$. 

We define $\Sigma_N^*$ to be the set of all finite sequences $\iii=(i_k)_{k=1}^n \in \{1,\ldots,N\}^n$, which we refer to as \emph{words}. If $\iii=(i_k)_{k=1}^n$ then we write $|\iii|=n$ and define this to be the  \emph{length} of the word $\iii$. Given two words $\iii=(i_k)_{k=1}^n, \jjj=(j_k)_{k=1}^m \in \Sigma_N^*$ we define their concatenation $\iii \jjj$ to be the word of length $|\iii|+|\jjj|=n+m$ with first $n$ symbols $i_1,\ldots,i_n$ and subsequent symbols $j_1,\ldots,j_m$. We define the concatenation of more than two words (e.g. $\iii \jjj \kkk$ where $\iii,\jjj,\kkk \in \Sigma_N^*$) in the obvious fashion, and if $\iii \in \Sigma_N^*$ and $n \geq 1$ we let $\iii^n$ denote the concatenation $\iii \iii \cdots \iii$ of $n$ copies of $\iii$. If $A_1,\ldots,A_N \in \GL_d(\mathbb{R})$ are understood then we write $A_\iii:=A_{i_1}\cdots A_{i_n}$ and observe that $A_\iii A_\jjj = A_{\iii\jjj}$ for all $\iii,\jjj \in \Sigma_N^*$. If $x=(x_k)_{k=1}^\infty \in \Sigma_N$ then we define $x|_n$ to be the word $(x_k)_{k=1}^n \in \Sigma_N^*$. If $\iii \in \Sigma_N^\ast$ then we define the \emph{cylinder set} $[\iii]$ to be the set of all $x \in \Sigma_N$ such that $x|_n=\iii$. Every cylinder set is clopen and cylinder sets form a basis for the topology of $\Sigma_N$. The linear span of the set of all characteristic functions of cylinder sets is dense in $C(\Sigma_N)$ and therefore a sequence of measures $(\mu_n)_{n=1}^\infty$ in $\mathcal{M}_\sigma$ converges to a measure $\mu \in \mathcal{M}_\sigma$ if and only if $\lim_{n \to \infty} \mu_n([\iii])=\mu([\iii])$ for every $\iii \in \Sigma_N^*$.

We will say that $\mu \in \mathcal{M}_\sigma$ is a \emph{Bernoulli measure} if there exists a probability vector $(p_1,\ldots,p_N)$ such that $\mu([i_1\cdots i_n])=p_{i_1}\cdots p_{i_n}$ for all $i_1,\ldots,i_n \in\{1,\ldots,N\}$ and all $n \geq 1$. (We permit cases in which some of the entries of the probability vector are zero.) Clearly Bernoulli measures on $\Sigma_N$ are in one-to-one correspondence with probability vectors $(p_1,\ldots,p_N)$. It is not difficult to see that the natural map from the $(N-1)$-simplex of probability vectors to the set of corresponding Bernoulli measures on $\Sigma_N$ is weak-* continuous, and in particular the set of all Bernoulli measures on $\Sigma_N$ is weak-* compact. Every Bernoulli measure is ergodic with respect to $\sigma$.

Let us say that a \emph{submultiplicative potential}, or simply a \emph{potential}, is a function $\Phi \colon \Sigma_N^* \to (0,+\infty)$ such that $\Phi(\iii \jjj) \leq \Phi(\iii)\Phi(\jjj)$ for all $\iii,\jjj \in \Sigma_N^*$. We define the \emph{pressure} of $\Phi$ to be the limit
\[P(\Phi):=\lim_{n \to \infty} \frac{1}{n}\log \sum_{|\iii|=n}\Phi(\iii)\]
and observe that this limit exists by subadditivity. If $\Phi$ is a submultiplicative potential then we define a sequence of functions $\Phi_n \colon \Sigma_N \to (0,+\infty)$ by $\Phi_n(x):=\Phi(x|_n)$ for every $x \in \Sigma_N$ and $n \geq 1$. In this case we observe that each $\Phi_n$ is continuous (since it depends on only finitely many co-ordinates of $x \in \Sigma_N$) and that the subadditivity property $\log \Phi_{n+m}(x) \leq \log \Phi_n(\sigma^mx)+\log \Phi_m(x)$ is satisfied by the sequence of continuous functions $\log \Phi_n \colon \Sigma_N \to \mathbb{R}$. As a consequence of this property, for each ergodic $\mu \in \mathcal{M}_\sigma$ the following limit exists and defines
\[\Lambda(\Phi,\mu):=\lim_{n \to \infty}\frac{1}{n}\int \log \Phi_n(x)\,d\mu(x) = \lim_{n \to \infty}\frac{1}{n}\sum_{|\iii|=n} \mu([\iii])\log\Phi(\iii) \in [-\infty,+\infty).\] The next result is a special case of the subadditive variational principle of Cao, Feng and Huang (\cite[Theorem 1.1]{CaFeHu08}):
\begin{proposition}\label{pr:varp}
Let $N \geq 2$ and let $\Phi \colon \Sigma_N \to (0,+\infty)$ be a submultiplicative potential. Then
\begin{equation}\label{eq:eq}P(\Phi)=\sup_{\mu \in \mathcal{M}_\sigma}\left[ h(\mu) + \Lambda(\Phi,\mu)\right].\end{equation}
\end{proposition}
When $\mu$ attains the supremum \eqref{eq:eq} we call it an \emph{equilibrium state} for the potential $\Phi$. If $\Phi$ is a submultiplicative potential then by subadditivity
\[\Lambda(\Phi,\mu)=\inf_{n \geq 1}\frac{1}{n}\sum_{|\iii|=n}\mu([\iii])\log\Phi(\iii)\]
and also
\[h(\mu)=\lim_{n \to \infty} \frac{1}{n}\sum_{|\iii|=n}-\mu([\iii])\log\mu([\iii]) = \inf_{n \geq 1}\frac{1}{n}\sum_{|\iii|=n}-\mu([\iii])\log\mu([\iii])\]
and since each function $\mu \mapsto \mu([\iii])$ is continuous, these formulas imply that the function $\mu \mapsto h(\mu)+\Lambda(\Phi,\mu)$ is the pointwise infimum of a family of continuous functions $\mathcal{M}_\sigma \to \mathbb{R}$, and hence is an upper semi-continuous function $\mathcal{M}_\sigma \to [-\infty,+\infty)$. In particular it attains its maximum by the compactness of $\mathcal{M}_\sigma$ and so at least one equilibrium state exists for any specified potential $\Phi$.

A submultiplicative potential $\Phi$ will be called \emph{quasi-multiplicative} if there exist a finite set $F \subset \Sigma_N^*$ and a real number $\delta>0$ such that
\begin{equation}\label{eq.def.quasimult}
\max_{\kkk \in F} \Phi(\iii\kkk\jjj)\geq \delta \Phi(\iii)\Phi(\jjj)
\end{equation}
for all $\iii,\jjj \in \Sigma_N^*$. The significance of this condition is that it both guarantees the uniqueness of the equilibrium state of $\Phi$ and provides explicit information about its structure:
\begin{proposition}\label{pr:qm-unique}
Let $\Phi \colon \Sigma_N^* \to \mathbb{R}$ be a submultiplicative and quasi-multiplicative potential. Then there exists a unique equilibrium state $\mu$ for $\Phi$. Furthermore there exists $C>0$ such that
\[C^{-1}e^{-|\iii|P(\Phi)} \Phi(\iii) \leq \mu([\iii]) \leq Ce^{-|\iii|P(\Phi)}\Phi(\iii)\]
for all $\iii \in \Sigma_N^*$.
\end{proposition}
We refer to the above inequality between $\mu([\iii])$ and $\Phi(\iii)$ as the \emph{Gibbs inequality} for the potential $\Phi$ and measure $\mu$. 
Proposition \ref{pr:qm-unique} has been proved and re-proved in various forms across a number of works: we mention for example \cite[Theorem 5.5]{Fe11}, \cite[\S3]{KaRe14}.

The fundamental example of a potential from the perspective of this article will be the singular value potential $\Phi^s(\iii):=\varphi^s(A_\iii)$, where $A_1,\ldots,A_N \in \GL_d(\mathbb{R})$ are understood; this potential was investigated extensively by A. K\"aenm\"aki in \cite{Ka04} and the properties of its equilibrium states were developed in subsequent articles such as \cite{BoMo18,FeKa11,KaMo18}. Our argument will however require us to work with potentials which have a unique equilibrium state, and the singular value potential does not have this property unless additional constraints are imposed beyond the hypotheses of Theorem \ref{th:main}. In particular, although the irreducibility of $(A_1,\ldots,A_N)$ as hypothesised in Theorem \ref{th:main} ensures this uniqueness for $d=2$, it is not sufficient for this when $d>2$ and $1<s<d-1$ (see for example \cite[\S9]{KaMo18}). This problem cannot be alleviated by assuming strong irreducibility in place of irreducibility \cite{MoSe19}. 

The core technical result of this article is the following: 
\begin{theorem}\label{th:main-tech}
Let $(A_1,\ldots,A_N)\in \GL_d(\mathbb{R})^N$ be irreducible and define a potential $\Phi \colon \Sigma_N^* \to (0,+\infty)$ by 
\[\Phi(\iii):=\prod_{i=1}^d \sigma_i(A_\iii)^{\alpha_i}\]
where $\alpha_1 \geq \alpha_2 \geq \cdots \geq \alpha_d \geq 0$ and $\alpha_1>\alpha_d$. If $\Phi$ has an equilibrium state which is a Bernoulli measure then the linear maps $|\det A_1|^{-1/d}A_1,\ldots,  |\det A_N|^{-1/d}A_N$ are all contained in a compact subgroup of $\GL_d(\mathbb{R})$. 
\end{theorem}
We observe that the submultiplicativity of the above potential $\Phi$ follows from the inequality
\begin{equation}\label{eq:svsa}\prod_{i=1}^k \sigma_i(AB) \leq \prod_{i=1}^k \sigma_i(A) \cdot \prod_{i=1}^k \sigma_i(B)\end{equation}
which is valid for all linear maps $A,B \colon \mathbb{R}^d \to \mathbb{R}^d$ and all $k=1,\ldots,d$, since we may write 
\[\prod_{i=1}^d \sigma_i(A_\iii)^{\alpha_i} = \prod_{k=1}^d \left(\prod_{i=1}^k \sigma_i(A_\iii)\right)^{\alpha_k-\alpha_{k+1}}\]
where $\alpha_{d+1}:=0$. We will find it convenient to approach the inequality \eqref{eq:svsa} via norms on exterior powers of $\mathbb{R}^d$, but an elementary proof may be found in for example \cite[Theorem 3.3.4]{HoJo94}.

If $0<s<d$ with $d \geq 2$ then clearly the singular value potential $\Phi^s$ corresponds to the case $\alpha_1=\cdots=\alpha_{\lfloor s\rfloor}=1$, $\alpha_{\lceil s\rceil}=s-\lfloor s\rfloor$, $\alpha_{\lceil s \rceil+1}=\cdots=\alpha_d=0$ of the above theorem. In particular Theorem \ref{th:main-tech} implies that if $(A_1,\ldots,A_N) \in \GL_d(\mathbb{R})^N$ is irreducible, $0<s<d$ and the singular value potential has an equilibrium state which is Bernoulli, then the linear maps $|\det A_1|^{-1/d}A_1,\ldots,  |\det A_N|^{-1/d}A_N$ are all contained in a compact subgroup of $\GL_d(\mathbb{R})$. 
As was indicated in the introduction, in combination with various more-or-less standard results from the literature, Theorem \ref{th:main-tech} is sufficient to prove Theorem \ref{th:main}. The derivation of Theorem \ref{th:main} from Theorem \ref{th:main-tech} is presented in the following section, and Theorem \ref{th:main-tech} itself is proved in sections \ref{se:rev} to \ref{se:gen-case}.

%
%

\section{Proof of Theorem \ref{th:main} conditional on Theorem \ref{th:main-tech}}\label{se:proof-of-main}

We begin the process of proving Theorem \ref{th:main} by collecting various results from the literature concerning the Lyapunov dimension, the affinity dimension, the natural projection from $\Sigma_N$ to the attractor, and self-affine measures.

\subsection{The Lyapunov and affinity dimensions}

The following result demonstrates that the affinity dimension has the properties alluded to in the introduction and introduces its counterpart for measures, the Lyapunov dimension:
\begin{lemma}\label{le:cty}
Let $A_1,\ldots,A_N \in \GL_d(\mathbb{R})$ with $\max_i \trip{A_i}<1$ for some norm $\trip{\cdot}$ on $\mathbb{R}^d$, and for each $s \geq 0$ define a potential $\Phi^s$ by $\Phi^s(\iii):=\varphi^s(A_\iii)$. Then:
\begin{enumerate}[(i)]
\item
The function $s\mapsto P(\Phi^s)=P(A_1,\ldots,A_N;s)$ is a continuous strictly decreasing function $[0,+\infty) \to \mathbb{R}$ with a unique zero, and this zero is strictly positive.
\item
For every $\mu \in \mathcal{M}_\sigma$ the function $s\mapsto h(\mu)+\Lambda(\Phi^s,\mu)$ is a continuous strictly decreasing function $[0,+\infty) \to \mathbb{R}$ with a unique zero.
\end{enumerate}
We define the \emph{affinity dimension} of $(A_1,\ldots,A_N)$ to be the unique zero of $s\mapsto P(\Phi^s)$, and the \emph{Lyapunov dimension} of $\mu \in \mathcal{M}_\sigma$ relative to $(A_1,\ldots,A_N)$, denoted $\dimlyap (\mu;A_1,\ldots,A_N)$,  to be the unique zero of  $s\mapsto h(\mu)+\Lambda(\Phi^s,\mu)$. 
\end{lemma}
The proof of the above lemma is a straightforward application of the inequalities
\[\varphi^{s_1}(A_\iii) \leq \left(C\trip{A_\iii}\right)^{s_1-s_2} \varphi^{s_2}(A_\iii) \leq C^{s_1-s_2}\left(\max_i \trip{A_i}\right)^{(s_1-s_2)|\iii|} \varphi^{s_2}(A_\iii)\]
and
\[\left(\min_i \sigma_d(A_i)\right)^{(s_1-s_2)|\iii|} \varphi^{s_2}(A_\iii) \leq \sigma_d(A_\iii)^{s_1-s_2}\varphi^{s_2}(A_\iii)\leq\varphi^{s_1}(A_\iii) \]
which are valid for all $\iii \in \Sigma_N^*$ and $s_1 \geq s_2 \geq 0$, where the constant $C>0$ depends only on $\trip{\cdot}$ and not on $\iii$, $s_1$ or $s_2$. The following relationship between Lyapunov dimension and affinity dimension was observed by A. K\"aenm\"aki \cite{Ka04}:
\begin{lemma}\label{le:high}
Let $A_1,\ldots,A_N \in \GL_d(\mathbb{R})$ with $\max_i \trip{A_i}<1$ for some norm $\trip{\cdot}$ on $\mathbb{R}^d$, and let $\mu \in \mathcal{M}_\sigma(\Sigma_N)$. Then $\dimlyap (\mu;A_1,\ldots,A_N) \leq \dimaff (A_1,\ldots,A_N)$, and equality holds if and only if $\mu$ is an equilibrium state of the potential $\Phi^s(\iii):=\varphi^s(A_\iii)$ where $s:=\dimaff(A_1,\ldots,A_N)$. 
\end{lemma}
\begin{proof}
For each $s\geq 0$  we have $h(\mu)+\Lambda(\Phi^s,\mu) \leq P(\Phi^s)$ by the variational principle, Proposition \ref{pr:varp}. In particular if $P(\Phi^s)<0$ for some $s>0$ then $h(\mu)+\Lambda(\Phi^s,\mu)<0$. It follows that
\[\left\{s \geq 0 \colon P(\Phi^s)< 0\right\} \subseteq  \left\{s \geq 0 \colon h(\mu)+\Lambda(\Phi^s,\mu) < 0\right\}\]
and since using Lemma \ref{le:cty}
\[\dimlyap (\mu;A_1,\ldots,A_N) = \inf\left\{s \geq 0 \colon h(\mu)+\Lambda(\Phi^s,\mu) < 0\right\}\]
and
\[\dimaff (A_1,\ldots,A_N)= \inf \left\{s \geq 0 \colon P(\Phi^s) < 0\right\}\]
it follows that $\dimlyap (\mu;A_1,\ldots,A_N) \leq \dimaff (A_1,\ldots,A_N)$ as required. If these two quantities are equal to one another with common value $s_0$, say, then we must have $h(\mu)+\Lambda(\Phi^{s_0},\mu)=0$ and $P(\Phi^{s_0})=0$ by continuity in view of Lemma \ref{le:cty}, which implies that $\mu$ is an equilibrium state for the potential $\Phi^{s_0}$ as claimed. The converse is trivial.
\end{proof}

\subsection{The natural projection and the dimension of self-affine measures}
If $T_1,\ldots,T_N$ are affine transformations of $\mathbb{R}^d$ which are contractions with respect to some norm $\trip{\cdot}$ on $\mathbb{R}^d$ then for every $v \in \mathbb{R}^d$ and $x =(x_k)_{k=1}^\infty \in \Sigma_N$ the limit
\[\Pi(x):=\lim_{n \to \infty} T_{x_1}T_{x_2}\cdots T_{x_n}v\]
exists and is independent of the choice of $v \in \mathbb{R}^d$. Indeed, if $\varepsilon>0$ is chosen such that $\trip{T_iu-T_iv} \leq (1-\varepsilon)\trip{u-v}$ for all $u,v \in \mathbb{R}^d$, and $v_0 \in \mathbb{R}^d$ is arbitary, then for every $r\geq\varepsilon^{-1} \max_i \trip{v_0-T_iv_0}$ every map $T_i$ preserves and contracts $\overline{B_r(v_0)}$, the closed $r$-ball centred on $v_0$ with respect to the norm $\trip{\cdot}$. It follows easily that $\Pi(x)=\bigcap_{n=1}^\infty T_{x_1}\cdots T_{x_n}\overline{B_r(v_0)}$. We deduce also that the diameter of the set $\Pi([\iii])$ is bounded by a constant times $(1-\varepsilon)^{|\iii|}$ and it follows that $\Pi \colon \Sigma_N \to \mathbb{R}^d$ is continous. It is not difficult to see that $\Pi(\Sigma_N)$ is contained in the attractor of $(T_1,\ldots,T_N)$ since the initial point $v$ may be taken to be in the attractor. It is also not difficult to see that $\Pi(\Sigma_N)$ is precisely the attractor, although this fact will not be used. We call $\Pi$ the \emph{natural projection} associated to $(T_1,\ldots,T_N)$.

The following result relating Bernoulli measures to self-affine measures via the natural projection follows from a more general theorem of J. E. Hutchinson \cite[\S4]{Hu81}. Although Hutchinson's proof assumes the probability vector $(p_1,\ldots,p_N)$ to be nondegenerate, it is not difficult to check that this stipulation is unnecessary.
\begin{lemma}\label{le:hutch}
Let $T_1,\ldots,T_N \colon \mathbb{R}^d \to \mathbb{R}^d$ be affine transformations which are contractions with respect to some norm on $\mathbb{R}^d$, and let $(p_1,\ldots,p_N)$ be a probability vector. Then a Borel probability measure $m$ on $\mathbb{R}^d$ satisfies $\sum_{i=1}^Np_i (T_i)_*m = m$ if and only if it satisfies $m=\Pi_*\mu$ where $\mu$ is the Bernoulli measure on $\Sigma_N$ characterised by the property $\mu([\iii])=p_{i_1}\cdots p_{i_n}$ for all $\iii=(i_k)_{k=1}^n \in \Sigma_N^*$. 
\end{lemma}

Finally, the following result connects the Lyapunov dimension with the dimension of a measure:
\begin{lemma}\label{le:rossi-joposi}
Let $T_1,\ldots,T_N \colon \mathbb{R}^d \to \mathbb{R}^d$ be affine transformations which are contractions with respect to some norm on $\mathbb{R}^d$ and let $\mu \in \mathcal{M}_\sigma$. Write $T_ix=A_ix+v_i$ for all $x\in\mathbb{R}^d$ and $i=1,\ldots,N$. Then $\dim \Pi_*\mu \leq \dimlyap (\mu;A_1,\ldots,A_N)$.
\end{lemma}
\begin{proof}
It is shown in \cite[Theorem 2.2]{Ro14} in the more general context of a countably infinite family of transformations $(T_i)_{i=1}^\infty$ that
\[\limsup_{r \to 0} \frac{\log \Pi_*\mu(B_r(\Pi(y)))}{\log r}\leq \dimlyap(\mu;A_1,\ldots,A_N) \]
for $\mu$-a.e. $y \in \Sigma$, and this obviously implies
\[\limsup_{r \to 0} \frac{\log \Pi_*\mu(B_r(x))}{\log r} \leq \dimlyap(\mu;A_1,\ldots,A_N)\]
for $\Pi_*\mu$-a.e. $x \in \mathbb{R}^d$, which yields the result. The result may also be derived from the proof of \cite[Theorem 4.3]{JoPoSi07}.
\end{proof}

\subsection{Further continuity properties of the Lyapunov and affinity dimensions}

Let $\mathrm{Cont}(\GL_d(\mathbb{R})^N)$ denote the set of all tuples $(A_1,\ldots,A_N) \in \GL_d(\mathbb{R})$ with the property that $\max_{1 \leq i \leq N} \trip{A_i}<1$ for some norm $\trip{\cdot}$ on $\mathbb{R}^d$ depending on $(A_1,\ldots,A_N)$. This is clearly an open subset of $\GL_d(\mathbb{R})^N$. The following two results will be key in proving the local uniformity of the dimension gap in Theorem \ref{th:main}:
\begin{proposition}\label{pr:gamma}
Define a function $\gamma \colon \mathrm{Cont}(\GL_d(\mathbb{R})^N) \to \mathbb{R}$  by
\[\gamma(B_1,\ldots,B_N):=\sup\left\{\dimlyap (\mu;B_1,\ldots,B_N)  \colon \mu\text{ is a Bernoulli  measure on }\Sigma_N\right\}.\]
Then $\gamma$ is upper semi-continuous, and additionally for every tuple $(B_1,\ldots,B_N) \in  \mathrm{Cont}(\GL_d(\mathbb{R})^N)$ the supremum in the definition of $\gamma$ is attained.
\end{proposition}
\begin{proof}
It is sufficient to prove the following statement: given a sequence of tuples $(A_1^{(n)},\ldots,A_N^{(n)}) \in  \mathrm{Cont}(\GL_d(\mathbb{R})^N)$ which converges to a limit $(A_1,\ldots,A_N) \in  \mathrm{Cont}(\GL_d(\mathbb{R})^N)$, there exists a Bernoulli measure $\mu$ on $\Sigma_N$ such that 
\begin{equation}\label{eq:gammagoal}\dimlyap(\mu;A_1,\ldots,A_N) \geq \limsup_{n \to \infty} \gamma(A_1^{(n)},\ldots,A_N^{(n)}).\end{equation}
Applying this result to a constant sequence of tuples $(A_1,\ldots,A_N)$ demonstrates that the supremum in the definition of $\gamma(A_1,\ldots,A_N)$ is attained; applying it to a nonconstant sequence directly implies that $\gamma$ is upper semi-continuous.

Let us prove this claim. For each $n \geq 1$ let $\mu_n$ be a Bernoulli measure such that
\[\dimlyap (\mu_n;A_1^{(n)},\ldots,A_N^{(n)}) > \gamma(A_1^{(n)},\ldots,A_N^{(n)})-\frac{1}{n}.\]
By passing to a subsequence if required, we may assume that the sequences of values $\gamma(A_1^{(n)},\ldots,A_N^{(n)})$ and $\dimlyap (\mu;A_1^{(n)},\ldots,A_N^{(n)})$ are convergent in $\mathbb{R}$, and similarly we may assume that $(\mu_n)$ converges to a limit $\mu$ in the weak-* topology. It is straightforward to verify that the set of Bernoulli measures on $\Sigma_N$ is closed in the weak-* topology and so the limit $\mu$ is necessarily Bernoulli. To prove  \eqref{eq:gammagoal} it is sufficient to prove that
\begin{equation}\label{eq:lyapgoal} \dimlyap(\mu;A_1,\ldots,A_N) \geq \lim_{n \to \infty}\dimlyap (\mu_n;A_1^{(n)},\ldots,A_N^{(n)}).\end{equation}
For each $n \geq 1$ and $s \geq 0$ define a potential $\Phi^{s,n} \colon \Sigma_N^* \to (0,+\infty)$ by $\Phi^{s,n}(\iii):=\varphi^s(A_\iii^{(n)})$, and define also $\Phi^s(\iii):=\varphi^s(A_\iii)$ for all $\iii \in \Sigma_N^*$. In the case where the limit $\lim_{n \to \infty} \dimlyap (\mu_n;A_1^{(n)},\ldots,A_N^{(n)})$ is zero the outcome \eqref{eq:lyapgoal} holds trivially, so we assume the limit to be strictly positive. In order to prove \eqref{eq:lyapgoal} it suffices to prove the following: for every positive real number $s<\lim_{n \to \infty} \dimlyap (\mu_n;A_1^{(n)},\ldots,A_N^{(n)})$ we have $h(\mu)+\Lambda(\Phi^s,\mu)\geq 0$.

Let us therefore fix $s<\lim_{n \to \infty} \dimlyap (\mu_n;A_1^{(n)},\ldots,A_N^{(n)})$. Let $n_0 \geq 1$ such that $\dimlyap(\mu_n;A_1^{(n)},\ldots,A_N^{(n)})>s$ for all $n \geq n_0$. For every $n \geq n_0$ we have $h(\mu_n)+\Lambda(\Phi^{s,n},\mu_n) \geq 0$ by the definition of the Lyapunov dimension. For each $n \geq 1$ we by definition have
\[h(\mu_n)=\inf_{m \geq 1} \frac{1}{m} \sum_{|\iii|=m} -\mu_n([\iii])\log\mu_n([\iii])=\lim_{m \to \infty} \frac{1}{m} \sum_{|\iii|=m} -\mu_n([\iii])\log\mu_n([\iii])\]
and
\[\Lambda(\Phi^{s,n},\mu_n) = \inf_{m \geq 1} \frac{1}{m}\sum_{|\iii|=m}\mu_n([\iii]) \Phi^{s,n}(\iii)=\lim_{m \to \infty} \frac{1}{m}\sum_{|\iii|=m}\mu_n([\iii]) \Phi^{s,n}(\iii),\]
so for each $n \geq n_0$ we have
\[\frac{1}{m}\sum_{|\iii|=m} -\mu_n([\iii])\log \mu_n([\iii]) +  \frac{1}{m}\sum_{|\iii|=m}\mu_n([\iii]) \Phi^{s,n}(\iii) \geq h(\mu_n)+\Lambda(\Phi^{s,n},\mu_n) \geq 0\]
for every $m \geq 1$. We have $\lim_{n \to \infty} \mu_n([\iii])=\mu([\iii])$ for every $\iii$ by weak-* convergence and $\lim_{n\to \infty} \Phi^{s,n}(\iii)=\Phi^s(\iii)$ for every $\iii$ by the $1$-Lipschitz continuity of the singular value functions $\sigma_k \colon \GL_d(\mathbb{R}) \to \mathbb{R}$. For fixed $m \geq 1$ it is thus clear that
\begin{eqnarray*}
{\lefteqn{\frac{1}{m}\sum_{|\iii|=m} -\mu([\iii])\log \mu([\iii]) +  \frac{1}{m}\sum_{|\iii|=m}\mu([\iii]) \Phi^{s}(\iii)}}& & \\
& & =\lim_{n \to \infty} \frac{1}{m}\sum_{|\iii|=m} -\mu_n([\iii])\log \mu_n([\iii]) +  \frac{1}{m}\sum_{|\iii|=m}\mu_n([\iii]) \Phi^{s,n}(\iii) \geq 0\end{eqnarray*}
and we deduce that
\[h(\mu)+\Lambda(\Phi^s,\mu) = \lim_{m \to \infty} \frac{1}{m}\sum_{|\iii|=m} -\mu([\iii])\log \mu([\iii]) +  \frac{1}{m}\sum_{|\iii|=m}\mu([\iii]) \Phi^{s}(\iii) \geq 0.\]
This demonstrates that $\dimlyap (\mu;A_1,\ldots,A_N) \geq s$ and the result follows.
\end{proof}
We also recall the following theorem of Feng and Shmerkin, which was originally proved in \cite{FeSh14} using thermodynamic formalism and the multiplicative ergodic theorem\footnote{The original result of Feng and Shmerkin works on the smaller space of tuples $(A_1,\ldots,A_N) \in \GL_d(\mathbb{R})^N$ such that $\max_i \|A_i\|<1$ for the Euclidean norm on $\mathbb{R}^d$. If we instead assume that $(A_1,\ldots,A_N) \in \mathrm{Cont}(\GL_d(\mathbb{R})^N)$ satisfies $\max_i \trip{A_i}<1$ for some norm $\trip{\cdot}$ on $\mathbb{R}^d$, then for some integer $n \geq 1$ and all $(B_1,\ldots,B_N)$ in a small neighbourhood of $(A_1,\ldots,A_N)$, the $N^n$-tuple $(B_1^n, B_1^{n-1}B_2,\ldots, B_N^{n-1}B_{N-1},B_N^n) \in \GL_d(\mathbb{R})^{N^n}$ is contracting in the Euclidean norm on $\mathbb{R}^d$ and has affinity dimension equal to $\dimaff(B_1,\ldots,B_N)$ by elementary consideration of the definition of the pressure function. In particular Feng and Shmerkin's result may be applied to these $N^n$-tuples in order to deduce the continuity of the affinity dimension with respect to $(B_1,\ldots,B_N)$ in the small neighbourhood.}. An alternative proof using linear algebra was given in \cite{Mo16}.
\begin{theorem}\label{th:feng-shmerkin}
The function $\dimaff \colon \mathrm{Cont}(\GL_d(\mathbb{R})^N) \to [0,+\infty)$ is continuous. 
\end{theorem}
We also require the following algebraic lemma. Although it can be deduced from the structure theory of reductive groups, we provide a brief elementary proof.
\begin{lemma}\label{le:irred.bdd}
Let $\mathsf{A}$ be an irreducible subset of $\GL_d(\mathbb{R})$. Suppose that for every $A$ in the semigroup generated by $\mathsf{A}$, the eigenvalues of $A$ all have absolute value $|\det A|^{1/d}$. Then $\{|\det A|^{-1/d}A\colon A \in \mathsf{A}\}$ is contained in a compact subgroup of $\GL_d(\mathbb{R})$.
\end{lemma}
\begin{proof}
Consider the semigroup $\Gamma$ generated by the set $\{|\det A|^{-1/d}A \colon A\in \mathsf{A}\}$, which is clearly irreducible. We claim that $\Gamma$ is bounded. To see this consider the closed subsemigroup $\overline{\mathbb{R}.\Gamma}:=\overline{\{\beta A \colon A \in \Gamma\text{ and }\beta \in \mathbb{R}\}}$ of the algebra of linear endomorphisms of $\mathbb{R}^d$. It is clear that for every $A \in \overline{\mathbb{R}.\Gamma}$ the eigenvalues of $A$ are also all of absolute value $|\det A|^{1/d}$, so in particular every element of $\overline{\mathbb{R}.\Gamma}$ is either invertible or nilpotent. It is easily seen that $\overline{\mathbb{R}.\Gamma}$ admits a nonzero nilpotent element if and only if $\Gamma$ is unbounded, so to prove the claim we will show that the only nilpotent element of $\overline{\mathbb{R}.\Gamma}$ is zero.

For a contradiction let $r$ be the minimal rank of a nilpotent nonzero element of $\overline{\mathbb{R}.\Gamma}$ and note that $0<r<d$. Fix a nilpotent element $B$ with rank $r$. Since $\rank (B^2)<\rank B$ by nilpotency we have $\rank (B^2)=0$ by minimality of $r$ so that $B^2=0$. The equation $B^2=0$ implies that the image $B\mathbb{R}^d$ is a subspace of the kernel of $B$. Since $\Gamma$ is irreducible, the nonzero $\Gamma$-invariant subspace $\sspan \{ABv \colon v \in \mathbb{R}^d\text{ and }A\in \Gamma\}$ must equal $\mathbb{R}^d$, so in particular there exists $A \in \Gamma$ such that $AB\mathbb{R}^d \not\subset \ker B$. The linear map $AB \in \overline{\mathbb{R}.\Gamma}$ has kernel equal to $\ker B$ since $A$ is invertible, it has rank precisely $r$, and it is nilpotent since every element of $\overline{\mathbb{R}.\Gamma}$ which is not invertible is nilpotent. But we have $(AB)^2 \neq 0$ because the image of $AB$ is not a subset of $\ker B = \ker AB$. This implies that $0<\rank AB<r$ which contradicts the minimality of $r$. We conclude that $\overline{\mathbb{R}.\Gamma}$ contains no nonzero nilpotents and therefore $\Gamma$ must be bounded as claimed.
 
To complete the proof it is sufficient to observe that the closure $\overline{\Gamma}$ is a group. Clearly this closure is a compact subsemigroup of $\GL_d(\mathbb{R})$. To see that it is a group it suffices to show that every $A \in \overline{\Gamma}$ satisfies $A^{-1} \in \GL_d(\mathbb{R})$, which may be achieved as follows. Given $A \in \overline{\Gamma}$ choose $(n_k)_{k=1}^\infty$ such that $\lim_{k \to \infty} A^{n_k}$ exists and $n_{k+1}\geq 2+n_k$ for all $k \geq 1$; it is clear that $\lim_{k \to \infty} A^{n_{k+1}-n_k-1}=A^{-1} \in \overline{\Gamma}$ as required.
\end{proof}

The final ingredient which we require for the proof of Theorem \ref{th:main} is the following: 
\begin{proposition}\label{pr:openness}
The set of all $(A_1,\ldots,A_N) \in \GL_d(\mathbb{R})^N$ satisfying hypotheses (i)--(iv) of Theorem \ref{th:main} is open.
\end{proposition}
\begin{proof}

It is obvious that if $(A_1,\ldots,A_N)$ satisfies $\max_i \trip{A_i}<1$ for some norm $\trip{\cdot}$ on $\mathbb{R}^d$ then so does every tuple $(B_1,\ldots,B_N)$ sufficiently close to $(A_1,\ldots,A_N)$. Similarly the set of all $(A_1,\ldots,A_N)$ satisfying (i) such that $0<\dimaff(A_1,\ldots,A_N)<d$ is open as a consequence of Theorem \ref{th:feng-shmerkin}. 

We claim that the set of all irreducible tuples $(A_1,\ldots,A_N) \in \GL_d(\mathbb{R})^N$ is open. To see this we observe that $(A_1,\ldots,A_N)$ is \emph{not} irreducible if and only if there exist unit vectors $u,v \in \mathbb{R}^d$ such that $\langle A_\iii u,v\rangle=0$ for all $\iii \in \Sigma_N^*$. Indeed, if such vectors exist then $\mathrm{span}\{A_\iii u \colon \iii \in \Sigma_N^*\}$ is an invariant subspace for $A_1,\ldots,A_N$ which is clearly not the zero subspace and is clearly a proper subspace since it does not contain $v$. On the other hand if an invariant subpace $U$ exists for $A_1,\ldots,A_N$ then we may choose arbitrary unit vectors $u \in U$ and $v \in U^\perp$ and see that the preceding condition is satisfied. Now observe that if for each $n$ the tuple $(A_1^{(n)},\ldots,A_N^{(n)})$ and unit vectors $u_n$ and $v_n$ satisfy $\langle A_\iii^{(n)} u_n,v_n\rangle=0$ for all $\iii \in \Sigma_N^*$, and $(A_1,\ldots,A_N)=\lim_{n \to \infty} (A_1^{n)},\ldots,A_N^{(n)})$, then any accumulation point $(u,v)$ of the sequence $(u_n,v_n)$ satisfies $\langle A_\iii u,v\rangle=0$ for all $\iii \in \Sigma_N^*$. Thus the set of all tuples $(A_1,\ldots,A_N) \in \GL_d(\mathbb{R})$ which are \emph{not} irreducible is closed.

As was discussed immediately subsequent to the statement of Theorem \ref{th:main}, the tuple $(A_1,\ldots,A_N)$ satisfies (iv) if and only if the linear maps $|\det A_i|^{-1/d}A_i$ are all contained in some compact subgroup of $\GL_d(\mathbb{R})$. We claim that if $(A_1,\ldots,A_N) \in \GL_d(\mathbb{R})$ is irreducible, then the linear maps $|\det A_i|^{-1/d}A_i$ are all contained in a compact subgroup of $\GL_d(\mathbb{R})$ if and only if for every $\iii \in \Sigma_N^*$, every eigenvalue of $A_\iii$ has absolute value equal to $|\det A_\iii|^{1/d}$. Indeed, if the first statement holds then every product $A_\iii$ has the property that the sequence $(|\det A_\iii|^{-n/d}A_\iii^n)_{n \in \mathbb{Z}}$ is bounded. Applying Gelfand's formula as $n \to +\infty$ it follows that $\rho(|\det A_\iii|^{-1/d}A_\iii)= 1$ and applying Gelfand's formula as $n\to-\infty$ we obtain $\rho(|\det A_\iii^{-1}|^{-1/d}A_\iii^{-1})= 1$. (Here and throughout this article $\rho(B)$ denotes the spectral radius of the linear map $B$.) These two identities together imply that every eigenvalue of $|\det A_\iii|^{-1/d}A_\iii$ has modulus $1$ and the second statement follows. The converse implication is given by Lemma \ref{le:irred.bdd}. We conclude that for an irreducible tuple $(A_1,\ldots,A_N)$, (iv) is equivalent to the statement that for every $\iii \in \Sigma_N^*$, every eigenvalue of $A_\iii$ has absolute value equal to $|\det A_\iii|^{1/d}$.

To complete the proof of the proposition we observe that a tuple $(A_1,\ldots,A_N) \in \GL_d(\mathbb{R})^N$ satisfies both (iii) and (iv) if and only if it belongs to the set of irreducible tuples (which is open) and avoids the set of tuples with the property that for every $\iii \in \Sigma_N^*$, every eigenvalue of $A_\iii$ has absolute value equal to $|\det A_\iii|^{1/d}$. The latter set is obviously closed. The result follows.
\end{proof}

\subsection{Proof of Theorem \ref{th:main}}
It is now a straightforward task to prove the main theorem. Proposition \ref{pr:openness} shows that if $(A_1,\ldots,A_N)$ satisfies hypotheses (i)--(iv) of Theorem \ref{th:main}, and $\mathsf{K}$ is a sufficiently small compact neighbourhood of $(A_1,\ldots,A_N)$, then every element of $\mathsf{K}$ satisfies (i)--(iv).

Fix a  compact subset $\mathsf{K}$ of $\GL_d(\mathbb{R})^N$ such that every $(A_1,\ldots,A_N) \in \mathsf{K}$ satisfies hypotheses (i)--(iv) of Theorem \ref{th:main}. By Lemma \ref{le:high} we have $\gamma(A_1,\ldots,A_N)-\dimaff(A_1,\ldots,A_N) \leq 0$ for all $(A_1,\ldots,A_N) \in \mathsf{K}$, and by the combination of Proposition \ref{pr:gamma} and Theorem \ref{th:feng-shmerkin} the function $(A_1,\ldots,A_N) \mapsto \gamma(A_1,\ldots,A_N) - \dimaff(A_1,\ldots,A_N)$ is upper semi-continuous. In particular its supremum is attained somewhere on $\mathsf{K}$, and is non-positive.

Suppose first that this supremum is equal to some negative number $-\kappa<0$. If $(A_1,\ldots,A_N) \in \mathsf{K}$, and $T_1,\ldots,T_N \colon \mathbb{R}^d \to \mathbb{R}^d$ are affine maps for which there exist $v_1,\ldots,v_N \in \mathbb{R}^d$ such that $T_ix=A_ix+v_i$ for all $x \in \mathbb{R}^d$, and $m$ is a self-affine measure with respect to $T_1,\ldots,T_N$, then by Lemma \ref{le:hutch} we have $m=\Pi_*\mu$ for some Bernoulli measure $\mu$ on $\Sigma_N$. Using Lemma \ref{le:rossi-joposi} it follows that 
\begin{align*}\dim m = \dim \Pi_*\mu &\leq \dimlyap (\mu;A_1,\ldots,A_N)\\
& \leq \gamma(A_1,\ldots,A_N) \leq \dimaff(A_1,\ldots,A_N)-\kappa\end{align*}
and we have established the conclusion of Theorem \ref{th:main}. To prove Theorem \ref{th:main} it therefore suffices to show that the supremum
\[\sup \left\{\gamma(A_1,\ldots,A_N) -\dimaff(A_1,\ldots,A_N)\colon (A_1,\ldots,A_N) \in \mathsf{K}\right\}\]
cannot be zero. If this supremum is zero then by the upper semi-continuity of $\gamma$, the continuity of $\dimaff$ and the compactness of $\mathsf{K}$ it must be the case that $\gamma(A_1,\ldots,A_N)=\dimaff(A_1,\ldots,A_N)$ for some $(A_1,\ldots,A_N) \in \mathsf{K}$. By Proposition \ref{pr:gamma} we have $\dimlyap (\mu;A_1,\ldots,A_N)=\gamma(A_1,\ldots,A_N)=\dimaff(A_1,\ldots,A_N)$ for some Bernoulli measure $\mu$ on $\Sigma_N$. By Lemma \ref{le:high} this implies that $\mu$ is an equilibrium state of the potential $\Phi(\iii):=\varphi^s(A_\iii)$ where $s:=\dimaff(A_1,\ldots,A_N) \in (0,d)$. By Theorem \ref{th:main-tech} the linear maps $|\det A_i|^{-1/d}A_i$ are all contained in a compact subgroup of $\GL_d(\mathbb{R})$, but as discussed subsequently to the statement of Theorem \ref{th:main}, this contradicts (iv). The proof of Theorem \ref{th:main} is complete.

%
%

\section{Review of linear algebraic groups}\label{se:rev}

\subsection{Reductive linear algebraic groups}

Here we include a brief overview of some aspects of reductive linear algebraic groups that will be useful in the proofs of the main results. Our principal reason of interest for this class of groups is that they arise as the Zariski closures of semigroups in $\GL_d(\mathbb{R})$ that act irreducibly on $\mathbb{R}^d$ (see below). For a more detailed exposition of the theory of reductive linear algebraic groups, we refer the reader to \cite{bq.book,borel.book,borel-tits,chevalley,knapp}. 

\subsubsection{Definition and relation to irreducible semigroups}\label{subsub.reductive.irred}
A linear Lie subgroup $G$ of $\GL_d(\mathbb{R})$ is said to be reductive if it has no non-trivial normal subgroup consisting of unipotent matrices. A connected reductive linear real Lie group $G$ is also a linear algebraic group in the sense that it is the connected component of identity $\mathbb{G}(\mathbb{R})^o$ of the group of real points $\mathbb{G}(\mathbb{R})$ of a (reductive) linear algebraic group $\mathbb{G}$ defined over $\mathbb{R}$. The linear algebraic group $\mathbb{G}$ admits a faithful rational representation $\mathbb{G} \to \GL_d$. In particular it can be seen as the set of zeros of polynomials in $\mathbb{R}[x_{ij},\det x^{-1}]$, where $x_{ij}$'s are the entries in $\Mat(d,\mathbb{R})$. Consequently, we can speak of the Zariski topology on $G$: a subset of $G$ is said to be Zariski closed if it is the set of common zeros of a set of polynomial maps. This defines the Zariski topology; the notions of Zariski closure and Zariski density are defined in the obvious way. The usual Hausdorff (analytic) topology on $G$ is finer than the Zariski topology. In the sequel, we shall speak of a real reductive group to mean a reductive linear real Lie group with finitely many connected components, and unless otherwise specified, topological notions refer to the analytic topology.

We will often work with semigroups in $\GL_d(\mathbb{R})$. We recall the elementary fact that the Zariski closure of a semigroup $\Gamma$ in $G$ is a (Zariski-closed) group, call it $H$. In particular, the Zariski closure of the group generated by $\Gamma$ is also $H$.

Before proceeding further, let us clarify the aforementioned relationship between irreducible, or rather completely reducible, families and real reductive groups. Recall that a semigroup $\Gamma$ in $\GL_d(\mathbb{R})$ is said to act completely reducibly if $\mathbb{R}^d$ decomposes into a direct sum $V_1 \oplus \ldots \oplus V_k$ of $\Gamma$-invariant subspaces $V_i$, on which $\Gamma$ acts irreducibly. It is equivalent to require that every $\Gamma$-invariant subspace has a $\Gamma$-invariant complement. Clearly, if $\Gamma$ acts irreducibly on $\mathbb{R}^d$, then it acts completely reducibly.

The action on $\mathbb{R}^d$ of a real reductive group $G<\GL_d(\mathbb{R})$ is completely reducible (see \cite[Ch.4]{chevalley}). Conversely,  let $\Gamma$ be a semigroup of $\GL_d(\mathbb{R})$ that acts completely reducibly on $\mathbb{R}^d$. Let $G$ be the Zariski closure of $\Gamma$. We claim that $G$ is a real reductive group. Indeed, being algebraic, $G$ has finitely many connected components. If it is not real reductive, then it contains a non-trivial normal subgroup $N$ consisting of unipotent matrices. Let $V_1$ be a $G$-irreducible subspace on which $N$ acts non-trivially.  By a classical result of Kolchin, the subspace $V_0$ of fixed vectors of $N$ in $V_1$ is a non-trivial proper subspace of $V_1$. Since $N$ is normal in $G$, $V_0$ is invariant under $G$, contradicting irreducibility of the $G$-action on $V_1$.

\subsubsection{Cartan space and roots}\label{subsub.cartanspace}

Let $A<G$ be a maximal connected real split torus so that it is a closed
Lie subgroup of $G$ that is isomorphic to $(\mathbb{R}^\ast_+)^d$ for some $d \in \mathbb{N}$. The integer $d$ is called the (real) rank of $G$. Let $Z(G)$ denote the center of $G$. The integer $d_S:=d-\dim Z(G)$ is called the semisimple rank of $G$. The Lie algebra $\mathfrak{a}$ of $A$ writes as $\mathfrak{a}=\mathfrak{a}_Z \oplus \mathfrak{a}_S$, where $\mathfrak{a}_Z$ is the Lie algebra of $A \cap Z(G)$ and $\mathfrak{a}_S$ is the Lie algebra of $A \cap [G,G]$. Here $[G,G]$ denotes the closed commutator subgroup of $G$, which is a semisimple Lie group.

Let $\mathfrak{g}$ be the Lie algebra of $G$ and let $\Ad: G \to \GL(\mathfrak{g})$ be the adjoint representation of $G$. A non-trivial character $\alpha:A \to \mathbb{R}^\ast_+$ is said to be a root of $G$ if it is a weight of $A$ for the $\Ad$-representation, i.e. the subspace $\mathfrak{g}_\alpha:=\{v \in \mathfrak{g}\, | \, \Ad(a)v=\alpha(a)v \; \forall a \in A\}$ is non-trivial. Given a character $\alpha$ of $A$, we denote by $\overline{\alpha}$ the element of $\mathfrak{a}^\ast$ satisfying $\exp(\overline{\alpha}(x))=\alpha(\exp(x))$ for every $x \in \mathfrak{a}$.  The set of non-zero $\overline{\alpha}$'s appearing in this form from the $\Ad$-representation forms a root system that we denote by $\Sigma$. Let $\{\overline{\alpha}_1,\ldots,\overline{\alpha}_{d_S}\}$ be a choice of simple roots so that $\Sigma$ splits into a disjoint union of positive roots $\Sigma_+$ (those elements of $\Sigma$ that can be written as a non-negative integer linear combination of $\overline{\alpha}_i$'s) and negative roots $-\Sigma_+$. 

We denote by $\mathfrak{a}^+$ the choice of a Weyl chamber in $\mathfrak{a}$ corresponding to a choice of simple roots: $x \in \mathfrak{a}$ belongs to $\mathfrak{a}^+$ if and only if for every $\overline{\alpha} \in \Sigma_+$, $\overline{\alpha}(x) \geq 0$. It is a closed fundamental domain for the action of the Weyl group $N_G(A)/Z_G(A)$, where $N_G(A)$ is the normalizer of $A$ in $G$ and $Z_G(A)$ is the centralizer of $A$ in $G$. The Weyl chamber $\mathfrak{a}^+$ is the direct sum of a salient cone $\mathfrak{a}^+ \cap \mathfrak{a}_S$ and the subspace $\mathfrak{a}_Z$.

An example of a real reductive group is $G=\GL_d(\mathbb{R})$ itself. In this case, the maximal real split torus $A$ can be taken to be diagonal matrices with positive coefficients. Its Lie algebra $\mathfrak{a}$ is the commutative Lie algebra of $d \times d$ diagonal matrices. The rank of $G$ is equal to $d$. The commutator $[G,G]=\SL(d,\mathbb{R})$ so that $\mathfrak{a}_S$ is the diagonal matrices whose coefficients sum to $0$. In particular, the semisimple rank of $G$ is $d-1$. The (log) roots are the linear forms $\overline{\alpha}_{i,j}$ with $i \neq j \in \{1,\ldots,d\}$ such that $\alpha_{i,j}(a)=\frac{a_i}{a_j}$ where $a_i$'s are the diagonal entries of $a$. A base of simple roots is given by $\overline{\alpha}_{i,i+1}$. The corresponding choice of Weyl chamber $\mathfrak{a}^+$ is the diagonal matrices with decreasing coefficients. The Weyl group is isomorphic to the symmetric group $S_d$ acting on $A$ by permuting the diagonal coefficients.

\subsubsection{Cartan and Jordan projections}\label{subsub.Cartan.Jordan}
Let $G$ be a real reductive group and let $K$ be a maximal compact subgroup of $G$ whose Lie algebra is orthogonal to $\mathfrak{a}$ for the Killing form. The Cartan decomposition of $G$ says that we have $G=KAK$. Here, given an element $g \in G$, its factor in the Cartan decomposition corresponding to the group $A$ is, up to the action of the Weyl group, uniquely determined. In particular for each $g \in G$ there exists a unique element $a_g \in A^+:=\exp(\mathfrak{a}^+)$ such that $g \in Ka_g K$. Accordingly we define the Cartan projection 
$$
\kappa:G \to \mathfrak{a}^+
$$ 
by setting $\kappa(g):=a_g$.

Every element $g \in G$ can also be decomposed as a commuting product $g=g_e g_h g_u$, where $g_e$ is an elliptic element (i.e.\ belonging to a compact group), $g_u$ is a unipotent element (i.e.\  $\Ad(g_u)$ is a unipotent linear transformation, where $\Ad:G \to \GL(\mathfrak{g})$ is the adjoint representation) and $g_h$ is a hyperbolic element (i.e.\ it is conjugate to an element of $A$). The hyperbolic part $g_h$ is uniquely determined and this allows us to define the Jordan projection
$$
\lambda:G \to \mathfrak{a}^+
$$ 
setting $\lambda(g)$ to be the logarithm of the unique element of $A^+$ conjugate to $g_h$. 

When $G=\GL_d(\mathbb{R})$, with the above choice of $A$, the maximal compact group $K$ can be taken to be the orthogonal group $O(d,\mathbb{R})$ and the Cartan decomposition is the polar decomposition: for $g \in \GL_d(\mathbb{R})$ its Cartan projection reads $\kappa(g)=(\log \sigma_1(g),\ldots, \log \sigma_d(g))$. The factorisation $g=g_e g_h g_u$ corresponds to Jordan block form and the Jordan projection $\lambda(g)$ reads $\lambda(g)=(\log |\lambda_1(g)|,\ldots,|\lambda_d(g)|)$.

\subsubsection{Representations and highest weights}\label{subsub.rep}

Let $G$ be a connected real reductive group and let $A<G$ and $\Sigma$ be as above. Let $U$ be a maximal unipotent subgroup of $G$ normalised by $A$ and whose Lie algebra is generated by the root spaces $(\mathfrak{g}_\alpha)_{\overline{\alpha} \in \Sigma_+}$. Let $V$ be a finite dimensional real vector space and $(\pi,V)$ an algebraic representation of $G$. An (algebraic) character $\chi$ of $A$ is said to be a restricted weight of $G$ in $(\pi,V)$ if the vector space $V^\chi:=\{v \in V\, | \, \pi(a)v=\chi(a)v \; \forall a \in A\}$ is non-trivial. Such a weight $\chi$ is said to be a parabolic weight if it is a weight of $A$ in the space $V^U:=\{v\in V \, |\, Uv=v\}$. It is said to be a dominant weight if it belongs to the Weyl chamber $\mathfrak{a}^+$ after the identification of $\mathfrak{a}$ with $\mathfrak{a}^\ast$ with an inner product on $\mathfrak{a}$ extending the restriction of the Killing form on $\mathfrak{a}_S$ and for which $\mathfrak{a}_S$ and $\mathfrak{a}_Z$ are orthogonal.

The choice of positive roots induces a partial order on the set of characters of $A$: we let $\chi_1 \leq \chi_2$ if and only if $\chi_2-\chi_1$ is a non-negative linear combination of positive, or equivalently simple, roots. An irreducible algebraic representation $(\pi,V)$ of $G$ admits a unique parabolic weight that we shall denote $\chi_V$. This is also the largest weight for the order induced by the choice of $\mathfrak{a}^+$ and this dominant weight is called \textit{the highest weight}. 

We will use the following fact that serves as a bridge between the geometry of $G$ and its representations. For its proof, see e.g. \cite[Lemma 8.17]{bq.book}
\begin{lemma}\label{lemma.weight.vs.eigenvalue}
Let $G$ be a connected real reductive group, $(\pi,V)$ be an irreducible linear representation of $G$ and $\chi$ be the highest weight. Then, for every $g\in G$, we have 
$$
\log |\lambda_1(\pi(g))|=\overline{\chi}(\lambda(g)).
$$
\end{lemma}

\subsubsection{A density result of Benoist}
In his study of asymptotic properties of linear groups and their actions on homogeneous spaces, Benoist \cite{benoist.linear2} (see also \cite{benoist.proper.actions}) introduced a notion of limit cone of a semigroup: given a semigroup $\Gamma$ in a real reductive group $G$, this is the smallest closed cone in $\mathfrak{a}^+$ containing all Jordan projections $\lambda(\gamma)$ of elements $\gamma \in \Gamma$. He proved in particular that the intersection of an affine translate of this cone with $\mathfrak{a}_S$ has non-empty interior in $\mathfrak{a}_S$ whenever $\Gamma$ is Zariski dense in $G$. The following density result of Benoist \cite{benoist.linear2}, later proven in a more elementary fashion by Quint \cite{quint.schottky}, is a refinement of the aforementioned property of this limit cone. In the proof of our main result, it will be instrumental in deducing the compactness of the image of $[G,G]$ under certain linear representations. 

We state a version of this result that is adapted to our purposes (see \cite[Proposition 9.8]{bq.book}): 
\begin{theorem}[\cite{benoist.linear2,quint.schottky,bq.book}]\label{thm.benoist.density}
Let $G$ be a connected real reductive group and $\Gamma<G$ a Zariski dense semigroup. The closed subgroup of $\mathfrak{a}$ spanned by the elements $\lambda(\gamma_1 \gamma_2)-\lambda(\gamma_1)-\lambda(\gamma_2)$ for $\gamma_1, \gamma_2 \in \Gamma$ is $\mathfrak{a}_S$.
\end{theorem}
We remark that in the work of Quint \cite{quint.schottky}, Benoist's non-arithmeticity result was also applied in a symbolic-dynamical context.

%
%

\section{The case of irreducible representations}\label{se:irr-case}

\subsection{Overview}
We may now commence working in earnest on the proof of Theorem \ref{th:main-tech}. We will study the potential $\Phi(\iii):=\prod_{i=1}^d \sigma_i(A_\iii)^{\alpha_i}$  by rewriting it in the form $\Phi(\iii)=\prod_{j=1}^d \|A_\iii^{\wedge j}\|^{\alpha_j-\alpha_{j+1}}$, where $\alpha_{d+1}:=0$. Since by hypothesis the semigroup $\Gamma:=\{A_\iii \colon \iii \in \Sigma_N^*\}$ acts irreducibly on $\mathbb{R}^d$, it follows from the discussion at the beginning of \S\ref{se:rev} that the Zariski closure of $\Gamma$ in $\GL_d(\mathbb{R})$ is a real reductive group $G$. We are thus in the following situation: we have a finite set of elements $g_1,\ldots,g_N$ of a real reductive group $G$ which generate a Zariski dense subsemigroup of $G$, a finite collection of representations $\pi_j$ from $G$ to $\GL(\wedge^j\mathbb{R}^d)$, a collection of non-negative real numbers $\beta_j$, and a potential $\Phi$ of the form $\Phi(\iii)=\prod_j \|\pi_j(g_\iii)\|^{\beta_j}$, where $g_\iii:=g_{i_1}\cdots g_{i_n}$ for $\iii=(i_t)_{t=1}^n$. (Since those indices $j$ for which $\beta_j=0$ have no effect on the value of $\Phi(\iii)$, we discard those indices. The condition $\alpha_1>\alpha_d$ implies that at least one $j<d$ is retained.) We wish to show that if $\Phi$ has an equilibrium state which is a Bernoulli measure, then $G$ must be a group of similitudes. Equivalently, we wish to show that the group $\{|\det g|^{-1/d}g \colon g \in G\}$ must be compact.

In the full generality of Theorem \ref{th:main-tech} we have no reason to believe that the representations $\pi_j$ are irreducible, which significantly complicates the argument. These representations are however completely reducible as a consequence of the reductiveness of the group $G$. We will therefore  first prove a version of Theorem \ref{th:main-tech} in the case of irreducible representations $\pi_j$, and then obtain the theorem in the general case by presenting the problem as a family of sub-cases each of which corresponds to a choice of a family of irreducible subspaces, one from each exterior power. The latter task is deferred to the following section. The objective of the present section will therefore be to prove the following:
\begin{theorem}\label{th:irreducible-case}
Let $G$ be a real reductive group. Given a positive integer $k$ and for each $j=1,\ldots,k$ a real inner product space $V_j$ of dimension $d_j \geq 1$,  let $\pi_j \colon G \to  \GL(V_j)$ be an irreducible linear representation. Let $g_1,\ldots,g_N \in G$ and write $g_\iii:=g_{i_1}\cdots g_{i_n}$ for all $\iii =(i_t)_{t=1}^n \in \Sigma_N^*$. Given constants $\beta_j>0$, define a potential $\Phi \colon \Sigma_N^* \to (0,+\infty)$ by
\[\Phi(\iii):=\prod_{j=1}^k \left\|\pi_j(g_\iii)\right\|^{\beta_j}.\]
Suppose that the semigroup generated by $g_1,\ldots,g_N$ is Zariski dense in $G$. Then the following are equivalent:
\begin{enumerate}[(i)]
\item
There exists an equilibrium state of $\Phi$ which is a Bernoulli measure.
\item
The potential $\Phi^{\det} \colon \Sigma_N^* \to (0,+\infty)$ defined by
\[\Phi^{\det}(\iii):= \prod_{j=1}^k \left|\det \pi_j(g_\iii) \right|^{\frac{\beta_j}{d_j}}\]
satisfies $P(\Phi)=P(\Phi^{\det})$.
\item
For every $j=1,\ldots,k$ the group
\[\left\{|\det \pi_j(g)|^{-\frac{1}{d_j}} \pi_j(g) \colon g \in G\right\}\]
is a compact subgroup of $\GL(V_j)$.
\end{enumerate}
\end{theorem}

The proof of the implications (iii)$\implies$(ii)$\implies$(i) is straightforward and almost all of the length of the proof of Theorem \ref{th:irreducible-case} arises from the implication (i)$\implies$(iii). As was described briefly in \S\ref{se:tech-thm} this proof itself consists of two somewhat separate parts, which we describe below.

\subsubsection{Comments on the proof}\label{subsub.comments}  The representations $\pi_j$ are irreducible but will not in general be strongly irreducible, so in general there exists for each $j$ a finite collection $U_j^1,\ldots,U_j^{n_j}$ of subspaces of $V_j$ which is permuted by the action of $G$ under the representation $\pi_j$. (If $\pi_j$ is strongly irreducible then we have $n_j=1$ and $U_j^1=V_j$.) We choose these subspaces to be of the least possible dimension and it is not difficult to deduce that they must have pairwise trivial intersection. Each $U_j^i$ is preserved by every element of the (Zariski) identity component\footnote{For the purposes of this description of the proof it makes no difference whether $G^o$ is taken to be the connected component with respect to the analytic topology or with respect to the Zariski topology. However, for technical reasons which will apply later, we define $G^o$ to be the group of real points of the Zariski connected component of $G$.} $G^o$, and in the first part of the proof we consider the action of $G^o$ on each $U_j^i$ via the restriction of $\pi_j$ to a representation $G^o \to \GL(U_j^i)$. By minimality of the dimension of $U_j^i$ this action is irreducible. Using the fact that that there exists a $\Phi$-equilibrium state which is a Bernoulli measure, a mechanism introduced in \cite{BoMo18} for writing $\Phi$ as the pointwise maximum of a finite collection of quasi-multiplicative potentials $\Phi^{\mathcal{W}}$, Proposition \ref{pr:qm-unique}, and Theorem \ref{thm.benoist.density}, we establish using the ideas outlined in the introduction that for each $j$ and $i$ the group $\pi_j(G^o)|_{U_j^i}$ is a group of linear similarity transformations of $U_j^i$ with respect to some inner product on $U_j^i$.

At this point we will have established that for each $j$, the elements of $\pi_j(G^o)$ can be simultaneously block diagonalised (using a splitting of the form $V_j=U_j^{i_1}\oplus \cdots \oplus U_j^{i_r}$) with each diagonal block equal to an orthogonal matrix times a positive real scalar. (This construction can be interpreted by saying that the elements of $\pi_j(G^o)$ are all normal matrices with respect to some consistent inner product structure on $V_j$.) In order to verify that $\pi_j(G^o)$ has the required property (iii) it remains to verify that for each fixed $g$ these scalars are the same for every block. In this part of the proof we must use not only the existence of a potential $\Phi^{\mathcal{W}_0}$ whose equilibrium state is a Bernoulli measure, but the fact the pressure $P(\Phi^{\mathcal{W}_0})$ is equal to the pressure $P(\Phi)$ of the original potential $\Phi$, or equivalently, the fact that $P(\Phi^{\mathcal{W}_0})$ is maximal among all of the pressures $P(\Phi^{\mathcal{W}})$. The underlying intuitive idea is that the products $\pi_j(g_\iii)$ necessarily have non-separated Lyapunov exponents with respect to the Bernoulli measure; this will be shown to imply that these products also have non-separated Lyapunov exponents with respect to the equilibrium measures of the other potentials $\Phi^{\mathcal{W}}$, since if this were not the case those equilibrium states would have a larger top Lyapunov exponent than is allowed by the variational principle. In practice this argument is implemented by comparing the values of various pressure functions associated to the different potentials $\Phi^{\mathcal{W}}$ (which are defined in terms of the growth rate of the norm of each representation and allow for separated Lyapunov exponents) and the potential $\Phi^{\det}$, which is defined in terms of the growth rates of determinants of representations (which does not perceive any difference between Lyapunov exponents). Once it has been shown that for each $g \in G^o$ the scalars associated to each diagonal block in the block diagonalisation of $\pi_j(g)$ are the same, it follows that $\pi_j(G^o)$ is contained in a group of linear similarity transformations of $\GL(V_j)$. The same result follows immediately for $\pi_j(G)$ since the remaining components of $\pi_j(G)$ form a finite collection of continuous images of $\pi_j(G^o)$. 

The respective functions of the two parts of the proof may be illustrated by considering two opposite extreme cases of the argument as follows. If it is known \emph{a priori} that each representation $\pi_j$ is strongly irreducible -- for example, if the group $G$ is known to be connected -- then we have $U_j^1=V_j$ for each $j$ and the first part of the proof establishes directly that each $\pi_j(G)$ is a group of linear similitudes as required. The proof is then complete without meaningful reference to the second part. If on the other hand it is known \emph{a priori} that for each $j$, there is a basis for $V_j$ with respect to which every $\pi_j(g_\iii)$ is represented by a generalised permutation matrix (that is, a matrix with exactly one nonzero entry in each row and in each column) then the subspaces $U_j^i$ are all one-dimensional, the action of $G^o$ on each subspace is trivially by a similitude since no other linear transformations of a one-dimensional space exist, and the first part of the proof is entirely redundant. In this case only the second part of the proof is required. 

\subsubsection{Remarks on a generalisation of Theorem \ref{th:irreducible-case}} 
Before starting the proof of Theorem \ref{th:irreducible-case}, we lastly remark that this theorem (and Theorem \ref{th:main-tech}) can easily be extended to the case of a linear Lie group $G$ which is not necessarily reductive. Indeed, using a reductivisation argument such as \cite[Proposition 6.8]{KaMo18} one may show that the equilibrium states of an affine iterated function system are determined only by the projections of the linear parts of the affinities to a  reductive Levi component (or in explicit co-ordinates, by the block diagonal parts of those linear maps when presented in block upper triangular form). This extended result does not lead to a significantly more powerful version of Theorem \ref{th:main} since in general it can easily occur that the equilibrium states are determined only by a proper subset of the diagonal blocks: the existence of a Bernoulli equilibrium state in the absence of irreducibility (but in the presence of complete reducibility) consequently can be used only to deduce that \emph{some} of the diagonal blocks of the affine transformations must consist of similitudes. Since this extended result requires few additional steps but lacks the clear interest of Theorem \ref{th:main} we leave it to the reader.


\subsection{Proof of the implications (iii)$\implies$(ii)$\implies$(i)}
The implication (iii)$\implies$(ii) is simple: if for each $j=1,\ldots,k$ the group
\[\{|\det \pi_j(g)|^{-1/d_j} \pi_j(g) \colon g \in G\}\]
is contained in a compact subset of $\GL(V_j)$, then we may find $K>0$ such that
\[K^{-1} |\det \pi_j(g)|^{1/d_j} \leq \left\|\pi_j(g)\right\| \leq K |\det \pi_j(g)|^{1/d_j} \]
for all $j=1,\ldots,k$ and all $g \in G$. It follows that for all $\iii \in \Sigma_N$ we have
\[K^{-\sum_{j=1}^k \beta_j} \Phi^{\det}(\iii) \leq \Phi(\iii) \leq K^{\sum_{j=1}^k \beta_j}\Phi^{\det}(\iii)\]
and we deduce that $P(\Phi)=P(\Phi^{\det})$ by direct reference to the definition of the pressure. This proves  (iii)$\implies$(ii). Let us now prove (ii)$\implies$(i). Assuming (ii), let $\mu$ be the Bernoulli measure on $\Sigma_N^*$ with probability vector $(p_1,\ldots,p_N)$ given by
\[p_{i_0}:= \frac{ \prod_{j=1}^k \left|\det \pi_j(g_{i_0}) \right|^{\frac{\beta_j}{d_j}}}{\sum_{i=1}^N \prod_{j=1}^k \left|\det \pi_j(g_i) \right|^{\frac{\beta_j}{d_j}}}\]
for every $i_0=1,\ldots,N$. Since
\begin{align*}P(\Phi^{\det}) &= \lim_{n \to \infty} \frac{1}{n}\log \sum_{|\iii|=n} \Phi^{\det}(\iii) = \lim_{n \to \infty} \frac{1}{n}\log \sum_{|\iii|=n}  \prod_{j=1}^k \left|\det \pi_j(g_\iii) \right|^{\frac{\beta_j}{d_j}}\\
&=\log \sum_{i=1}^N \prod_{j=1}^k \left|\det \pi_j(g_i) \right|^{\frac{\beta_j}{d_j}}\end{align*}
using the multiplicativity of the determinant, we observe that 
\[\mu([\iii])=\frac{ \prod_{j=1}^k \left|\det \pi_j(g_\iii ) \right|^{\frac{\beta_j}{d_j}}}{\left(\sum_{i=1}^N  \prod_{j=1}^k \left|\det \pi_j(g_i) \right|^{\frac{\beta_j}{d_j}}\right)^{|\iii|}}=\frac{ \Phi^{\det}(\iii)}{e^{|\iii|P(\Phi^{\det})}}\]
for every $\iii \in \Sigma_N^*$. Now, for each $n \geq 1$ we have
\begin{eqnarray*}
{\lefteqn{\sum_{|\iii|=n} -\mu([\iii])\log \mu([\iii]) +\sum_{|\iii|=n} \mu([\iii])\log \Phi^{\det}(\iii)}}& &  \\
& & = \sum_{|\iii|=n} \mu([\iii]) \left(nP(\Phi^{\det}) - \log \Phi^{\det}(\iii)+\log \Phi^{\det}(\iii)\right)\\
& & = nP(\Phi^{\det}) \sum_{|\iii|=n}\mu([\iii])= nP(\Phi^{\det}) \end{eqnarray*}
and since
\[h(\mu)= \lim_{n \to \infty} \frac{1}{n}\sum_{|\iii|=n} -\mu([\iii])\log \mu([\iii])\]
and
\[\Lambda\left(\Phi^{\det},\mu\right) = \lim_{n \to \infty}\frac{1}{n}\sum_{|\iii|=n} \mu([\iii])\log \Phi^{\det}(\iii)\]
we conclude that
\[h(\mu)+\Lambda(\Phi^{\det},\mu)    = P\left(\Phi^{\det}\right).\]
Now, clearly
\[\Phi^{\det}(\iii)= \prod_{j=1}^k \left|\det \pi_j(g_\iii) \right|^{\frac{\beta_j}{d_j}} \leq \prod_{j=1}^k \left\|\pi_j(g_\iii)\right\|^{\beta_j}=\Phi(\iii)\]
for every $\iii \in \Sigma_N^*$ using the elementary bound $|\det B| \leq \|B\|^{\dim V_j}$ valid for all $B \in \GL(V_j)$. It follows directly that $\Lambda(\Phi^{\det},\mu) \leq \Lambda(\Phi,\mu)$. We deduce that
\[P(\Phi)=P(\Phi^{\det}) = h(\mu)+\Lambda(\Phi^{\det},\mu) \leq h(\mu)+\Lambda(\Phi,\mu) \leq P(\Phi)\]
where we have used the hypothesis (ii) and, in the final inequality, the subadditive variational principle. 
It follows that $h(\mu)+\Lambda(\Phi,\mu)=P(\Phi)$ and thus the Bernoulli measure $\mu$ is an equilibrium state for the potential $\Phi$. This completes the proof of (ii)$\implies$(i).

\subsection{Proof of (i) $\implies$ (iii)}

\subsubsection{The family of subspaces with finite orbit}\label{subsub.family.of.subspaces}

For each $j=1,\ldots,k$ let $\ell_j \geq 1$ be the smallest possible dimension of a nonzero subspace of $V_j$ which is invariant under $\pi_j(g)$ for all $g \in G^o$, and choose $U_j \subseteq V_j$ to be such an $\ell_j$-dimensional subspace. It is not difficult to see that the function $g \mapsto \pi_j(g)U_j$ is constant on each connected component of $G$: if $g_1,g_2$ belong to the same component $G_i$ then $g_1^{-1}G_i$ is a connected component which contains the identity, hence is $G^o$, hence $g_1^{-1}g_2 \in G^o$, so $\pi_j(g_1^{-1}g_2)U_j=U_j$ and therefore $\pi_j(g_1)U_j=\pi_j(g_2)U_j$. For fixed $j=1,\ldots,k$, let $U_j^1,\ldots,U_j^{n_j}$ denote the complete list of subspaces of $V_j$ having the form $\pi_j(g)U_j$ for some $g \in G$. 

Fix $j \in \{1,\ldots,k\}$. We observe that $\mathrm{span} \bigcup_{i=1}^{n_j} U_j^i$ is a nonzero subspace of $V_j$ which is preserved by $\pi_j(g)$ for every $g \in G$, since each $\pi_j(g)$ acts on the spaces $U_j^i$ by permutation. By irreducibility it follows that this subspace must equal the whole of $V_j$. We now make the following claim: if $i_1,\ldots,i_{t+1}$ are distinct integers in the range $1$ to $n_j$, where $t \geq 1$, then $U_j^{i_{t+1}}$ either is a subspace of the vector space $\mathrm{span} \bigcup_{s=1}^t U_j^{i_s}$ or has trivial intersection with it. Indeed, if neither of these statements is true then 
$0<\dim  U_j^{i_{t+1}}\cap \left(\mathrm{span} \bigcup_{s=1}^t U_j^{i_s}\right)<\dim U_j^{i_{t+1}}=\ell_j$, in which case $U_j^{i_{t+1}}\cap \left(\mathrm{span} \bigcup_{s=1}^t U_j^{i_s}\right)$ is a subspace of $V_j$ which is fixed by $\pi_j(g)$ for all $g \in G^o$ but has dimension strictly less than $\ell_j$, contradicting minimality, and we deduce the truth of the claim. Now let $r_j$ be the largest integer such that we can find distinct integers $i_1,\ldots,i_{r_j}$ for which the spaces $U_j^{i_1},\ldots,U_j^{i_{r_j}}$ form a direct sum. (We observe that $r_j$ is at least $1$ and at most $n_j$, hence is well-defined.) If $U_j^{i_1}\oplus \cdots \oplus U_j^{i_{r_j}}$ is not equal to $V_j$ then by the preceding observation there must be some subspace $U_j^t$ which is not contained in it, hence has trivial intersection with it, allowing us to extend the direct sum, which is a contradiction. We therefore have $V_j=U_j^{i_1}\oplus \cdots \oplus U_j^{i_{r_j}}$ and in particular $r_j\ell_j=d_j$. 

We now claim there exists $C_1>0$ such that
\begin{equation}\label{eq:kappa}\prod_{j=1}^k \left\|\pi_j(g)\right\|^{\beta_j}  \leq C_1\prod_{j=1}^k \max_{1 \leq i \leq n_j} \left\|\pi_j(g)|_{U_j^i}\right\|^{\beta_j} \end{equation}
 for all $g \in G$. It is clearly sufficient to show that for each $j$ there exists $\tau_j>0$ such that $\max_{1 \leq i \leq n_j} \|B|_{U_j^i}\| \geq \tau_j\|B\|$ for every linear map $B \colon V_j \to V_j$, since then we may take $C_1:=\prod_{j=1}^k \tau_j^{-\beta_j}$. By homogeneity it is clearly sufficient to restrict to the case where $\|B\|=1$. If we can show that $\max_{1 \leq i \leq n_j} \|B|_{U_j^i}\|>0$ for every $B \in \mathrm{End}(V_j)$ with $\|B\|=1$ then the existence of $\tau_j$ follows by the compactness of the unit sphere of $\mathrm{End}(V_j)$. But if this inequality fails for some $B \in \mathrm{End}(V_j)$ with $\|B\|=1$ then we have found a nonzero linear map from $V_j$ to itself which is zero on every $U_j^i$, and this is impossible since the spaces $U_j^i$ together span $V_j$. The claim is proved.

\subsubsection{Transitivity classes and the construction of quasi-multiplicative potentials}\label{subsub.transitivity.classes}

Let $\mathfrak{W}$ denote the set of all $k$-tuples $(U_j^{i_j})_{j=1}^k$ such that $1 \leq i_j \leq n_j$ for all $j=1,\ldots,k$. We observe that $G$ acts on $\mathfrak{W}$ by taking the pair $(g, (U_j^{i_j})_{j=1}^k)$ to the tuple $(\pi_j(g)U_j^{i_j})_{j=1}^k$. Since the value of $(\pi_j(g)U_j^{i_j})_{j=1}^k$ depends only on the connected component of $G$ to which $g$ belongs, the $G$-action on $\mathfrak{W}$ factors through $G^o$ and yields an action of the finite group $G/G^o$ on $\mathfrak{W}$. Let us say that a \emph{transitivity class} is a subset of $\mathfrak{W}$ which corresponds to the orbit of a single tuple $(U_j^{i_j})_{j=1}^k$, and denote the set of transitivity classes by $\mathscr{W}$. Obviously, the number of transitivity classes is finite. For every transitivity class $\mathcal{W} \in \mathscr{W}$ let us define a potential $\Phi^{\mathcal{W}} \colon \Sigma_N^* \to (0,+\infty)$ by
\begin{equation*}
\Phi^{\mathcal{W}}(\iii):=\max_{(W_j)_{j=1}^k \in \mathcal{W}} \prod_{j=1}^k \left\|\pi_j(g_\iii)|_{W_j}\right\|^{\beta_j}.
\end{equation*}
The inequality $\Phi^{\mathcal{W}}(\iii\jjj) \leq \Phi^{\mathcal{W}}(\iii)\Phi^{\mathcal{W}}(\jjj)$ follows easily from the definition. It is clear that for each $\iii \in \Sigma_N^*$
\[\Phi(\iii) =  \prod_{j=1}^k \left\|\pi_j(g_\iii)\right\|^{\beta_j} \leq  C_1\prod_{j=1}^k \max_{1 \leq i \leq n_j} \left\|\pi_j(g_\iii)|_{U_j^i}\right\|^{\beta_j} \leq C_1\prod_{j=1}^k \left\|\pi_j(g_\iii)\right\|^{\beta_j} = C_1\Phi(\iii)\] 
and also
\[\prod_{j=1}^k \max_{1 \leq i \leq n_j} \left\|\pi_j(g_\iii)|_{U_j^i}\right\|^{\beta_j}  = \max_{(U_j^{i_j})_{j=1}^k\in\mathfrak{W}} \prod_{j=1}^k  \left\|\pi_j(g_\iii)|_{U_j^{i_j}}\right\|^{\beta_j} =\max_{\mathcal{W} \in \mathscr{W}} \Phi^{\mathcal{W}}(\iii)\]
so that
\begin{equation}\label{eq:kappa-again}C_1^{-1}\Phi(\iii) \leq \max_{\mathcal{W} \in \mathscr{W}}\Phi^{\mathcal{W}}(\iii) \leq \Phi(\iii)\end{equation}
for all $\iii \in \Sigma_N^*$. We observe in particular that $P(\Phi^{\mathcal{W}}) \leq P(\Phi)$ for every transitivity class $\mathcal{W}$ by direct appeal to the definition of the pressure.

By \cite[Theorem 6]{BoMo18}\footnote{See also a predecessor of this result by Quint, based on the first property of ``produit g\'{e}n\'{e}rique'' in \cite[Proposition I.2]{quint.div}, cf.~Step 2 of proof of Theorem \ref{thm.cartan.state}.} there exist $\delta>0$ and a finite subset $F$ of the semigroup $\{g_\iii \colon \iii \in \Sigma_N^*\}$ such that for every $\iii,\jjj \in \Sigma_N^*$ we have
\[\max_{\kkk \in F}\Phi^{\mathcal{W}}(\iii \kkk \jjj) \geq \delta \Phi^{\mathcal{W}}(\iii)\Phi^{\mathcal{W}}(\jjj). \]
By Proposition \ref{pr:qm-unique} this implies that for each transitivity class $\mathcal{W}$ there exists a unique measure $\nu \in \mathcal{M}_\sigma$ which is an equilibrium state for $\Phi^{\mathcal{W}}$, and this measure satisfies the Gibbs inequality
\begin{equation*}
C^{-1}_2e^{-|\iii|P(\Phi^{\mathcal{W}})} \Phi^{\mathcal{W}}(\iii) \leq \nu([\iii]) \leq C_2e^{-|\iii|P(\Phi^{\mathcal{W}})} \Phi^{\mathcal{W}}(\iii)
\end{equation*}
for every $\iii \in \Sigma_N^*$, where $C_2>0$ does not depend on $\iii$. Since the number of transitivity classes is finite, we may choose $C_2$ to be independent of the choice of $\mathcal{W}$ also. We observe in particular that $\nu([\iii])$ is always nonzero.

By hypothesis there exists a Bernoulli measure $\mu \in \mathcal{M}_\sigma$ which is an equilibrium state for $\Phi$. Since $\mu$ is a Bernoulli measure it is ergodic, so by the subadditive ergodic theorem we have for $\mu$-a.e. $x \in \Sigma_N$
\[\lim_{n \to \infty} \frac{1}{n}\log \Phi^{\mathcal{W}}(x|_n)=\Lambda(\Phi^{\mathcal{W}},\mu)\]
for every transitivity class $\mathcal{W}$, and also
\[\lim_{n \to \infty} \frac{1}{n}\log \Phi(x|_n)=\Lambda(\Phi,\mu).\]
In particular for $\mu$-a.e. $x \in \Sigma_N$
\begin{equation}\label{eq.max.trans.class}
\begin{aligned}\Lambda(\Phi,\mu) &= \lim_{n \to \infty} \frac{1}{n}\log \Phi(x|_n) = \lim_{n \to \infty} \frac{1}{n}\log \max_{\mathcal{W}\in\mathscr{W}} \Phi^{\mathcal{W}}(x|_n)\\
&= \max_{\mathcal{W}\in\mathscr{W}} \lim_{n \to \infty} \frac{1}{n}\log  \Phi^{\mathcal{W}}(x|_n)=\max_{\mathcal{W}\in\mathscr{W}} \Lambda(\Phi^{\mathcal{W}},\mu)\end{aligned}
\end{equation}
where we have used \eqref{eq:kappa-again} in the second equation. Choose a transitivity class $\mathcal{W}_0$ which attains this maximum, which we fix for the remainder of the proof. We have
\[P(\Phi) =h(\mu)+\Lambda(\Phi,\mu) = h(\mu)+\Lambda(\Phi^{\mathcal{W}_0},\mu)\leq P(\Phi^{\mathcal{W}_0}) \leq P(\Phi) \]
using the variational principle and the inequality $P(\Phi^{\mathcal{W}}) \leq P(\Phi)$ established earlier. Since the first and last terms in this chain of inequalities are equal, the inequalities must be equations. It follows that $\mu$ is the unique equilibrium state of the potential $\Phi^{\mathcal{W}_0}$. 

\subsubsection{Investigation of the transitivity class $\mathcal{W}_0$}\label{sss:c3} ${}$ We now investigate the transitivity class $\mathcal{W}_0$ specified in the previous paragraph which attains the maximum in \eqref{eq.max.trans.class}. We claim that the fact that the potential $\Phi^{\mathcal{W}_0}$ has a Bernoulli measure as its equilibrium state implies an additional relationship between the tuples $(W_j)_{j=1}^k$ which constitute the transitivity class $\mathcal{W}_0$. Specifically we claim that there exists $C_3>0$ such that for all $\iii\in \Sigma_N^*$ such that $g_\iii \in G^o$,
\begin{equation}\label{eq:thank-you-reviewer-C}\Phi^{\mathcal{W}_0}(\iii) \leq C_3\min_{(W_j)_{j=1}^k \in \mathcal{W}_0}\prod_{j=1}^k \left\|\pi_j(g_{\iii})|_{W_j}\right\|^{\beta_j}.\end{equation}
Before beginning the proof of the claim we make the following observation. By the Gibbs inequality established previously, there exists $C_2>0$ such that for all $\iii \in \Sigma_N^*$,
\[C^{-1}_2 e^{|\iii|P(\Phi)}\mu([\iii])\leq \Phi^{\mathcal{W}_0}(\iii) \leq C_2e^{|\iii|P(\Phi)}\mu([\iii]).\]
If $\iii,\jjj \in \Sigma_N^*$ are arbitrary then we notice that $\mu([\iii\jjj])=\mu([\iii])\mu([\jjj])$ because $\mu$ is Bernoulli, and therefore
\begin{align}\label{eq:bernoulli-1}\Phi^{\mathcal{W}_0}(\iii\jjj) &\geq C^{-1}_2e^{|\iii\jjj|P(\Phi)}\mu([\iii\jjj])\\\nonumber
& = C^{-1}_2e^{|\iii|P(\Phi)}\mu([\iii])e^{|\jjj|P(\Phi)}\mu([\jjj])\geq C^{-3}_2\Phi^{\mathcal{W}_0}(\iii)\Phi^{\mathcal{W}_0}(\jjj).\end{align}
We will use this property to prove the claim.

Let $r$ be the number of (Zariski) connected components of $G$. Since the semigroup $\{g_\iii \colon \iii \in \Sigma_N^*\}$ is Zariski dense in $G$, we may choose $\jjj_1,\ldots,\jjj_r \in \Sigma_N^*$ such that every connected component of $G$ contains precisely one of the elements $g_{\jjj_r}, g_{\jjj_r\jjj_{r-1}},\ldots,g_{\jjj_1\cdots \jjj_r}$ and therefore the sequence $g_{\jjj_r}G^o, g_{\jjj_{r-1}\jjj_{r}}G^o,\ldots, g_{\jjj_1 \cdots \jjj_{r}}G^o$ lists the components of $G$. It follows that if $(W_j)_{j=1}^k \in \mathcal{W}_0$ is arbitrary, then $(\pi_j(g_{\jjj_i\cdots \jjj_r})W_j)_{j=1}^k$ lists all of the elements of $\mathcal{W}_0$ (possibly with repetitions) as $i$ runs through $1,\ldots,r$. 

Now let $\iii \in \Sigma_N^*$ be an arbitrary word such that $g_\iii \in G^o$, and let $(W_j')_{j=1}^k \in \mathcal{W}_0$ such that
\[\prod_{j=1}^k \left\|\pi_j(g_\iii)|_{W_j'}\right\|^{\beta_j}= \min_{(W_j)_{j=1}^k \in \mathcal{W}_0} \prod_{j=1}^k \left\|\pi_j(g_\iii)|_{W_j}\right\|^{\beta_j}.\]
Observe that by definition there exists $(W_j)_{j=1}^k \in \mathcal{W}_0$ such that
\[\Phi^{\mathcal{W}_0}(\iii\jjj_1\iii \jjj_2\iii \cdots \jjj_{r-1}\iii\jjj_r )= \prod_{j=1}^k  \left\|\pi_j(g_{\iii\jjj_1\iii \jjj_2\iii \cdots \jjj_{r-1}\iii\jjj_r})|_{W_j}\right\|^{\beta_j}. \]
Repeated application of \eqref{eq:bernoulli-1} yields
\begin{equation}
\label{eq:1st-est}\Phi^{\mathcal{W}_0}(\iii\jjj_1\iii \jjj_2\iii \cdots \jjj_{r-1}\iii\jjj_r )\geq C_2^{-3(2r-1)} \Phi^{\mathcal{W}_0}(\iii)^r \left(\prod_{t=1}^r \Phi^{\mathcal{W}_0}(\jjj_t)\right) \geq \tau \Phi^{\mathcal{W}_0}(\iii)^r,\end{equation}
say, where $\tau>0$ is independent of $\iii$. In the other direction we obtain
\begin{align*}\Phi^{\mathcal{W}_0}(\iii\jjj_1\iii \jjj_2\iii \cdots \jjj_{r-1}\iii\jjj_r )&= \prod_{j=1}^k \left\|\pi_j(g_{\iii\jjj_1\iii \jjj_2\iii \cdots \jjj_{r-1}\iii\jjj_r})|_{W_j}\right\|^{\beta_j} \\
&\leq \left(\prod_{t=1}^r  \prod_{j=1}^k \left\|\pi_j(g_{\iii \jjj_t})|_{\pi_j(g_{\iii \jjj_{t+1}\cdots \iii\jjj_r}) W_j}\right\|^{\beta_j}\right) \\
&= \left(\prod_{t=1}^r  \prod_{j=1}^k \left\|\pi_j(g_{\iii \jjj_t})|_{\pi_j(g_{\jjj_{t+1}\cdots \jjj_r}) W_j}\right\|^{\beta_j}\right)\end{align*}
where we have used the fact that $(\pi_j(g_\iii)W_j)_{j=1}^k = (W_j)_{j=1}^k$ for every $(W_j)_{j=1}^k \in \mathcal{W}_0$ since $g_\iii \in G^o$. This is clearly bounded by
\[\left(\prod_{t=1}^r\prod_{j=1}^k  \left\|\pi_j(g_{\jjj_t})|_{\pi_j(g_{\jjj_{t+1}\cdots \jjj_r}) W_j}\right\|^{\beta_j}\right)\left(\prod_{t=1}^r  \prod_{j=1}^k \left\|\pi_j(g_{\iii})|_{\pi_j(g_{\jjj_{t}\cdots \jjj_r}) W_j}\right\|^{\beta_j}\right)\]
and hence by 
\[K\left(\prod_{t=1}^r  \prod_{j=1}^k \left\|\pi_j(g_{\iii})|_{\pi_j(g_{\jjj_{t}\cdots \jjj_r}) W_j}\right\|^{\beta_j}\right) \]
where $K:=\prod_{t=1}^r \Phi^{\mathcal{W}_0}(\jjj_t)$, which clearly does not depend on $\iii$. Thus
\[\Phi^{\mathcal{W}_0}(\iii\jjj_1\iii \jjj_2\iii \cdots \jjj_{r-1}\iii\jjj_r )\leq K \prod_{t=1}^r  \prod_{j=1}^k \left\|\pi_j(g_{\iii})|_{\pi_j(g_{\jjj_{t}\cdots \jjj_r}) W_j}\right\|^{\beta_j}.\]
But this in turn  is clearly bounded by
\[K \left(\prod_{j=1}^k \left\|\pi_j(g_{\iii})|_{W_j'}\right\|^{\beta_j}\right)\left(\max_{(W_j'')_{j=1}^k \in \mathcal{W}_0} \prod_{j=1}^k \left\|\pi_j(g_{\iii})|_{W_j''}\right\|^{\beta_j}\right)^{r-1}\]
because as $t$ ranges from $1$ to $r$ the tuple $(\pi_j(g_{\jjj_t\cdots \jjj_r})W_j)_{j=1}^k$ ranges over all of the elements of $\mathcal{W}_0$ and in particular is equal to $(W_j')_{j=1}^k$ for at least one value of $t$. Thus
\begin{equation}\label{eq:2nd-est}\Phi^{\mathcal{W}_0}(\iii\jjj_1\iii \jjj_2\iii \cdots \jjj_{r-1}\iii\jjj_r )\leq K \left(\prod_{j=1}^k \left\|\pi_j(g_{\iii})|_{W_j'}\right\|^{\beta_j}\right)\Phi^{\mathcal{W}_0}(\iii)^{r-1}.\end{equation}
Combining \eqref{eq:1st-est} and \eqref{eq:2nd-est} yields
\[\tau \Phi^{\mathcal{W}_0}(\iii)^r \leq K \left(\prod_{j=1}^k \left\|\pi_j(g_{\iii})|_{W_j'}\right\|^{\beta_j}\right)\Phi^{\mathcal{W}_0}(\iii)^{r-1}\]
where $K,\tau>0$ do not depend on $\iii$, and dividing by $\tau\Phi^{\mathcal{W}_0}(\iii)^{r-1}$ proves the claim.

\subsubsection{A multiplicativity property on a dense subsemigroup of the identity component}\label{subsub.get.multiplicative}
We now claim that for every $\iii,\jjj \in \Sigma_N^*$ such that $g_\iii,g_\jjj\in G^o$ and every $(W_j)_{j=1}^k \in \mathcal{W}_0$, we have
\begin{equation}\label{eq:mult-identity}\prod_{j=1}^k \rho(\pi_j(g_\iii g_\jjj)|_{W_j})^{\beta_j} = \left(\prod_{j=1}^k \rho(\pi_j(g_\iii)|_{W_j})^{\beta_j} \right)\left(\prod_{j=1}^k \rho(\pi_j(g_\jjj)|_{W_j})^{\beta_j} \right)\end{equation}
where $\rho(B)$ denotes the spectral radius of the linear map $B$.
Fix words $\iii$ and $\jjj$ such that $g_\iii,g_\jjj\in G^o$, and fix $(W_j)_{j=1}^k \in \mathcal{W}_0$. We observe that $\pi_j(g_\iii)W_j=W_j$ and $\pi_j(g_\jjj)W_j=W_j$ for all $j=1,\ldots,k$. Using the fact that $\mu$ is a Bernoulli measure we have $\mu([(\iii\jjj)^n])=\mu([\iii])^n\mu([\jjj])^n=\mu([\iii^n])\mu([\jjj^n])$ for every $n \geq 1$, so by the Gibbs inequality
\begin{align*}\Phi^{\mathcal{W}_0}(\iii^n)\Phi^{\mathcal{W}_0}(\jjj^n)&\leq C^2_2 e^{n(|\iii|+|\jjj|)P(\Phi)} \mu([\iii^n])\mu([\jjj^n]) \\
&=C^2_2 e^{n(|\iii|+|\jjj|)P(\Phi)} \mu([(\iii\jjj)^n])\\
&\leq C^3_2 \Phi^{\mathcal{W}_0}((\iii\jjj)^n)\end{align*}
and similarly
\[ \Phi^{\mathcal{W}_0}((\iii\jjj)^n) \leq C^3_2\Phi^{\mathcal{W}_0}(\iii^n)\Phi^{\mathcal{W}_0}(\jjj^n).\]
We have
\[ \prod_{j=1}^k \left\|\pi_j(g_{\iii\jjj}^n|_{W_j})\right\|^{\beta_j} \leq \Phi^{\mathcal{W}_0}((\iii\jjj)^n)\]
by the definition of $\Phi^{\mathcal{W}_0}$, and since $g_{\iii\jjj}^n \in G^o$ we have
\[\Phi^{\mathcal{W}_0}((\iii\jjj)^n) \leq C_3\min_{(W_j')_{j=1}^k \in \mathcal{W}_0} \prod_{j=1}^k \left\|\pi_j(g_{\iii\jjj}^n|_{W_j'})\right\|^{\beta_j}\leq C_3\prod_{j=1}^k \left\|\pi_j(g_{\iii\jjj}^n|_{W_j})\right\|^{\beta_j}\]
by the previous claim. Likewise
\[ \prod_{j=1}^k \left\|\pi_j(g_{\iii}^n|_{W_j})\right\|^{\beta_j} \leq \Phi^{\mathcal{W}_0}(\iii^n) \leq C_3 \prod_{j=1}^k \left\|\pi_j(g_{\iii}^n|_{W_j})\right\|^{\beta_j}\]
and
\[ \prod_{j=1}^k \left\|\pi_j(g_{\jjj}^n|_{W_j})\right\|^{\beta_j} \leq \Phi^{\mathcal{W}_0}(\jjj^n) \leq C_3 \prod_{j=1}^k \left\|\pi_j(g_{\jjj}^n|_{W_j})\right\|^{\beta_j}.\]
Thus
\begin{align*}\left(\prod_{j=1}^k \left\|\pi_j(g_{\iii}^n|_{W_j})\right\|^{\beta_j}\right)\left(\prod_{j=1}^k \left\|\pi_j(g_{\jjj}^n|_{W_j})\right\|^{\beta_j}\right) &\leq \Phi^{\mathcal{W}_0}(\iii^n)\Phi^{\mathcal{W}_0}(\jjj^n)\\
& \leq  C^3_2 \Phi^{\mathcal{W}_0}((\iii\jjj)^n)\\
&\leq C^3_2 C_3\prod_{j=1}^k \left\|\pi_j(g_{\iii\jjj}^n|_{W_j})\right\|^{\beta_j} \end{align*}
and similarly
\begin{align*}\prod_{j=1}^k \left\|\pi_j(g_{\iii\jjj}^n|_{W_j})\right\|^{\beta_j} &\leq \Phi^{\mathcal{W}_0}((\iii\jjj)^n) \\&\leq C^3_2 \Phi^{\mathcal{W}_0}(\iii^n)\Phi^{\mathcal{W}_0}(\jjj^n)\\
&\leq C^3_2 C_3^2 \left(\prod_{j=1}^k \left\|\pi_j(g_{\iii}^n|_{W_j})\right\|^{\beta_j}\right)\left(\prod_{j=1}^k \left\|\pi_j(g_{\jjj}^n|_{W_j})\right\|^{\beta_j}\right).\end{align*}
We have obtained 
\[C^{-3}_2 C_3^{-1} \leq \frac{\prod_{j=1}^k \left\|\pi_j(g_{\iii\jjj}^n|_{W_j})\right\|^{\beta_j} }{ \left(\prod_{j=1}^k \left\|\pi_j(g_{\iii}^n|_{W_j})\right\|^{\beta_j}\right)\left(\prod_{j=1}^k \left\|\pi_j(g_{\jjj}^n|_{W_j})\right\|^{\beta_j}\right)} \leq C^3_2 C_3^2\]
for all $n \geq 1$. Taking the power $\frac{1}{n}$ and letting $n\to \infty$ we obtain by Gelfand's formula
\[\frac{\prod_{j=1}^k \rho(\pi_j(g_\iii g_\jjj)|_{W_j})^{\beta_j}}{ \left(\prod_{j=1}^k \rho(\pi_j(g_\iii)|_{W_j})^{\beta_j} \right)\left(\prod_{j=1}^k \rho(\pi_j(g_\jjj)|_{W_j})^{\beta_j} \right)}=1\]
for all $\iii,\jjj \in \Sigma_N^*$ such that $g_\iii,g_\jjj \in G^o$, and this is precisely \eqref{eq:mult-identity}.

\subsubsection{Application of the theorem of Benoist}\label{sss:ben}
We now apply the work of Benoist to show that the identity \eqref{eq:mult-identity} severely restricts the possible structures of the groups $\{\pi_j(g)|_{W_j} \colon g \in G^o\}$ for $(W_j)_{j=1}^k \in \mathcal{W}_0$. 
Fix an arbitrary tuple $(W_j)_{j=1}^k \in \mathcal{W}_0$ and define
\begin{equation}\label{eq.defn.xi}
\xi(g):=\prod_{j=1}^k \rho(\pi_j(g)|_{W_j})^{\beta_j}  
\end{equation}
for all $g \in G^o$. The identity \eqref{eq:mult-identity} asserts that $\xi(g_\iii g_\jjj)=\xi(g_\iii)\xi(g_\jjj)$ for all $\iii,\jjj \in \Sigma_N^*$ such that $g_\iii,g_\jjj \in G^o$. 

Recall that by construction (\S \ref{subsub.family.of.subspaces}), for each $j=1,\ldots,k$, the restriction of $\pi_j$ to the connected reductive group $G^o$ gives rise to an irreducible linear representation of $G^o$ on $W_j$. Denote this representation by $\hat{\pi}_j$. We will show that for each $j=1,\ldots,k$ the image $\hat{\pi}_j([G^o,G^o])$ is compact. If $[G^o,G^o]$ is itself compact then this result is trivial, so without loss of generality we assume that the semisimple group $[G^o,G^o]$ is non-compact. For each $j$ let $\hat{\chi}_j$ be the highest weight of $\hat{\pi}_j$ so that $\overline{\hat{\chi}}_j \in \mathfrak{a}^\ast$ where $\mathfrak{a}$ is a fixed Cartan subspace in the Lie algebra $\mathfrak{g}$ of $G$ (see \S \ref{subsub.cartanspace} and \S \ref{subsub.rep}). By Lemma \ref{lemma.weight.vs.eigenvalue}, $(\ref{eq.defn.xi})$ can be rewritten as
\begin{equation}\label{eq.xi.to.chi}
\log \xi(g)=\sum_{j=1}^k \beta_j \overline{\hat{\chi}}_j(\lambda(g))
\end{equation}
where $\lambda$ is the Jordan projection on a fixed Weyl chamber $\mathfrak{a}^+$ in $\mathfrak{a}$ (\S \ref{subsub.Cartan.Jordan}). 

Denote by $\Gamma$ the semigroup in $G$ generated by $\{g_1,\ldots,g_N\}$ and by $\Gamma_o$ the intersection $G^o \cap \Gamma$. Since by hypothesis $\Gamma$ is Zariski dense in $G$, the semigroup $\Gamma_o$ is Zariski dense in $G^o$. Setting $\bar{\chi}:=\sum_{j=1}^k \beta_j \overline{\hat{\chi}}_j$, in view of $ (\ref{eq.defn.xi})$ and $(\ref{eq.xi.to.chi})$, the equation $(\ref{eq:mult-identity})$ implies that the set 
\[
\{\lambda(\gamma_1\gamma_2)-\lambda(\gamma_1)-\lambda(\gamma_2)\colon\gamma_1, \gamma_2 \in \Gamma_o\}
\]
is contained in the subspace $\ker \overline{\chi}$. Since the latter is closed, by Theorem \ref{thm.benoist.density} we deduce that the semisimple part $\mathfrak{a}_S$ of the Cartan space $\mathfrak{a}$ is contained in $\ker \overline{\chi}$. Furthermore, since for each $j=1,\ldots,k$, $\overline{\hat{\chi}}_j$ is a dominant weight (in particular, it takes non-negative values on the cone $\mathfrak{a}_S \cap \mathfrak{a}^+$) and $\beta_j>0$, this implies that for each $j=1,\ldots,k$, we have $\mathfrak{a}_S \subseteq \ker \overline{\hat{\chi}}_j$. Hence by Lemma \ref{lemma.weight.vs.eigenvalue} the spectral radius of every element of $\hat\pi_j([G^o,G^o])$ is $1$. The determinant of every element of $\hat\pi_j([G^o,G^o])$ is also $1$ as a direct consequence of the definition of $[G^o,G^o]$ (as closure of a group generated by elements of type $ghg^{-1}h^{-1}$), so every element of $\hat\pi_j([G^o,G^o])$ has every eigenvalue equal to $1$ in modulus. Since $[G^o,G^o]$ is semisimple it acts completely reducibly on $W_j$, so by applying Lemma \ref{le:irred.bdd} to each subspace in a decomposition of $W_j$ into invariant subspaces on which $[G^o,G^o]$ acts irreducibly, it follows that $\hat\pi_j([G^o,G^o])$ is a compact subgroup of $\GL(W_j)$ as required.

On the other hand, since $\hat{\pi}_j$ is an irreducible representation (\S \ref{subsub.family.of.subspaces}), by Schur's lemma, $\R Z(\hat{\pi}_j(G^o)) \leq \End_{\R \hat{\pi}_j(G^o)}(W_j)$ is isomorphic to either $\R$ or $\C$ as a real division algebra. In the first case, $Z(\hat{\pi}_j(G^o))$ is contained in the group of homotheties $\simeq \R^\ast$ of $W_j$ and in the latter case it is contained in a copy of $\SO(2,\R) \times \R^\ast$ in $\GL(W_j)$. Finally we recall that the connected real reductive group $G^o$ is an almost direct product of its center $Z(G^o)$ and $[G^o,G^o]$ (\cite[Proposition 2.2]{borel-tits}), which is to say the map $Z(G^\circ) \times [G^o,G^o] \to G^o$ defined by $(z,g) \mapsto zg$ is surjective with finite kernel. We conclude that $\hat{\pi}_j(G^o)$ is contained in a compact subgroup of $\GL(W_j)$ modulo factoring out the absolute value of the determinant of each element, and therefore each of the groups $\hat{\pi}_j(G^o)$ is a group of linear similarity transformations of $W_j$ with respect to some Euclidean structure on $W_j$. 

Now recall that, for each $j=1,\ldots,k$, the finite group $G/G^o$ acts transitively on $\{U^i_j \colon i=1,\ldots,n_j\}$. Since for each $j=1,\ldots,k$ we have $W_j=U_j^i$ for some $i \in \{1,\ldots,n_j\}$, by transitivity of $G/G^o$, repeating the same argument above for every $(W_j)_{j=1}^k \in \mathcal{W}_0$, we deduce that up to conjugation in $\GL(V_j)$, $\pi_j(G^o)|_{U^i_j}$ is contained in the group of linear similarities of $U^i_j$ for every $i=1,\ldots,n_j$, for every $j=1,\ldots,k$. In particular, passing to matrix representation by  convenient choice of bases for $U^{i_\ell}_j$'s for $\ell=1,\ldots,r_j$ and $j=1,\ldots,k$, $\pi_j(G^o)$ is contained in the group of block diagonal matrices of the form
\begin{equation}\label{eq.matrix.form}
\begin{bmatrix}
\gamma_1 O_1 & 0 & \dots & 0 \\
0 & \gamma_2 O_2 & \dots & \vdots \\
\vdots &  & \ddots & 0\\
0 & \dots & 0 & \gamma_{r_j} O_{r_j}
\end{bmatrix}
\end{equation}
where the $\gamma_i$'s are scalars in $\mathbb{R}^\ast_+$ and $O_i$'s are $\ell_j \times \ell_j$ orthogonal matrices. We have completed the first of the two parts of the proof as described in \S \ref{subsub.comments}.

\subsubsection{The identity of the scalars.}

In the second part of the proof we wish to show that for every $g \in G^o$, in the matrix representation \eqref{eq.matrix.form} we have $\gamma_1=\cdots =\gamma_{r_j}$. Since obviously each $\gamma_i$ is equal to $|\det (\gamma_i O_i)|^{1/\ell_j}$, the idea is to show that for each $g \in G^o$ and $j=1,\ldots,k$ the quantity $|\det \pi_j(g)|_{U_j^i}|^{1/\ell_j}$ is independent of $i$. Since $V_j$ can be written as a direct sum of a sub-collection of spaces $U_j^{i_1},\ldots,U_j^{i_{r_j}}$, this in turn is clearly equivalent to the identity
\begin{equation}\label{eq:dets-same}\left|\det \left(\pi_j(g)|_{U_j^i}\right)\right|^{\frac{1}{\ell_j}} = |\det \pi_j(g)|^{\frac{1}{d_j}}\end{equation}
for every $i=1,\ldots,n_j$ and $j=1,\ldots,k$, which is what shall be shown. It will then be a straightforward matter to conclude the theorem. 

We therefore undertake to prove \eqref{eq:dets-same}. To establish this equality we must use the fact that $\Phi^{\mathcal{W}_0}$ has the greatest pressure of any $\Phi^{\mathcal{W}}$, which we did not previously substantially use. The key fact which we shall ultimately demonstrate is that there exists $C>0$ such that $C^{-1}\Phi^{\mathcal{W}_0}(\iii) \leq \Phi^{\mathcal{W}}(\iii) \leq C\Phi^{\mathcal{W}_0}(\iii)$ for every $\iii \in \Sigma_N^*$ such that $g_\iii \in G^o$, for every transitivity class $\mathcal{W}$.

\subsubsection{A first identity involving determinants}\label{sss:first-id}
Fix $j \in \{1,\ldots,k\}$. If we knew that the number $n_j$ of spaces $U_j^i$ was equal to exactly $d_j/\ell_j$ then we would have $V_j = \bigoplus_{i=1}^{n_j} U_j^i$ and the identity
\begin{equation}\label{eq:dets-ad-mortem}\left(\prod_{i=1}^{n_j}\left|\det \left(\pi_j(g)|_{U_j^i}\right)\right|^{\frac{1}{\ell_j}} \right)^{\frac{1}{n_j}} =|\det \pi_j(g)|^{\frac{1}{d_j}}\end{equation}
would be obvious. However, in general we do not necessarily have $n_j=d_j/\ell_j$. Our first task will be to show that the above identity remains true even when $n_j>d_j/\ell_j$ and the spaces $U_j^1,\ldots,U_j^{n_j}$ do not form a direct sum.  The proof of this equality is conducted by exploring the combinatorial relationships between the similarity ratios $\gamma_i(g):=|\det (\pi_j(g)|_{U_j^i})|^{1/\ell_j}$ and subspaces $U_j^i$. The fundamental task will be to show that the list of spaces $U_j^1,\ldots,U_j^{n_j}$ may be partitioned into equal-sized classes in such a way that every $g \in G^o$ has constant similarity ratio on each class, and such that the spans of the classes form a direct sum.

For $i=1,\ldots,n_j$ and $g \in G^o$, let $\gamma_i(g) :=|\det (\pi_j(g)|_{U_j^i})|^{1/\ell_j}\in \mathbb{R}^\ast_+$ denote the similarity ratio of $\pi_j(g)|_{U^{i}_j}$. Define an equivalence relation $\sim$ on $\{1,\ldots,n_j\}$ by writing $i_1 \sim i_2$ if and only if $\gamma_{i_1}(g)=\gamma_{i_2}(g)$ for all $g \in G^o$. Let $\mathsf{x}_1,\ldots,\mathsf{x}_p$ denote the equivalence classes under $\sim$. There is a natural action of $G/G^o$ on $\{1,\ldots,n_j\}$ which takes the pair $([g],i)$ to the unique integer $i'$ such that $\pi_j(g)U_j^i=U_j^{i'}$, and this action is obviously transitive since $G/G^o$ acts transitively on the spaces $U_j^1,\ldots,U_j^{n_j}$.  For distinct $i_1$ and $i_2$ and arbitrary $g \in G^o$ and $h \in G$ it is not difficult to see that $\pi_j(g)$ has distinct similarity ratios on $U_j^{i_1}$ and $U_j^{i_2}$ if and only if $\pi_j(hgh^{-1})$ has distinct similarity ratios on $\pi_j(h)U_j^{i_1}$ and $\pi_j(h)U_j^{i_2}$, so the action on $\{1,\ldots,n_j\}$ respects the equivalence relation $\sim$ and in particular has the effect of inducing a permutation of the equivalence classes $\mathsf{x}_1,\ldots,\mathsf{x}_p$. The transitivity of the action of $G/G^o$ on $\{1,\ldots,n_j\}$ easily implies that this action of $G/G^o$ on the set of equivalence classes is transitive. It follows in particular that the equivalence classes must all have the same cardinality: we have $\#\mathsf{x}_t=n_j/p$ for every $t=1,\ldots,p$.

For each equivalence class $\mathsf{x}_t$ define $X_t$ to be the span of the union of all the subspaces $U_j^i$ such that $i \in \mathsf{x}_t$. Arguing as in the second paragraph of \S\ref{subsub.family.of.subspaces} we note that every $X_t$ must be equal to a direct sum $U_j^{i_1}\oplus \cdots \oplus U_j^{i_q}$ for some suitable choice of indices $i_1,\ldots,i_q \in \mathsf{x}_t$ and for some integer $q \geq 1$ which \emph{a priori} might depend on $t$. (To see this, consider a direct sum $U_j^{i_1}\oplus \cdots \oplus U_j^{i_q}\subseteq X_t$ with $i_1,\ldots,i_q \in \mathsf{x}_t$ which is maximal in the sense that it cannot be extended by a further direct summand $U_j^{i_{q+1}}$ such that $i_{q+1} \in \mathsf{x}_t$. If every $U^i_j$ satisfying $i \in \mathsf{x}_t$ is a subspace of this direct sum then the direct sum equals $X_t$ as required. Otherwise, there exists $U^i_j$ satisfying $i \in \mathsf{x}_t$ which neither is a subspace of $U_j^{i_1}\oplus \cdots \oplus U_j^{i_q}$ nor forms a direct sum with it, in which case the intersection $(U_j^{i_1}\oplus \cdots \oplus U_j^{i_q}) \cap U_j^i$ is nonzero, has finite orbit under the action of $\pi_j(G)$, and has dimension smaller than $\ell_j$, contradicting the definition of $\ell_j$. We conclude that any such maximal direct sum yields a decomposition of $X_t$ with the claimed properties.) Now, as a consequence of the result shown in \S\ref{sss:ben}, every $U^i_j$ admits an inner product structure with respect to which every $g \in G^o$ acts on $U^i_j$ as a similarity transformation. Combined with the existence of the aforementioned direct sums this implies that for every $t=1,\ldots,p$ there exists an inner product structure on $X_t$ with respect to which every $g \in G^o$ acts on $X_t$ as a similarity transformation. For distinct $t_1,t_2$ in the range $1,\ldots,p$, by the definition of $\sim$ there exists $g \in G^o$ such that $\pi_j(g)$ has different similarity ratios on $X_{t_1}$ and on $X_{t_2}$, and this implies that necessarily $X_{t_1} \cap X_{t_2}=\{0\}$. We conclude that the spaces $X_1,\ldots,X_p$ form a direct sum, which is equal to the span of the spaces $U_j^1,\ldots,U_j^{n_j}$ and hence is equal to $V_j$. Since $G/G^o$ transitively permutes the set of equivalence classes $\mathsf{x}_1,\ldots,\mathsf{x}_p$ it follows that the action $([g],X_t) \mapsto \pi_j(g)X_t$ transitively permutes the spaces $X_1,\ldots,X_p$. These spaces are  therefore pairwise isomorphic, so $\dim X_t$ is independent of $t$ and therefore $\dim X_t=d_j/p$ for every $i=1,\ldots,p$. 

We may now prove \eqref{eq:dets-ad-mortem}. We observe that for every $g \in G^o$ and $t \in \{1,\ldots,p\}$
\[\left|\det \left(\pi_j(g)|_{X_{t}}\right)\right|^{\frac{1}{\dim X_{t}}} = \left(\prod_{i \in \mathsf{x}_{t}}  \left|\det \left(\pi_j(g)|_{U_j^i}\right)\right|^{\frac{1}{\ell_j}}\right)^{\frac{1}{\#\mathsf{x}_{t} }}  \]
 because the term on the left is the similarity ratio of $\pi_j(g)$ on $X_{t}$, which is also the similarity ratio of $\pi_j(g)$ on $U_j^i$ for every $i \in \mathsf{x}_t$. This is to say
 \[\left|\det \left(\pi_j(g)|_{X_t}\right)\right|^{\frac{p}{d_j}} = \left(\prod_{i \in \mathsf{x}_{t}}  \left|\det \left(\pi_j(g)|_{U_j^i}\right)\right|^{\frac{1}{\ell_j}}\right)^{\frac{p}{n_j}}  \]
for every $t=1,\ldots,p$. Since $V_j=\bigoplus_{t=1}^p X_t$, we also have
 \[\prod_{t=1}^p \det \left(\pi_j(g)|_{X_t}\right) = \det \pi_j(g).\]
Hence
\begin{align*}
|\det \pi_j(g)| &= \prod_{t=1}^p \left|\det \left(\pi_j(g)|_{X_t}\right)\right|\\
& = \prod_{t=1}^p\left(\prod_{i \in \mathsf{x}_{t}}  \left|\det \left(\pi_j(g)|_{U_j^i}\right)\right|^{\frac{1}{\ell_j}}\right)^{\frac{d_j}{n_j}}  =\left(\prod_{i=1}^{n_j} \left|\det \left(\pi_j(g)|_{U_j^i}\right)\right|^{\frac{1}{\ell_j}}\right)^{\frac{d_j}{n_j}}  \end{align*}
and this is precisely \eqref{eq:dets-ad-mortem}.

\subsubsection{A second identity involving determinants} Here, we will apply \eqref{eq:dets-ad-mortem} to derive a further identity: we claim that for all $g \in G^o$ and $\mathcal{W} \in \mathscr{W}$
\begin{equation}\label{eq:dets-again}\left(\prod_{(W_j)_{j=1}^k \in \mathcal{W}} \prod_{j=1}^k \left|\det \left(\pi_j(g)|_{W_j}\right)\right|^{\frac{\beta_j}{\ell_j}} \right)^{\frac{1}{\#\mathcal{W}}} = \prod_{j=1}^k|\det \pi_j(g)|^{\frac{\beta_j}{d_j}}.\end{equation}
To see this fix $g \in G^o$, let $\mathcal{W}$ be a transitivity class and let $(W_j')_{j=1}^k \in \mathcal{W}$ be arbitrary. We note that the sets
\[\left\{[h] \in G/G^o \colon (\pi_j(h)W_j')_{j=1}^k =(W_j)_{j=1}^k\right\}\]
for distinct $(W_j)_{j=1}^k \in \mathcal{W}$ form a partition of $G/G^o$ into cosets, hence each has the same cardinality. We deduce that
\[\left(\prod_{(W_j)_{j=1}^k \in \mathcal{W}} \prod_{j=1}^k \left|\det \left(\pi_j(g)|_{W_j}\right)\right|^{\frac{\beta_j}{\ell_j}}\right)^{\frac{1}{\#\mathcal{W}}}=\left(\prod_{[h]\in G/G^o} \prod_{j=1}^k \left|\det \left(\pi_j(g)|_{\pi_j(h)W_j'}\right)\right|^{\frac{\beta_j}{\ell_j}}\right)^{\frac{1}{\#G/G^o}}.\] 
It is therefore sufficient to show that for each $j=1,\ldots,k$, for every $i_0 \in \{1,\ldots,n_j\}$,
\[\left(\prod_{[h]\in G/G^o} \left|\det \left(\pi_j(g)|_{\pi_j(h)U_j^{i_0}}\right)\right|^{\frac{1}{\ell_j}}\right)^{\frac{1}{\#G/G^o}} = |\det \pi_j(g)|^{\frac{1}{d_j}}.\]
Fix such a $j$ and $i_0$. As before, the sets
\[\left\{[h] \in G/G^o \colon \pi_j(h)U_j^{i_0}=U^i_j\right\}\]
form a partition of $G/G^o$ into cosets and hence have equal cardinality, which implies that
\[\left(\prod_{[h]\in G/G^o} \left|\det \left(\pi_j(g)|_{\pi_j(h)U_j^{i_0}}\right)\right|^{\frac{1}{\ell_j}}\right)^{\frac{1}{\#G/G^o}} = \left(\prod_{i=1}^{n_j} \left|\det \left(\pi_j(g)|_{U_j^i}\right)\right|^{\frac{1}{\ell_j}}\right)^{\frac{1}{n_j}}.\] 
By \eqref{eq:dets-ad-mortem} this last expression is equal to $|\det \pi_j(g)|^{1/d_j}$, so combining the identities obtained so far yields \eqref{eq:dets-again}.

\subsubsection{Two inequalities between potentials}\label{subsub.two.ineq}
Let us define a new potential by
\[\Phi^{\det}(\iii):=\prod_{j=1}^k |\det \pi_j(g_\iii)|^{\frac{\beta_j}{d_j}}\]
for all $\iii \in \Sigma_N^*$. We clearly have $\Phi^{\det}(\iii\jjj)=\Phi^{\det}(\iii)\Phi^{\det}(\jjj)$ for all $\iii,\jjj \in \Sigma_N^*$. We aim to show that 
\begin{equation}\label{eq:pressures-same}P(\Phi)=P(\Phi^{\mathcal{W}})=P(\Phi^{\det})\end{equation}
for all transitivity classes $\mathcal{W}$.

In pursuit of \eqref{eq:pressures-same} we will prove two inequalities. We first claim that there exists $C_4>0$ such that for every transitivity class $\mathcal{W}$ we have $\Phi^{\det}(\iii)\leq C_4\Phi^{\mathcal{W}}(\iii) $ for all $\iii \in \Sigma_N^*$. We begin by considering the case where $\iii \in \Sigma_N^*$ satisfies $\iii \in G^o$. It follows easily from \eqref{eq:dets-again} that
\begin{align}\label{eq:sixteen}\Phi^{\mathcal{W}}(\iii)&=\max_{(W_j)_{j=1}^k \in \mathcal{W}} \prod_{j=1}^k \left\|\pi_j(g_\iii)|_{W_j}\right\|^{\beta_j} \geq \max_{(W_j)_{j=1}^k \in \mathcal{W}} \prod_{j=1}^k  \left|\det \left(\pi_j(g_\iii)|_{W_j}\right)\right|^{\frac{\beta_j}{\ell_j}} \\\nonumber
&\geq \left(\prod_{(W_j)_{j=1}^k \in \mathcal{W}} \prod_{j=1}^k \left|\det \left(\pi_j(g_\iii)|_{W_j}\right)\right|^{\frac{\beta_j}{\ell_j}} \right)^{\frac{1}{\#\mathcal{W}}}\\\nonumber
& = |\det \pi_j(g_\iii)|^{\frac{\beta_j}{d_j}} = \Phi^{\det}(\iii)\end{align}
for every transitivity class $\mathcal{W}$ and every $\iii \in \Sigma_N^*$ such that $g_\iii \in G^o$. Now observe that by the Zariski density of the semigroup $\{g_\iii \colon \iii \in \Sigma_N^*\}$ in $G$, we may choose $\kkk_1,\ldots,\kkk_r$ such that every connected component of $G$ contains one of the elements $g_{\kkk_t}$. Given $\iii \in \Sigma_N^*$ observe that we can choose $t_0 \in \{1,\ldots,r\}$ such that $g_{\iii \kkk_{t_0}} \in G^o$. We have $\Phi^{\mathcal{W}}(\iii \kkk_{t_0}) \leq \Phi^{\mathcal{W}}(\iii)\Phi^{\mathcal{W}}(\kkk_{t_0})$ and $\Phi^{\det}(\iii\kkk_{t_0})=\Phi^{\det}(\iii)\Phi^{\det}(\kkk_{t_0})$, and therefore using \eqref{eq:sixteen}
\[\frac{\Phi^{\det}(\iii)}{\Phi^{\mathcal{W}}(\iii)} \leq \left(\frac{\Phi^{\mathcal{W}}(\kkk_{t_0})}{\Phi^{\mathcal{W}}(\iii\kkk_{t_0})}\right)\left(\frac{\Phi^{\det}(\iii\kkk_{t_0})}{\Phi^{\det}(\kkk_{t_0})}\right) \leq \frac{\Phi^{\mathcal{W}}(\kkk_{t_0})}{\Phi^{\det}(\kkk_{t_0})}\leq C_4,\]
say, where
\[C_4:=\max_{\mathcal{W}\in\mathscr{W}} \max_{1 \leq t \leq r} \frac{\Phi^{\mathcal{W}}(\kkk_{t})}{\Phi^{\det}(\kkk_{t})}\]
which proves the claim.

We now establish our second inequality: we claim that there exists $C_5>0$ such that $\Phi^{\mathcal{W}_0}(\iii) \leq C_5 \Phi^{\det}(\iii)$ for every $\iii \in \Sigma_N^*$. By the inequality \eqref{eq:thank-you-reviewer-C} established in \S\ref{sss:c3} we have
\[\Phi^{\mathcal{W}_0}(\iii) \leq C_3 \prod_{j=1}^k \left\|\pi_j(g_\iii)|_{W_j}\right\|^{\beta_j} \]
for some  $C_3>0$ and every  $g_\iii \in G^o$ and $(W_j)_{j=1}^k\in\mathcal{W}_0$. It follows in particular that \[ \frac{\Phi^{\mathcal{W}_0}(\iii)}{\prod_{j=1}^k \left| \det \left(\pi_j(g_\iii)|_{W_j}\right)\right|^{\frac{\beta_j}{\ell_j}}} \leq \frac{C_3 \prod_{j=1}^k \left\|\pi_j(g_\iii)|_{W_j}\right\|^{\beta_j}}{\prod_{j=1}^k \left| \det \left(\pi_j(g_\iii)|_{W_j}\right)\right|^{\frac{\beta_j}{\ell_j}}}\]
for all $g_\iii \in G^o$ and $(W_j)_{j=1}^k\in\mathcal{W}_0$.
Since for each $j$
\[\left\{\left| \det \left(\pi_j(g)|_{W_j}\right)\right|^{-\frac{1}{\ell_j}} \pi_j(g)|_{W_j}\colon g \in G^o\right\}\]
is contained in a compact subset of $\GL(W_j)$, it follows that there exists $K>0$ such that
\[ \frac{\Phi^{\mathcal{W}_0}(\iii)}{\prod_{j=1}^k \left| \det \left(\pi_j(g_\iii)|_{W_j}\right)\right|^{\frac{\beta_j}{\ell_j}}} \leq K\]
for all $g_\iii \in G^o$ and $(W_j)_{j=1}^k\in\mathcal{W}_0$. Taking  the geometric mean over all $(W_j)_{j=1}^k\in\mathcal{W}_0$ for fixed $g_\iii$ using \eqref{eq:dets-again} yields
\[\frac{\Phi^{\mathcal{W}_0}(\iii)}{\Phi^{\det}(\iii)}   \leq K\]
for all $\iii \in \Sigma_N^*$ such that $g_\iii \in G^o$. We now extend to the case of general words $\iii$. Fix $\iii \in \Sigma_N^*$ and observe that we may choose $t_0 \in \{1,\ldots,r\}$ such that $g_{\iii\kkk_{t_0}} \in G^o$. For some $(W_j)_{j=1}^k \in \mathcal{W}^0$ we have
 \[\Phi^{\mathcal{W}_0}(\iii) =  \prod_{j=1}^k \left\|\pi_j(g_\iii)|_{W_j}\right\|^{\beta_j}\]
and therefore
\begin{align*} \Phi^{\mathcal{W}_0}(\iii)&=\prod_{j=1}^k \left\|\pi_j(g_{\iii} g_{\kkk_{t_0}} g_{\kkk_{t_0}}^{-1})|_{W_j}\right\|^{\beta_j}\leq  \prod_{j=1}^k \left\|\pi_j(g_{\iii \kkk_{t_0}})|_{\pi_j(g_{\kkk_{t_0}}^{-1})W_j}\right\|^{\beta_j}\prod_{j=1}^k \left\|\pi_j(g_{\kkk_{t_0}}^{-1})|_{W_j}\right\|^{\beta_j} \\
&\leq \left( \prod_{j=1}^k \left\|\pi_j(g_{\iii \kkk_{t_0}})|_{\pi_j(g_{\kkk_{t_0}}^{-1})W_j}\right\|^{\beta_j} \right)\left(\max_{1 \leq t \leq r}\prod_{j=1}^k \left\|\pi_j(g_{\kkk_{t}}^{-1})|_{W_j}\right\|^{\beta_j} \right) \\
 &\leq C \Phi^{\mathcal{W}_0}(\iii \kkk_{t_0}) \leq KC\Phi^{\det}(\iii \kkk_{t_0})\\
 & \leq KC \left(\max_{1 \leq t\leq r} \Phi^{\det}(\kkk_t)\right)\Phi^{\det}(\iii)\leq C_5\Phi^{\det}(\iii),\end{align*}
where we took $C:=\max_{1 \leq t \leq r}\prod_{j=1}^k \left\|\pi_j(g_{\kkk_{t}}^{-1})|_{W_j}\right\|^{\beta_j}$ and $C_5:=KC\max_{1 \leq t\leq r} \Phi^{\det}(\kkk_t)$. This proves the claim. 
 
\subsubsection{The Gibbs property and a third inequality between potentials}
 
The two inequalities just proved assert that for some $C>0$
\begin{equation}\label{eq:three-potentials}C^{-1} \Phi^{\mathcal{W}_0}(\iii) \leq \Phi^{\det}(\iii) \leq C\Phi^{\mathcal{W}}(\iii)\end{equation} 
 for all $\iii \in \Sigma_N^*$ and all transitivity classes $\mathcal{W}$. It follows directly that 
 \[P(\Phi)=P(\Phi^{\mathcal{W}_0})\leq P(\Phi^{\det})\leq P(\Phi^{\mathcal{W}}) \leq P(\Phi)  \]
for all transitivity classes $\mathcal{W}$, and we have proved the identity \eqref{eq:pressures-same}: $P(\Phi)=P(\Phi^{\mathcal{W}})=P(\Phi^{\det})$ for all transitivity classes $\mathcal{W}$.

We may now prove that $\mu$ is the equilibrium state of $\Phi^{\mathcal{W}}$ for \emph{every} transitivity class $\mathcal{W}$, and is also the equilibrium state of $\Phi^{\det}$. Indeed, for each transitivity class $\mathcal{W}$ the inequality \eqref{eq:three-potentials} yields
 \[\Lambda(\Phi^{\mathcal{W}_0},\mu)\leq  \Lambda(\Phi^{\det},\mu) \leq \Lambda(\Phi^{\mathcal{W}},\mu)\]
 and therefore
 \begin{align*}P(\Phi) = P(\Phi^{\det}) = P(\Phi^{\mathcal{W}_0})
 & =  h(\mu)+\Lambda(\Phi^{\mathcal{W}_0},\mu) \leq     h(\mu)+ \Lambda(\Phi^{\det},\mu)\\
  & \leq h(\mu)+\Lambda(\Phi^{\mathcal{W}},\mu) \leq P(\Phi^{\mathcal{W}})= P(\Phi)\end{align*}
 so that
 \[P(\Phi^{\mathcal{W}}) = h(\mu)+\Lambda(\Phi^{\mathcal{W}},\mu) \qquad \text{and} \qquad P(\Phi^{\det}) = h(\mu)+\Lambda(\Phi^{\det},\mu)\]
 as required for $\mu$  to be an equilibrium state of $\Phi^{\mathcal{W}}$ and $\Phi^{\det}$ respectively.

We now make further use of the Gibbs inequality. Each $\Phi^{\mathcal{W}}$ has a unique equilibrium state and satisfies the Gibbs inequality with respect to that equilibrium state, and the equilibrium state of each such potential is $\mu$. The same remarks apply to $\mu$ and the potential $\Phi^{\det}$. Therefore there exists $C_6>0$ such that 
\[C_6^{-1} \leq \frac{\Phi^{\mathcal{W}}(\iii)}{e^{-|\iii|P(\Phi^{\mathcal{W}})} \mu([\iii])} = \frac{\Phi^{\mathcal{W}}(\iii)}{e^{-|\iii|P(\Phi)} \mu([\iii])} \leq C_6\]
for all $\iii \in \Sigma_N^*$ and all transitivity classes $\mathcal{W}$, and also 
\[C_6^{-1} \leq \frac{\Phi^{\det}(\iii)}{e^{-|\iii|P(\Phi^{\det})} \mu([\iii])} = \frac{\Phi^{\det}(\iii)}{e^{-|\iii|P(\Phi)} \mu([\iii])} \leq C_6\]
for all $\iii \in\Sigma_N^*$. We deduce the inequality $\Phi^{\mathcal{W}}(\iii) \leq C_6^2 \Phi^{\det}(\iii)$ for all $\iii \in \Sigma_N^*$ and transitivity classes $\mathcal{W}$.

 \subsubsection{A final determinant identity}

Let $\iii \in \Sigma_N^*$ such that $g_\iii \in G^o$. We have
\begin{align*}&\max_{(W_j)_{j=1}^k \in \mathcal{W}} \prod_{j=1}^k \left|\det \left(\pi_j(g_\iii)|_{W_j}\right)\right|^{\frac{\beta_j}{\ell_j}} \leq \max_{(W_j)_{j=1}^k \in \mathcal{W}} \prod_{j=1}^k \left\|\pi_j(g_\iii)|_{W_j}\right\|^{\beta_j}\\
&=\Phi^{\mathcal{W}}(\iii)\leq C_6^2\Phi^{\det}(\iii)=C_6^2\prod_{j=1}^k |\det \pi_j(g_\iii)|^{\frac{\beta_j}{d_j}}\\
&=C_6^2\left(\prod_{(W_j)_{j=1}^k \in \mathcal{W}} \prod_{j=1}^k \left|\det \left(\pi_j(g_\iii)|_{W_j}\right)\right|^{\frac{\beta_j}{\ell_j}} \right)^{\frac{1}{\#\mathcal{W}}}\\
&\leq C_6^2\left(\min_{(W_j)_{j=1}^k \in \mathcal{W}} \prod_{j=1}^k \left|\det \left(\pi_j(g_\iii)|_{W_j}\right)\right|^{\frac{\beta_j}{\ell_j}} \right)^{\frac{1}{\#\mathcal{W}}}\\
&\qquad\cdot\left(\max_{(W_j)_{j=1}^k \in \mathcal{W}} \prod_{j=1}^k \left|\det \left(\pi_j(g_\iii)|_{W_j}\right)\right|^{\frac{\beta_j}{\ell_j}}\right)^{\frac{\#\mathcal{W}-1}{\#\mathcal{W}}} \end{align*}
and we obtain
\[\max_{(W_j)_{j=1}^k \in \mathcal{W}} \prod_{j=1}^k \left|\det \left(\pi_j(g_\iii)|_{W_j}\right)\right|^{\frac{\beta_j}{\ell_j}}\leq C_6^{2(\#\mathcal{W})} \min_{(W_j)_{j=1}^k \in \mathcal{W}} \prod_{j=1}^k \left|\det \left(\pi_j(g_\iii)|_{W_j}\right)\right|^{\frac{\beta_j}{\ell_j}}\]
for all transitivity classes $\mathcal{W}$ and all $g_\iii \in G^o$. It follows that if $(W_j')_{j=1}^k$ is any element of any transitivity class $\mathcal{W}$, then for every $g_\iii \in G^o$ 
\begin{align*}\prod_{j=1}^k \left|\det \left(\pi_j(g_\iii)|_{W_j'}\right)\right|^{\frac{\beta_j}{\ell_j}}& \geq \min_{(W_j)_{j=1}^k \in \mathcal{W}} \prod_{j=1}^k \left|\det \left(\pi_j(g_\iii)|_{W_j}\right)\right|^{\frac{\beta_j}{\ell_j}}\\
 &\geq C_6^{-2(\#\mathcal{W})} \max_{(W_j)_{j=1}^k \in \mathcal{W}} \prod_{j=1}^k \left|\det \left(\pi_j(g_\iii)|_{W_j}\right)\right|^{\frac{\beta_j}{\ell_j}}\\
&\geq C_6^{-2(\#\mathcal{W})} \left(\prod_{(W_j)_{j=1}^k \in \mathcal{W}} \left(\prod_{j=1}^k \left|\det \left(\pi_j(g_\iii)|_{W_j}\right)\right|^{\frac{\beta_j}{\ell_j}}\right)\right)^{\frac{1}{\#\mathcal{W}}}\\
&=C_6^{-2(\#\mathcal{W})}\prod_{j=1}^k |\det \pi_j(g_\iii)|^{\frac{\beta_j}{d_j}}\end{align*}
where we have used \eqref{eq:dets-again} again, and from the preceding chain of inequalities
\[\prod_{j=1}^k \left|\det \left(\pi_j(g_\iii)|_{W_j'}\right)\right|^{\frac{\beta_j}{\ell_j}}\leq \max_{(W_j)_{j=1}^k \in \mathcal{W}} \prod_{j=1}^k \left|\det \left(\pi_j(g_\iii)|_{W_j}\right)\right|^{\frac{\beta_j}{\ell_j}} \leq C_6^2\prod_{j=1}^k |\det \pi_j(g)|^{\frac{\beta_j}{d_j}}.\]
We have found that if $\iii \in \Sigma_N^*$ such that $g_\iii \in G^o$, $\mathcal{W}$ is any transitivity class and $(W_j)_{j=1}^k$ any element of $\mathcal{W}$
\[C_6^{-2(\#\mathcal{W})}\prod_{j=1}^k |\det \pi_j(g_\iii)|^{\frac{\beta_j}{d_j}} \leq  \prod_{j=1}^k \left|\det \left(\pi_j(g_\iii)|_{W_j}\right)\right|^{\frac{\beta_j}{\ell_j}}\leq C_6^2\prod_{j=1}^k |\det \pi_j(g_\iii)|^{\frac{\beta_j}{d_j}}.\]
Applying this estimate to $g_{\iii^n}=g_\iii^n$ in place of $g_\iii$, taking the power $\frac{1}{n}$ and letting $n \to \infty$ yields
\begin{equation}\label{eq:dets-same-2}\prod_{j=1}^k |\det \pi_j(g_\iii)|^{\frac{\beta_j}{d_j}} = \prod_{j=1}^k \left|\det \left(\pi_j(g_\iii)|_{W_j}\right)\right|^{\frac{\beta_j}{\ell_j}}\end{equation}
for every $g_\iii \in G^o$ and every $(W_j)$ in any transitivity class. 

\subsubsection{Conclusion of the proof}

The equation \eqref{eq:dets-same-2} suffices to yield \eqref{eq:dets-same}. Fix $j_0 \in \{1,\ldots,k\}$ and $1 \leq i_1,i_2 \leq n_{j_0}$. Let $W_{j_0}:=U_{j_0}^{i_1}$ and $W_{j_0}':=U_{j_0}^{i_2}$, and for $j \neq j_0$, set $W_j:=U_j^1$ and $W_j':=U_j^1$. Applying \eqref{eq:dets-same-2} gives
\[\frac{\left|\det\left(\pi_{j_0}(g_\iii)|_{U_{j_0}^{i_1}}\right)\right|^{\frac{\beta_{j_0}}{\ell_{j_0}}}}{\left|\det\left(\pi_{j_0}(g_\iii)|_{U_{j_0}^{i_2}}\right)\right|^{\frac{\beta_{j_0}}{\ell_{j_0}}}} = \frac{\prod_{j=1}^k \left|\det \left(\pi_j(g_\iii)|_{W_j}\right)\right|^{\frac{\beta_j}{\ell_j}}}{\prod_{j=1}^k \left|\det \left(\pi_j(g_\iii)|_{W_j'}\right)\right|^{\frac{\beta_j}{\ell_j}}} =\frac{\prod_{j=1}^k |\det \pi_j(g_\iii)|^{\frac{\beta_j}{d_j}} }{\prod_{j=1}^k |\det \pi_j(g_\iii)|^{\frac{\beta_j}{d_j}} } =1\]
for every $g_\iii \in G^o$. Hence for every $g_\iii \in G^o$ and every $j \in \{1,\ldots,k\}$,
\[\left|\det\left(\pi_j(g_\iii)|_{U_j^{i}}\right)\right|^{\frac{1}{\ell_j}}\]
is independent of $i \in \{1,\ldots,n_j\}$ and in particular must be equal to its geometric mean with respect to $i \in \{1,\ldots,n_j\}$, which by \eqref{eq:dets-ad-mortem} is $\left|\det \pi_j(g_\iii)\right|^{1/d_j}$.
This establishes \eqref{eq:dets-same} which in turn allows us to readily conclude. Indeed, together with \eqref{eq.matrix.form}, it implies that for every $g \in G^o$ and $j\in \{1,\ldots,k\}$, $\pi_j(g)=|\det(\pi_j(g))|^\frac{1}{d_j}O_j(g)$ where $O_j(g) \in O(V_j)$ for some Euclidean structure on $V_j$ not depending on $g$. Therefore 
\[\left\{\left| \det \pi_j(g)\right|^{-\frac{1}{d_j}} \pi_j(g)\colon g \in G^o\right\}\]
is a compact subgroup of $\GL(V_j)$ and since the index $[G:G^o]$ is finite, the same is true of
\[\left\{\left| \det \pi_j(g)\right|^{-\frac{1}{d_j}} \pi_j(g)\colon g \in G\right\}.\]
The proof is complete.


%
%

\section{Proof of Theorem \ref{th:main-tech}}\label{se:gen-case}

Let $(A_1,\ldots,A_N) \in \GL_d(\mathbb{R})^N$ be irreducible and let $\alpha_1 \geq \cdots \geq \alpha_d\geq 0$ with $\alpha_1>\alpha_d$.  Let $G\leq \GL_d(\mathbb{R})$ denote the Zariski closure of the subsemigroup of $\GL_d(\mathbb{R})$ generated by $A_1,\ldots,A_N$; it is a real reductive group (\S \ref{subsub.reductive.irred}). Define $\alpha_{d+1}:=0$ and let $k_1,\ldots,k_r$ be the list of all integers $i \in \{1,\ldots,d\}$ for which the difference $\alpha_i-\alpha_{i+1}$ is positive, where $k_1<\cdots<k_r$. We observe that since $\alpha_1>\alpha_d$ we have $r \neq 0$ and also $k_1<d$. Define $\beta_j:=\alpha_{k_j}-\alpha_{1+k_{j}}>0$ for each $j=1,\ldots,r$, and for each $j=1,\ldots,r$ let $\pi_j \colon G \to \GL(\wedge^{k_j}\mathbb{R}^d)$ denote the exterior power representation $\pi_j(g):=g^{\wedge k_j}$. We have
\[\prod_{j=1}^d \sigma_j(g)^{\alpha_j} =\prod_{j=1}^d \left(\prod_{i=1}^j \sigma_i(g)\right)^{\alpha_j-\alpha_{j+1}}=\prod_{j=1}^d \left\|g^{\wedge j}\right\|^{\alpha_j-\alpha_{j+1}}= \prod_{j=1}^r \left\|\pi_j(g)\right\|^{\beta_j}\]
for every $g \in G$, and in particular the potential $\Phi$ defined in the statement of the theorem satisfies the description
\[\Phi(\iii)=\prod_{j=1}^r \left\|\pi_j(A_\iii)\right\|^{\beta_j}.\]

Since the representations $\pi_j \colon G\to \GL(\wedge^{k_j}\mathbb{R}^d)$ are not in general irreducible, Theorem \ref{th:irreducible-case} is not directly applicable to the potential $\Phi$. We will study $\Phi$ by writing it as the maximum of a finite collection of simpler potentials to which Theorem \ref{th:irreducible-case} may be applied. Since $G$ is reductive, the rational representations $\pi_j$'s are completely reducible (\S \ref{subsub.reductive.irred}), in other words, for each $j=1,\ldots,r$ we may write $\wedge^{k_j}\mathbb{R}^d=V_1^j \oplus \cdots \oplus V_{n_j}^j$ where each $V_i^j$ is an invariant subspace of the group $\pi_j(G)$ on which $\pi_j(G)$ acts irreducibly. For each $j=1,\ldots,r$ and $1 \leq \ell \leq n_j$ define an irreducible representation $\pi_{j,\ell}\colon G \to \GL(V_\ell^j)$ by $\pi_{j,\ell}(g):=\pi_j(g)|_{V_\ell^j}$ for all $g \in G$. Let $\mathfrak{L}$ denote the set of all tuples of integers $\mathfrak{l}=(\ell_1,\ldots,\ell_r)$ such that $1 \leq \ell_j \leq n_j$ for each $j=1,\ldots,r$. For each $\mathfrak{l}=(\ell_1,\ldots,\ell_r) \in \mathfrak{L}$ define a potential $\Phi_{\mathfrak{l}} \colon \Sigma_N^* \to (0,+\infty)$ by
\begin{equation}\label{eq:frakpotential1} \Phi_{\mathfrak{l}}(\iii):= \prod_{j=1}^r \left\|\pi_j(A_\iii)|_{V_{\ell_j}^j}\right\|^{\beta_j} = \prod_{j=1}^r \left\|\pi_{j,\ell_j}(A_\iii)\right\|^{\beta_j}.\end{equation}
For each fixed $\mathfrak{l}=(\ell_1,\ldots,\ell_r)$ the representations $\pi_{j,\ell_j}$ for $j=1,\ldots,r$ are irreducible, so each $\Phi_{\mathfrak{l}}$ satisfies the hypotheses of Theorem \ref{th:irreducible-case}. Clearly we also have 
\begin{align}\label{eq:l-max}\Phi(\iii)=\prod_{j=1}^r \left\|\pi_j(A_\iii)\right\|^{\beta_j} &=\prod_{j=1}^r \max_{1 \leq \ell \leq n_j} \left\|\pi_j(A_\iii)|_{V_\ell^j}\right\|^{\beta_j}\\\nonumber
& =\max_{(\ell_1,\ldots,\ell_r)\in \mathfrak{L}} \prod_{j=1}^r \left\|\pi_j(A_\iii)|_{V_{\ell_j}^j}\right\|^{\beta_j} = \max_{\mathfrak{l} \in \mathfrak{L}} \Phi_{\mathfrak{l}}(\iii) \end{align}
for every $\iii \in \Sigma_N^*$. We will find it helpful to define further potentials as follows. For each $\mathfrak{l}=(\ell_1,\ldots,\ell_r) \in \mathfrak{L}$ define
\begin{equation}\label{eq:frakpotential2} \Phi^{\det}_{\mathfrak{l}}(\iii):= \prod_{j=1}^r \left|\det \left(\pi_{j,\ell_j}(A_\iii)\right)\right|^{\frac{\beta_j}{\dim V_{\ell_j}^j}}\end{equation}
for all $\iii \in \Sigma_N^*$. Define also 
\[\Phi^{\det}(\iii) = \prod_{j=1}^r \left|\det A_\iii \right|^{\frac{k_j\beta_j}{d}},\]
for all $\iii \in \Sigma_N^*$.

Our strategy in proving Theorem \ref{th:main-tech} will be to establish the identity
\begin{equation}\label{eq:pressure-goal}P(\Phi_{\mathfrak{l}}) =P(\Phi_{\mathfrak{l}}^{\det})\end{equation}
for all $\mathfrak{l} \in \mathfrak{L}$. This will permit the implication (ii)$\implies$(iii) of Theorem \ref{th:irreducible-case} to be applied, establishing that each of the groups $\pi_{j,\ell}(G)$ is compact modulo factoring out the determinant. The compactness of each $\pi_j(G)$ modulo factoring out the determinant will then follow via some additional bookkeeping to ensure that for each $j=1,\ldots,r$ the determinant which is factored out of the representation $\pi_{j,\ell}$ is consistent across all $\ell \in \{1,\ldots,n_j\}$, and the compactness of $G$ modulo factoring out the determinant will follow by some simple manipulations involving singular values.

Much as in the second half of the proof of Theorem \ref{th:irreducible-case}, before commencing the proof of \eqref{eq:pressure-goal} we must first establish an identity involving determinants. The proof of this identity is relatively long and comprises a large proportion of this section. Specifically, we make the following claim: for every $j=1,\ldots,r$, for all $\ell=1,\ldots,n_j$ we have  
\begin{equation}\label{eq:dets}\left|\det \left(\pi_{j,\ell}(g)\right)\right|^{\frac{1}{\dim V_{\ell}^j}} =  \left|\det g \right|^{\frac{k_j}{d}} \end{equation}
for all $g \in G$.

To prove the claim it is sufficient for us to establish \eqref{eq:dets} for all $g \in G^o$, since if this has been proven then for any given $g \in G$ we have $g^n \in G^o$ for some integer $n \geq 1$ and hence clearly
\[\left|\det \left(\pi_{j,\ell}(g)\right)\right|^{\frac{1}{\dim V_{\ell}^j}} =\left|\det \left(\pi_{j,\ell}(g^n)\right)\right|^{\frac{1}{n\cdot \dim V_{\ell}^j}}    = \left|\det (g^n) \right|^{\frac{k_j}{nd}} =\left|\det g \right|^{\frac{k_j}{d}}   \]
as required. We therefore restrict our attention to the task of proving \eqref{eq:dets} for all $g \in G^o$. To this end let us fix $j$ and $\ell$ and define a continuous group homomorphism $\hat\pi$ from $G^o$ to the multiplicative group of positive real numbers by $\hat\pi(g):=|\det \left(\pi_{j,\ell}(g)\right)|^{1/k_j \cdot \dim V_\ell^j}$. 
Our objective is now to show that $\hat\pi(g)=|\det g|^{1/d}$ for all $g \in G^o$. The set of all $g \in G^o$ satisfying this equation is obviously a group, and this set obviously includes $[G^o,G^o]$ as a subset since by the commutativity of real multiplication we have $\hat\pi(g)=1=|\det g|^{1/d}$ for all $g \in [G^o,G^o]$. Since $G^o$ is equal to an almost direct product of $Z(G^o)$ and $[G^o,G^o]$, the claim will therefore follow if we can prove that $\hat\pi(z)=|\det z|^{1/d}$ for all $z\in Z(G^o)$. 

We begin this process by analysing the action of $Z(G^o)$ on $\R^d$. By Clifford's  \cite[Theorem 1.7]{wehrfritz} applied to the irreducible group $G \leq \GL_d(\R)$ and its normal subgroup $G^o$, we obtain a direct sum decomposition $\R^d=X_1 \oplus \ldots \oplus X_p$ consisting of the homogeneous (isotypic) components of the $G^o$-representation. By the same result of Clifford the collection of subspaces $X_i$ are permuted by the component group $G/G^o$. In particular these subspaces all have the same dimension, which we denote by $m \in \N$. (Here of course we have  $m=d/p$.) Decomposing each $X_i$ into a sum of irreducible subspaces for the $G^o$-action and using Schur's lemma, we deduce that there exists an inner product structure on each $X_i$ with respect to which every $g \in Z(G^o)$ acts on $X_i$ by a similarity transformation.

Now, by \cite[Proposition 8.15]{borel.book} there exist a maximal compact subgroup $Z(G^o)_A$ and a maximal real diagonalisable subgroup $Z(G^o)_D$ of $Z(G^o)$ such that $Z(G^o)_A \cap Z(G^o)_D$ is finite and $Z(G^o)= Z(G^o)_D Z(G^o)_A$, which is to say $Z(G^o)$ is an almost direct product of the subgroups $Z(G^o)_D$ and $Z(G^o)_A$. 
The group $\hat\pi(Z(G^o)_A)$ is a compact subgroup of the positive reals and hence is equal to $\{1\}$, and similarly the image of $Z(G^o)_A$ under the homomorphism $z \mapsto |\det z|^{1/d}$ must also equal $\{1\}$, so we have $\hat\pi(z)=|\det z|^{1/d}$ for all $z \in Z(G^o)_A$. Hence  the claim will be proved if we can show that $\hat\pi(z)=|\det z|^{1/d}$ for every $z \in Z(G^o)_D$. 

Since each $X_i$ is a sum of isomorphic irreducible representations of $G^o$, it follows from Schur's lemma and (real) diagonalisability that every $z \in Z(G^o)_D$ acts on each $X_i$ by a scalar transformation $v \mapsto \gamma_i(z)v$ for some nonzero real number $\gamma_i(z)$ for $i=1,\ldots,p$.
On the other hand, since $V^j_\ell$ is $Z(G^o)_D$-invariant and since $Z(G^o)_D$ is abelian, $V^j_\ell$ writes as a direct sum of $Z(G^o)_D$-irreducible subspaces in $\wedge^{k_j}\mathbb{R}^d$. But $Z(G^o)_D$ is also a split torus, and therefore so is its image in the exterior power representations. Hence, these $Z(G^o)_D$-irreducible subspaces of $V^j_\ell$ are $1$-dimensional subspaces. Each gives rise to a character of $Z(G^o)_D$ of the form $\gamma_1(z)^{t_1}\cdots \gamma_p(z)^{t_p}$ for some non-negative integers $t_1,\ldots,t_p$ whose sum is equal to $k_j$. 
The quantity $\det \pi_{j,\ell}(z)=\det z^{\wedge k_j}|_{V_\ell^j}$ is a product of precisely $\dim V_\ell^j$ such characters, so it has the form $\gamma_1(z)^{t_1'}\cdots \gamma_p(z)^{t_p'}$ for some non-negative integers $t_1',\ldots,t_p'$ such that $\sum_{i=1}^p t_p'=k_j\cdot \dim V_\ell^j$. Taking the absolute value and raising to the power $1/(k_j\cdot \dim V_\ell^j)$ as in the definition of $\hat\pi$, it follows that there exist non-negative rational numbers $r_1,\ldots,r_p$ such that $\hat\pi(z)=|\gamma_1(z)|^{r_1}\cdots |\gamma_p(z)|^{r_p}$ for all $z \in Z(G^o)_D$ and such that $\sum_{i=1}^p r_i=1$. On the other hand clearly $\det z =\gamma_1(z)^m\cdots \gamma_p(z)^m$ for every $z \in Z(G^o)_D$ since $\mathbb{R}^d=\bigoplus_{i=1}^p X_i$ and $\det (z|_{X_i})=\gamma_i(z)^m$ for every $i=1,\ldots,p$, where we recall that $m=d/p$ is the dimension of each of the spaces $X_i$. Hence $|\det z|^{1/d} = |\gamma_1(z)\cdots \gamma_p(z)|^{1/p}$ for all $z \in Z(G^o)_D$.

Now, if $z \in Z(G^o)_D$ and $g \in G$ then $gzg^{-1}$ also belongs to $Z(G^o)$ and also acts on each $X_i$ by a scalar transformation, which by the maximality of $Z(G^o)_D$ as a real diagonalisable subgroup of $Z(G^o)$ implies $gzg^{-1} \in Z(G^o)_D$. For every $[g] \in G/G^o$ there exists a permutation $\varsigma$ of $\{1,\ldots,p\}$ such that $gX_i = X_{\varsigma(i)}$ for every $i =1,\ldots,p$ and every $g \in [g]$, and the corresponding element $gzg^{-1}$ of $Z(G^o)_D$ satisfies $\gamma_i(gzg^{-1}) = \gamma_{\varsigma(i)}(z)$ for all $i=1,\ldots,p$.  For each $i \in \{1,\ldots,p\}$ the transitivity of the action of $G/G^o$ on $X_1,\ldots,X_p$ implies that the sets $\left\{[g]\in G/G^o \colon gX_{i}=X_j\right\}$
for $j=1,\ldots,p$ form a partition of $G/G^o$ into cosets of equal cardinality $(\#G/G^o)/p$ and therefore
\begin{align}\label{eq:more-determinant-like-stuff}\prod_{[g]\in G/G^o} |\gamma_{i}(gzg^{-1})| &=\prod_{j=1}^p \left(\prod_{\substack{[g]\in G/G^o \\ gX_i=X_j}} |\gamma_j(z)| \right)= \left(\prod_{j=1}^p |\gamma_j(z)|\right)^{\frac{\#G/G^o}{p}}=|\det z|^{\frac{\#G/G^o}{d}} \end{align}
for each $i=1,\ldots,p$ and $z \in Z(G^o)_D$. We obviously have $\hat\pi(gzg^{-1})=\hat\pi(z)$ for every $z \in Z(G^o)_D$ and $g \in G$ by the commutativity of real multiplication. Hence for every $z \in Z(G^o)_D$
\begin{align*}\hat\pi(z)&=\left(\prod_{[g]\in G/G^o} \hat\pi(gzg^{-1})\right)^{\frac{1}{\#G/G^o}}= \left(\prod_{[g] \in G/G^o}\prod_{i=1}^p  |\gamma_i(gzg^{-1})|^{r_i }\right)^{\frac{1}{\#G/G^o}}\\
& = \prod_{i=1}^p\left(\prod_{[g] \in G/G^o}  |\gamma_i(gzg^{-1})|\right)^{\frac{r_i}{ \#G/G^o}}=\prod_{i=1}^p |\det z|^{\frac{r_i}{d}} = |\det z|^{\frac{1}{d}}\end{align*}
where we have used \eqref{eq:more-determinant-like-stuff} and the equation $r_1+\cdots +r_p=1$. We have obtained $\hat\pi(z)=|\det z|^{1/d}$ for all $z \in Z(G^o)_D$ and we deduce that the claimed identity \eqref{eq:dets} is valid for every $g \in G$ as required.


We may now return to the main direction of the proof.  Our first step towards the desired identity \eqref{eq:pressure-goal} is to observe that
\[P(\Phi)\geq \max_{\mathfrak{l} \in \mathfrak{L}} P(\Phi_{\mathfrak{l}})\]
as a direct consequence of  \eqref{eq:l-max} together with the definition of the pressure. Furthermore, for each $\mathfrak{l} \in \mathfrak{L}$ we have $\Phi_{\mathfrak{l}}(\iii) \geq \Phi_{\mathfrak{l}}^{\det}(\iii)$ for all $\iii \in \Sigma_N^*$. This follows by comparing \eqref{eq:frakpotential1} and \eqref{eq:frakpotential2} and using the elementary inequality $|\det B| \leq \|B\|^{\dim V}$ for all $B \in \GL(V)$, and it entails that $P(\Phi_{\mathfrak{l}}) \geq P(\Phi_{\mathfrak{l}}^{\det})$ for every $\mathfrak{l} \in \mathfrak{L}$. We have thus far obtained
\begin{equation}\label{eq:so-far}P(\Phi)\geq P(\Phi_{\mathfrak{l}})\geq P(\Phi_{\mathfrak{l}}^{\det})\end{equation}
for every $\mathfrak{l} \in \mathfrak{L}$.

Using the identity \eqref{eq:dets}, we immediately deduce that $\Phi_{\mathfrak{l}}^{\det}(\iii)=\Phi^{\det}(\iii)$ for all $\iii \in \Sigma_N^*$ simply by applying the equation \eqref{eq:dets} to the definition of the two potentials. Combining this observation with \eqref{eq:so-far} it follows that
\begin{equation}\label{eq:so-far-2}P(\Phi)\geq P(\Phi_{\mathfrak{l}})\geq P(\Phi_{\mathfrak{l}}^{\det})=P(\Phi^{\det})\end{equation}
for every $\mathfrak{l} \in \mathfrak{L}$.

Let us now show that $P(\Phi)=P(\Phi^{\det})$. By hypothesis there exists a Bernoulli measure $\mu$ which satisfies $h(\mu)+\Lambda(\Phi,\mu)=P(\Phi)$. Since $\mu$ is Bernoulli, it is ergodic, so by the subadditive ergodic theorem we have for $\mu$-a.e. $x\in \Sigma_N$
\begin{align*}\Lambda(\Phi,\mu) &= \lim_{n \to \infty} \frac{1}{n}\log \Phi(x|_n) = \lim_{n \to \infty}  \frac{1}{n}\log \max_{\mathfrak{l} \in \mathfrak{L}} \Phi_{\mathfrak{l}}(x|_n)\\
&= \max_{\mathfrak{l} \in \mathfrak{L}} \lim_{n \to \infty}  \frac{1}{n}\log  \Phi_{\mathfrak{l}}(x|_n)=\max_{\mathfrak{l} \in \mathfrak{L}} \Lambda(\Phi_{\mathfrak{l}},\mu).\end{align*}
Thus $P(\Phi)=h(\mu)+\Lambda(\Phi_{\mathfrak{l}_0},\mu)$ for some particular $\mathfrak{l}_0 \in \mathfrak{L}$, and therefore 
\[P(\Phi)\geq \max_{\mathfrak{l} \in \mathfrak{L}}P(\Phi_{\mathfrak{l}}) \geq P(\Phi_{\mathfrak{l}_0}) \geq h(\mu)+\Lambda(\Phi_{\mathfrak{l}_0},\mu) = P(\Phi)\]
where we have used the subadditive variational principle in the third inequality. We conclude that $P(\Phi)=P(\Phi_{\mathfrak{l}_0})$ and that $\mu$ is an equilibrium state of $\Phi_{\mathfrak{l}_0}$. By Theorem \ref{th:irreducible-case} applied to the potential $\Phi_{\mathfrak{l}_0}$ we have $P(\Phi_{\mathfrak{l}_0})=P(\Phi_{\mathfrak{l}_0}^{\det})$. We have seen already that $\Phi_{\mathfrak{l}_0}^{\det}$ is identically equal to $\Phi^{\det}$, so
\[P(\Phi^{\det}) = P(\Phi_{\mathfrak{l}_0}^{\det}) = P(\Phi_{\mathfrak{l}_0}) =P(\Phi)\geq P(\Phi_{\mathfrak{l}})\geq P(\Phi_{\mathfrak{l}}^{\det})=P(\Phi^{\det})\]
for every $\mathfrak{l} \in \mathfrak{L}$, where we have invoked \eqref{eq:so-far-2}.

 We have now established the desired identity 
 \[P(\Phi_{\mathfrak{l}}) =P(\Phi_{\mathfrak{l}}^{\det})\]
 for every $\mathfrak{l} \in \mathfrak{L}$. Since every $\Phi_{\mathfrak{l}}$ satisfies the hypotheses of Theorem \ref{th:irreducible-case} it follows from the implication (ii)$\implies$(iii) of that theorem that for each $\mathfrak{l}=(\ell_1,\ldots,\ell_r) \in \mathfrak{L}$, for every $j=1,\ldots,r$ the group
\begin{align*}\left\{\left|\det \left(\pi_{j,\ell_j}(g)\right)\right|^{-\frac{1}{\dim V_{\ell_j}^j}} \pi_{j,\ell_j}(g) \colon g \in G\right\} &=\left\{\left|\det g\right|^{-\frac{k_j}{d}} g^{\wedge k_j}|_{V_{\ell_j}^j} \colon g \in G\right\}\\
&=\left\{\left( |\det g|^{-\frac{1}{d}} g\right)^{\wedge k_j}|_{V_{\ell_j}^j} \colon g \in G\right\}\end{align*}
is compact, where we have again used \eqref{eq:dets}. Since $\mathfrak{l}$ is arbitrary we deduce that the group
\[\left\{(|\det g|^{-\frac{1}{d}}g)^{\wedge k_j} \colon g \in G\right\}\]
is compact for every $j=1,\ldots,r$. In particular it is compact for $j=1$, so there exists $K>0$ such that for every $g \in G$ we have $\|(|\det g|^{-\frac{1}{d}}g)^{\wedge k_1}\|\leq K$.  

Let $g \in G$ and define $h:=|\det g|^{-1/d}g$. We observed at the beginning of the proof that $k_1<d$. 
Since $1=|\det h|=\sigma_1(h)\cdots \sigma_d(h)$ we have
\begin{align*}\|h\|=\sigma_1(h) &=\sigma_2(h)^{-1}\cdots \sigma_d(h)^{-1}=\sigma_1(h^{-1})\cdots \sigma_{d-1}(h^{-1})\\
&\leq \left(\sigma_1(h^{-1})\cdots \sigma_{k_1}(h^{-1})\right)^{\frac{d-1}{k_1}}=\|(h^{-1})^{\wedge k_1}\|^{\frac{d-1}{k_1}} \leq K^{\frac{d-1}{k_1}}\end{align*}
where we have used $k_1 \leq d-1$ in order to pass from the first line to the second. The same reasoning obviously applies to $h^{-1}$, and we conclude that the group
\[\left\{|\det g|^{-\frac{1}{d}} g \colon g \in G\right\}\leq \GL_d(\mathbb{R})\]
is contained in the compact set
\[\left\{h \in \GL_d(\mathbb{R}) \colon \max\{\|h\|,\|h^{-1}\|\} \leq K^{\frac{d-1}{k_1}}\right\}\]
and hence is compact. Since obviously that group contains all of the linear maps $|\det A_i|^{-1/d}A_i$ the theorem is proved.

%
%

\appendix

\section{Equilibrium states of linear Cartan potentials} Our main technical results, Theorem \ref{th:main-tech} and \ref{th:irreducible-case}, admit a counterpart that can be expressed in more intrinsic terms (with respect to the linear algebraic group given by the Zariski closure of the semigroup generated by the linear parts of the iterated function system). Beyond its relative elegance, as we shall see, this formulation will allow us to have an understanding on the structure of equilibrium states of matrix potentials from a representation theoretic perspective. To avoid further technicalities we will restrict our considerations here to Zariski connected real reductive groups (e.g.\ $\GL_d(\R)$).

We retain the notation used in \S \ref{se:rev} where we described some preliminary facts concerning real reductive groups. We will require some additional facts about the representation theory of those groups further to the exposition in \S \ref{se:rev}. As before we refer the reader to \cite{bq.book,borel.book,borel-tits,chevalley,knapp} for more detailed exposition of this theory. For a focused account containing all the material which we will require we suggest \cite[Sections 2 \& 3]{ggkw}.

Let $G$ be a Zariski connected real reductive group. Let us say that it is of non-compact type if the derived group $[G,G]$ is not compact. Let $K$ be a maximal compact subgroup with Lie algebra $\mathfrak{k}$ and $\mathfrak{g}=\mathfrak{k} \oplus \mathfrak{k}^\perp$ the orthogonal decomposition of $\mathfrak{g}$ with respect to the Killing form. Let $\mathfrak{a} < \mathfrak{k}^\perp$ be a Cartan subspace of $\mathfrak{g}$ and $\mathfrak{a}=\mathfrak{a}_S \oplus \mathfrak{a}_Z$ its decomposition into the semisimple and central parts (as defined in \S \ref{se:rev}). Let $\Sigma \subset \mathfrak{a}^\ast$ denote the non-zero (restricted) roots, $\Delta:=\{\overline{\alpha}_1, \ldots, \overline{\alpha}_{d_S}\} \subset \Sigma$ a choice of simple roots, where $d_S \in \N$ is the semisimple real rank of $G$. We denote by $\mathfrak{a}_S^+$ the salient cone given by $\mathfrak{a}_S \cap \mathfrak{a}^+$. 
We fix an inner product $(\cdot,\cdot)$ on $\mathfrak{a}$ extending the restriction of the Killing form to $\mathfrak{a}_S$ and satisfying $\mathfrak{a}_S^\perp=\mathfrak{a}_Z$. This induces an identification of $\mathfrak{a}$ with $\mathfrak{a}^\ast$ and we use the same notation to denote the corresponding inner product on $\mathfrak{a}^\ast$. We denote by $\mathcal{W}$ the set of restricted weights of $G$, which is equal to 
$$\left\{\omega \in \mathfrak{a}^\ast \; | \; 2\frac{(\omega,\alpha)}{(\alpha,\alpha)} \in \Z \; \; \forall \alpha \in \Sigma\right\}.$$ For a subspace $V \leq \mathfrak{a}$, denote by $V^0$ its annihilator subspace in $\mathfrak{a}^\ast$. For each $\alpha \in \Delta$, we denote by $\omega_{\alpha}$ or $\omega_i$ (when $\alpha=\alpha_i$) the corresponding fundamental weight, i.e.\ the weight in $\mathcal{W} \cap \mathfrak{a}_Z^0$ satisfying $2\frac{(\omega,\alpha)}{(\alpha,\alpha)}=\delta_{\alpha,\beta}$ for every $\beta \in \Delta$. The set $\{\omega_i\}$ is a basis of $\mathfrak{a}_Z^0$. By choosing linearly independent $\omega_{d_S+1},\ldots,\omega_{d} \in \mathfrak{a}_S^0$ (where $d \in \N$ is the real rank of $G$) complete $\{\omega_i \; | \; i=1, \ldots, d_S\}$ to a basis of $\mathfrak{a}^\ast$. Given $\phi \in \mathfrak{a}^\ast$ with $\phi=\sum_{i=1}^d c_i \omega_i$, write $\phi_S=\sum_{i=1}^{d_S} c_i \omega_i$ and $\phi_Z=\phi-\phi_S$. Let $\mathcal{C}$ be the cone in $\mathfrak{a}^\ast$ defined by
$$\mathcal{C}:=\{\phi \in \mathfrak{a}^\ast \; | \;  \phi_S\neq 0\text{ and }\phi_S \text{ has only non-negative coefficients in the basis} \; \{\omega_i\}\}.
$$
Finally, as before, we denote by $\kappa: G \to \mathfrak{a}^+$ the associated Cartan projection. Theorem \ref{th:main-tech} admits the following abstract articulation:


\begin{theorem}\label{thm.cartan.state}
Let $G$ be a Zariski connected real reductive group of non-compact type and $(g_1,\ldots,g_N)$ be a tuple of elements such that the semigroup generated by $\{g_1,\ldots,g_N\}$ is Zariski dense in $G$. Let $\Psi$ be a finite subset of $\mathcal{C}$. Then,  the potential 
$$
\Phi^\Psi(\iii)=\max_{\phi \in \Psi} e^{\phi(\kappa(g_\iii))}
$$
does not have a Bernoulli equilibrium state. 
\end{theorem}

The proof of Theorem \ref{thm.cartan.state} follows a closely analogous path to the first half of the proof of the implication  (i) $\implies$ (iii) in Theorem \ref{th:irreducible-case} (corresponding to \S \ref{subsub.family.of.subspaces}--\S \ref{sss:ben} in the main text). We therefore give the proof in outline only. The proof starts by showing that the potential $\Phi^\Psi$ is indeed submultiplicative, justifying the use of the terminology of subadditive thermodynamic formalism in the statement above.

\begin{proof}[Proof of Theorem \ref{thm.cartan.state}]
\underline{\textit{Step 0}} (Submultiplicativity of $\Phi^\Psi$): ${}$ It clearly suffices to show that given any $\phi \in \mathcal{C}$, the potential $\Phi^{\phi}$ defined by $\Phi^{\phi}(\iii)=e^{\phi(\kappa(g_\iii))}$ is submultiplicative. Since for every $\omega \in \mathfrak{a}_S^0$ and $g,h \in G$, we have $\omega(\kappa(gh))=\omega(\kappa(g))+\omega(\kappa(h))$,  $\Phi^{\phi_Z}$ is a multiplicative potential. Therefore, it suffices to show that $\Phi^{\phi_S}$ is submultiplicative. By \cite[Lemma 8.15]{bq.book}, for every $i=1,\ldots,d_S$, there exists a rational irreducible proximal representation $(\rho_i,V_i)$ such that the highest weight $\chi_i$ of $\rho_i$ is a (positive integer) multiple of the fundamental weight $\omega_i$. It follows by \cite[Lemma 8.17]{bq.book} that for each $i=1,\ldots,d_S$, we can choose an inner product norm $||.||_i$ on $V_i$ such  $\omega_i(\kappa(g))=\log ||\rho_i(g)||_i$, where we also denote by $||.||_i$ the associated operator norm. Since the coefficients of $\phi_S$ in the basis $\{\omega_i\}$ are non-negative, it follows by submultiplicativity of the operator norms $||.||_i$ that the potential defined by $\iii \mapsto e^{\phi_S(\kappa(g_\iii))}$ is submultiplicative as desired.\\[3pt]
\underline{\textit{Step 1}} (Specialising to a maximal linear form, cf.~\S \ref{subsub.transitivity.classes}): ${}$ Arguing by contradiction, we suppose that the submultiplicative potential $\Phi^{\Psi}$ has a Bernoulli equilibrium state $\mu$. It follows by the same argument in \S \ref{subsub.transitivity.classes} that there exists $\phi \in \Psi$ such that $\mu$ is an equilibrium state for the submultiplicative potential $\Phi^{\phi}$ defined by $\Phi^{\phi}(\iii)= e^{\phi(\kappa(g_\iii))}$.\\[3pt]
\underline{\textit{Step 2}} (Obtaining quasi-multiplicativity, cf.~\S \ref{sss:c3}): ${}$ Using  Quint's \cite[Proposition I.2]{quint.div}\footnote{Namely the first property of the ``Produit g\'{e}n\'{e}rique'' which does not use the discreteness assumption.}, which is based on the aforementioned representation theoretic ingredients \cite[Lemmas 8.15 \& 8.17]{bq.book} and the main result of Abels--Margulis--Soifer \cite{AMS}, we deduce that there exists a finite set $F$ in the semigroup $\Gamma$ generated by $\{g_1,\ldots,g_N\}$ and a constant $K>0$ such that for every $g,h \in \Gamma$, there exists $f \in F$ satisfying 
$$
||\kappa(gfh)-\kappa(g)-\kappa(h)|| \leq K.
$$
It immediately follows that the the potential $\Phi^{\phi}$ is quasi-multiplicative in the sense of \eqref{eq.def.quasimult} --- the aforementioned result of Quint is a predecessor of \cite[Theorem 6]{BoMo18} and indeed in the proof of Theorem \ref{th:main-tech}, this step is analogous to where we use \cite[Theorem 6]{BoMo18} to obtain quasi-multiplicativity.\\[3pt]
\underline{\textit{Step 3}} (Thermodynamic ingredients, cf.~\S \ref{sss:c3}): ${}$
It now follows from Proposition \ref{pr:qm-unique} that the Bernoulli measure $\mu$ is the unique equilibrium state of the potential $\Phi^{\phi}$ and satisfies the relevant Gibbs inequality.
\\[3pt]
\underline{\textit{Step 4}} (From the Gibbs inequality to the additivity of the Jordan projection, cf.~\S \ref{subsub.get.multiplicative}): ${}$ Denote by $\lambda:G \to \mathfrak{a}^+$ the Jordan projection in $G$ (see \S \ref{subsub.Cartan.Jordan}). We possess the following  ``Gelfand formula'' \cite[Remark 8.7]{bq.book} relating the Cartan $\kappa$ and Jordan $\lambda$ projections: for every $g \in G$, we have 
\begin{equation}\label{eq.cartan.gelfand}
\frac{1}{n} \kappa(g^n) \to \lambda(g).
\end{equation}
Using the same reasoning as in \S \ref{subsub.get.multiplicative}, replacing the usual Gelfand formula with \eqref{eq.cartan.gelfand}, we find that for every $g,h\in \Gamma$ we have
\begin{equation}\label{eq.lambda.additive.general}
\phi(\lambda(gh))=\phi(\lambda(g))+\phi(\lambda(h)).
\end{equation}
\underline{\textit{Step 5}} (Applying Benoist's non-arithmeticity, cf.~\S \ref{sss:ben}): ${}$
In view of \eqref{eq.lambda.additive.general}, the subspace $\ker \phi$ contains the set $\{\lambda(gh)-\lambda(g)-\lambda(h) \; | \; g,h \in \Gamma\}$. One therefore deduces from Benoist's Theorem \ref{thm.benoist.density} that $\ker \phi$ contains $\mathfrak{a}_S$. This contradicts the assumption that $\phi_S \neq 0$ (or equivalently that $\phi \notin \mathfrak{a}_S^0$) and finishes the proof.
\end{proof}

\begin{remark}[On the set of equilibrium states of $\Phi^{\Psi}$]
1. Since the potential $\Phi^{\Psi}$ is not quasi-multiplicative in general, it is not guaranteed that it possesses a unique equilibrium state. Indeed, under the assumptions of the previous theorem it was shown in \cite{MoSe19} that a potential of the form $\Phi^{\Psi}$ can have several distinct equilibrium states. On the other hand the number of ergodic equilibrium states of $\Phi^{\Psi}$ is always finite, as follows from the main result of \cite{BoMo18}.\\[2pt]
2. Steps 1 $\&$ Step 2 of the above proof imply that the number of ergodic equilibrium states of $\Phi^{\Psi}$ is bounded above by the cardinality of $\Psi$. In fact, these two steps can be seen as part of the proof of the main result of \cite{BoMo18}. Indeed, the latter is proved by additionally applying a reductivisation argument (cf.\ \cite[Proposition 6.2]{KaMo18}) and dealing with non-connectivity (as in \S \ref{sss:c3}).\\[2pt]
3. For a salient cone $C$ of non-empty interior in $\mathfrak{a}$, consider the partial order $\leq_C$ on $\mathfrak{a}^\ast$ defined as $\ell \leq_C \ell'$ if and only if $\ell(x) \leq \ell'(x)$ for every $x \in C$. Now if $G$ is moreover semisimple, denoting by $\Psi_{\max} \subseteq \Psi$ the set of maximal elements of $\Psi$ for the partial order $\leq_{\mathfrak{a}^+}$ on $\mathfrak{a}^\ast$, it is not hard to see that the number of ergodic equilibrium states of $\Phi^{\Psi}$ is in fact bounded above by the cardinality of $\Psi_{\max}$. A further refinement can be given by using the notion of Benoist limit cone $\mathcal{BC}(\Gamma) \subseteq \mathfrak{a}^+$ of the semigroup $\Gamma$ generated by the IFS (\cite{benoist.linear1}): the number of ergodic equilibrium states of $\Phi^{\Psi}$ is bounded above by the cardinality of the maximal set $\Psi_{\max}^{\Gamma} \subseteq \Psi_{\max} \subseteq \Psi$ for the refined partial order $\leq_{\mathcal{BC}(\Gamma)}$.
\end{remark}

\section*{Acknowledgements}
The research of I.D. Morris was partially supported by the Leverhulme Trust (Research Project Grant RPG-2016-194). During the realisation of this project, C.S. was supported by SNF grants 182089, 178958 and 193481. The authors are grateful to several anonymous referees for numerous helpful remarks and bibliographical suggestions, and to one referee in particular for suggesting the corollaries of this work which are described in the appendix. They also thank Emmanuel Breuillard for helpful conversations.

I.D. Morris wishes to thank Roger Tribe (who was his tutor at the University of Warwick from 1999 to 2001) for suggesting that he follow Jonathan Munn's lecture course on Lie Groups in the 2000-01 academic year. This was excellent advice which he still regrets not having followed at the time. 



\end{document}